\documentclass[12pt]{article}

\usepackage{amsfonts,graphicx,amsmath,amssymb,amsthm,color,nicefrac,enumerate, layout, bbm, vmargin, listings,  hyperref,mathrsfs}
\usepackage[latin1]{inputenc}
\newcommand{\R}{\mathbb{R}}
\newcommand{\W}{\mathcal{W}}
\newcommand{\N}{\mathbb{N}}

\newcommand{\E}{\mathbb{E}}
\newcommand{\und}{\underline}
\newcommand{\C}{\mathbb{C}}
\newcommand{\K}{\mathbb{K}}

\newcommand\numberthis{\addtocounter{equation}{1}\tag{\theequation}}

\renewcommand\Re{\operatorname{\mathrm{Re}}}

\newcommand{\F}[0]{\ensuremath{\mathcal{F}}}
\newcommand{\one}[0]{\ensuremath{\mathbbm{1}}}

\renewcommand{\P}{\mathbb{P}}

\newcommand{\lf}{\lfloor}
\newcommand{\rf}{\rfloor}
\newcommand{\Y}{\mathcal{X}}
\newcommand{\fl}[1]{\lfloor #1 \rfloor_{h}}
\newcommand{\Id}{\mathrm{Id}}

\newtheorem{theorem}{Theorem}[section]
\newtheorem{lemma}[theorem]{Lemma}
\newtheorem{remark}[theorem]{Remark}
\newtheorem{cor}[theorem]{Corollary}
\newtheorem{definition}[theorem]{Definition}

\newtheorem{prop}[theorem]{Proposition}

\begin{document}
\title{Strong convergence for explicit space-time \\discrete 
numerical approximation  methods\\
for 
stochastic Burgers  equations}

\author{ Arnulf Jentzen, Diyora Salimova, and Timo Welti \\
	\small  ETH Z\"urich (Switzerland)}

\maketitle

\begin{abstract}
In this paper we propose and analyze explicit space-time discrete  numerical approximations for additive space-time white noise driven stochastic partial differential equations (SPDEs) with non-globally monotone nonlinearities such as the stochastic Burgers equation with space-time white noise.
The main result of this paper proves that the proposed  explicit space-time discrete approximation method converges strongly to the solution process of the stochastic Burgers equation with space-time white noise.
To the best of our knowledge, the main result of this work is the first result in the literature which establishes strong convergence for a space-time discrete approximation method in the case of the stochastic Burgers equations with space-time white noise.
\end{abstract}

\tableofcontents

\section{Introduction}
Numerical approximations for infinite-dimensional SEEs with superlinearly growing nonlinearities have been intensively studied in the literature 
(cf., e.g., \cite{j08b,Doersek2012,KamraniBloemker2017,Kamrani2015,HairerMatetski2016,CoxWelti2016,YangZhang2017,CuiHongLiu2017,CuiHongLiuZhou2017} and the references mentioned therein).
In applications one is often interested in statistical quantities of the solution process of the considered SEE and, in view of this, one is especially interested in strong and weak numerical approximations for the considered SEE (cf., e.g., Heinrich~\cite{Heinrich1998,Heinrich2001}, Giles~\cite{Giles2008},  and Creutzig et al.~\cite{Creutzig2009}).
It has been established in the literature that the linear-implicit Euler scheme, the explicit Euler scheme,
and the exponential Euler scheme
converge, in general, neither numerically weakly nor strongly 
in the case of  such SEEs; cf., e.g., Hutzenthaler, Jentzen, \& Kloeden~\cite[Theorem~2.1]{hjk11} and Hutzenthaler, Jentzen, \& Kloeden~\cite[Theorem~2.1]{HutzenthalerJentzenKloeden2013}. Fully drift-implicit Euler schemes, in contrast, do converge strongly in the case of several SEEs with superlinearly growing nonlinearities;
cf., e.g.,
Hu~\cite[Theorem~2.4]{Hu1996}
and
Higham, Mao, \& Stuart~\cite[Theorem~3.3]{Higham2002}
for finite-dimensional SEEs
and cf., e.g.,
Gy{\"o}ngy \& Millet~\cite[Theorem~2.10]{gm05}, Brze{\'z}niak, Carelli, \& Prohl~\cite[Theorem~7.1]{Brzezniak2013}, Kov{\'a}cs, Larsson, \& Lindgren~\cite[Theorem~1.1]{kll2015},
and
Furihata et al.~\cite[Theorem~5.4]{FurihataKovacsLarssonLindgren2016}
for infinite-dimensional SEEs.
In order to implement these schemes, a nonlinear equation has to be solved  approximatively in each time step and this results in additional computational effort (see, e.g.,  Hutzenthaler, Jentzen, \& Kloeden~\cite[Figure~4]{HutzenthalerJentzenKloeden2012}). Moreover, it has not yet been shown in the literature that these approximate implementations of fully drift-implicit Euler methods converge strongly. Lately, a series of appropriately modified versions of the explicit Euler method has been introduced and proven to converge strongly for some SEEs with superlinearly growing nonlinearities; cf., e.g., Hutzenthaler, Jentzen, \& Kloeden~\cite{HutzenthalerJentzenKloeden2012}, Wang \& Gan~\cite{WangGan2013},
Hutzenthaler \& Jentzen~\cite{Hutzenthaler2015}, Tretyakov \& Zhang~\cite{TretyakovZhang2013},
and Sabanis~\cite{Sabanis2013,Sabanis2013E}
for finite-dimensional SEEs
and cf., e.g.,
Gy{\"o}ngy, Sabanis, \& {\v{S}}i{\v{s}}ka~\cite{GoengySabanisS2015},
Jentzen \& Pu{\v s}nik~\cite{jp2015},  Becker \& Jentzen~\cite{BeckerJentzen2016},  and Hutzenthaler et al.~\cite{Salimova2016}
for infinite-dimensional SEEs.
These modified versions are  easily realizable, explicit, and  truncate/tame superlinearly growing nonlinearities  in order to prevent strong divergence. 
However, except for Becker \& Jentzen~\cite{BeckerJentzen2016} and Hutzenthaler et al.~\cite{Salimova2016}, each of the above mentioned 
 strong convergence results applies
only to trace class noise driven SEEs and excludes SEEs driven by the more irregular space-time white noise.
In \cite{BeckerJentzen2016}  a coercivity/Lyapunov-type condition
has been employed  to obtain
strong convergence  for stochastic Allen-Cahn equations
with additive space-time white noise; 
cf.\  \cite[(85), Lemma~6.2, and Corollaries~6.16--6.17]{BeckerJentzen2016}.
However, the machinery in  \cite{BeckerJentzen2016}
assumes the coercivity/Lyapunov-type coefficient in the coercivity/Lyapunov-type condition to be a constant
(cf.\   \cite[(85)]{BeckerJentzen2016} with \eqref{eq:coer:burgers} below)
and,
therefore, applies only to temporal semi-discrete approximation methods for  stochastic Allen-Cahn equations 
but excludes a series of important additive space-time white noise driven 
SEEs 
such as stochastic Burgers equations with space-time white noise. The approach in Hutzenthaler et al.~\cite{Salimova2016}, in turn, does not require the coercivity/Lyapunov-type coefficient to be a constant and allows it to be a function of the noise process. Nevertheless, the article~\cite{Salimova2016} imposes some serious restrictions on the coercivity/Lyapunov-type coefficient,  which are satisfied in the case of stochastic Kuramoto-Sivashinsky equations (see, e.g.,  \cite[Lemma~5.2 and Theorem~4.6]{Salimova2016}) but not in the case of stochastic Burgers equations (see, e.g., Lemma~\ref{coer:burgers} below). More precisely, the composition of the coercivity/Lyapunov-type coefficient and  the driving Ornstein-Uhlenbeck process needs to admit suitable exponential integrability properties (cf., e.g., \cite[Theorem~4.6 and Corollary~5.8]{Salimova2016}) because the  coercivity/Lyapunov-type coefficient is employed in a Gronwall-type argument (cf., e.g., \cite[Corollary~3.2 and Corollary~2.6]{Salimova2016}), which, in turn, requires a suitable exponential term to be integrable.
To the best of our knowledge, there exists neither a strong nor a weak temporal numerical approximation result for stochastic Burgers equations with space-time white noise in the 
scientific literature. It is a key contribution of this paper to relax the restrictions  on the coercivity/Lyapunov-type coefficient in \cite{BeckerJentzen2016} and  \cite{Salimova2016} so that strong convergence for numerical approximations for stochastic Burgers equations with space-time white noise can be achieved.
In order to obtain a strong convergence result for stochastic Burgers equations, we prove  that suitable exponential integrability properties of the composition of the coercivity/Lyapunov-type coefficient and a transformed driving Ornstein-Uhlenbeck process also yield strong convergence (see  Theorem~\ref{thm:strong} and Proposition~\ref{prop:exists} below). Additional important ingredients in the proof of the convergence result are  transformations of semigroups for solutions of SPDEs (see  Proposition~\ref{prop:transform_SG} below) and  Fernique's theorem (see, e.g., Proposition~\ref{thm:fernique} below). 

To illustrate the main result of this paper (see Theorem~\ref{thm:strong} below) we specialize it to the case of stochastic Burgers equations. This is the subject of the following theorem.

\begin{theorem}
	\label{thm:intro}
Let  $ T \in (0,\infty)$, 
$\varrho \in (\nicefrac{1}{8}, \nicefrac{1}{4})$,
$ \chi \in  (0, \nicefrac{\varrho }{2 } - \nicefrac{1}{16}] $, $ H = L^2((0,1); \R)$, 
let $ A \colon D(A) \subseteq H \to H $ be the Laplace operator with Dirichlet boundary conditions on $H$,
let $ ( H_r, \left< \cdot , \cdot \right>_{ H_r }, \left\| \cdot \right\|_{ H_r } ) $, $ r \in \R $, be a family of interpolation spaces associated to $ -A $,
let
$\xi \in H_{\nicefrac{1}{2}} $,
let
$F \colon H_{\nicefrac{1}{8}} \to H_{-\nicefrac{1}{2}} $,
$(e_n)_{n \in \N = \{1,2,3, \ldots\}} \colon \N  \to H $,
$(P_n)_{n \in \N} \colon \N \to L(H) $,
and
$(h_n)_{n \in \N} \colon \N \to (0, T] $
be functions which satisfy
for all $v \in H_{\nicefrac{1}{8}}$, $n \in \N$ that
$F(v)= -\frac{1}{2}(v^2)'$,
$ e_n =  (\sqrt{2} \sin(n \pi x) )_{x \in (0,1)}$,
$ P_n(v) = \sum_{k = 1 }^n \langle e_k, v \rangle_H e_k $,
and
$ \limsup_{ m \to \infty} h_m =0$,
let $ ( \Omega, \F, \P ) $ be a probability space,
let $(W_t)_{t \in [0, T]}$ be an $\Id_H$-cylindrical $( \Omega, \F, \P )$-Wiener process,
let $ \Y^n, \mathcal{O}^n \colon [0, T] \times \Omega \to P_n(H)$, $ n \in \N$, be stochastic processes,  and assume that for all  $ n \in \N $, $t \in [0, T]$ it holds $\P$-a.s.~that    $\mathcal{O}_t^n  = \int_0^t P_n \, e^{(t-s)A}  \, dW_s$
and
\begin{equation}
\label{eq:intro}
 \Y_t^n = P_n \, e^{ t A } \xi + \int_0^t P_n \,  e^{  ( t - s ) A } \, \one_{ \{ \| \Y_{ \lf s \rf_{h_n} }^n \|_{ H_{\varrho} } + \| \mathcal{O}_{ \lf s \rf_{h_n} }^n +P_n \, e^{ \lf s \rf_{ h_n } A } \xi \|_{ H_{\varrho} } \leq | h_n|^{ - \chi } \}} \, F \big(  \Y_{ \lf s \rf_{ h_n } }^n \big) \, ds + \mathcal{O}_t^n. 
\end{equation} Then
\begin{enumerate}[(i)]
	\item \label{item:intro:exists} there exists an up-to-indistinguishability  unique stochastic process $ X\colon [0, T] \times \Omega \to H_{\varrho}$ with continuous sample paths which satisfies that for all $t \in [0,T]$ it holds $\P$-a.s.\ that 
	\begin{equation}
	X_t = e^{ t A }   \xi  + \int_0^t e^{  ( t - s ) A}  \, F (  X_s ) \, ds  + \int_0^t e^{(t-s)A} \, dW_s
	\end{equation}
and 
	\item \label{item:intro:conv} it holds for all $p \in (0, \infty)$ that
	\begin{align}
	\limsup_{n \to \infty} \sup_{t \in [0,T]} \E \big[ \| X_t -\Y_t^n \|_H^p \big] = 0.
	\end{align} 
\end{enumerate}

\end{theorem}

In the framework of Theorem~\ref{thm:intro} we note that the stochastic process $ X $ in \eqref{item:intro:exists} in Theorem~\ref{thm:intro} is a mild solution process of the stochastic Burgers equation
\begin{align}
\tfrac{\partial}{\partial t} X_t(x) = \tfrac{\partial^2}{\partial x^2} X_t(x) - X_t(x) \cdot \tfrac{\partial}{\partial x} X_t(x) + \tfrac{\partial}{\partial t} W_t(x)
\end{align}
with  $X_0(x) = \xi(x)$ and $X_t(0)= X_t(1)=0$
for  $t \in [0,T]$, $x \in (0,1)$.
Observe that~\eqref{item:intro:exists} in Theorem~\ref{thm:intro} follows, e.g., from  Bl\"omker \& Jentzen~\cite[Theorem~3.1 and Subsection~4.3]{BloemkerJentzen2013}, while~\eqref{item:intro:conv} in Theorem~\ref{thm:intro} is an immediate consequence of Corollary~\ref{cor:burgers:short} below. 
Moreover, we note that the scheme proposed in~\eqref{eq:intro} is a modified version of the scheme proposed in \cite{Salimova2016} for stochastic Kuramoto-Sivashinsky equations (cf.  \cite[(90)]{Salimova2016} with \eqref{eq:intro} above).

The remainder of this paper is organized as follows. In Section~\ref{sec:a_priori} the required a priori moment bounds for the proposed scheme are established. In Section~\ref{sec:main} the error analysis is first performed in the pathwise sense to obtain pathwise convergence in Proposition~\ref{prop:main_det} and pathwise a priori bounds in Proposition~\ref{prop:main_det:2}. Combining these pathwise results  allows
us to accomplish strong convergence in Theorem~\ref{thm:strong} for a large class of SEEs on general separable $\R$-Hilbert spaces. In Section~\ref{sec:abstract} we verify the assumptions of Theorem~\ref{thm:strong} in the case of more concrete SPDEs on the Hilbert space $L^2((0,1);\R)$ and we prove strong convergence in Proposition~\ref{abs:prop:last}. To derive Proposition~\ref{abs:prop:last} we also employ the properties of stochastic convolution processes in  Proposition~\ref{prop:exists}, which  are obtained using  Fernique's theorem (see Section~\ref{sec:fernique}). In Section~\ref{sec:examples} we apply Proposition~\ref{abs:prop:last} in the case of stochastic Burgers and stochastic Allen-Cahn equations and establish strong convergence in Corollary~\ref{cor:burgers:short} and in Corollary~\ref{cor:cahn:short}, respectively.
\subsection{Notation}
\label{sec:notation}

Throughout this article the following notation is frequently used.
For every set $A$ we denote by $\mathcal{P}(A)$ the power set of $A$,
we denote by $\#_{A} \in \N_0 \cup \{\infty\}$ the number of elements of $A$,
and we denote by $\mathcal{P}_0(A)$ the set given by $\mathcal{P}_0(A) = \{ B \in \mathcal{P}(A) \colon \#_B < \infty \}$.
For all measurable spaces $( A, \mathcal{A})$ and $( B, \mathcal{B})$ we denote by $ \mathcal{M}(\mathcal{A}, \mathcal{B})$ the set of all $\mathcal{A} \slash \mathcal{B}$-measurable functions.
For every set $A \in \mathcal{B}(\R)$ we denote by $\lambda_A \colon \mathcal{B}(A) \to [0, \infty]$ the Lebesgue-Borel measure on $ ( A, \mathcal{B}(A) ) $.
For every measure space $(\Omega, \F, \mu)$, every measurable space $(S, \mathcal{S})$, every set $R$, and every function $f \colon \Omega \to R$ we denote by $\left[ f \right]_{\mu, \mathcal{S}} $ the set given by $\left[f \right]_{\mu, \mathcal{S}} = \left\{ g \in \mathcal{M}(\F, \mathcal{S}) \colon ( \exists \, A \in \F \colon \mu(A)=0 \,\, \text{and} \,\, \{ \omega \in \Omega \colon f(\omega) \neq g(\omega)\} \subseteq A ) \right\}$.
For every $ h \in (0, \infty)$ we denote by $ \lf \cdot \rf_h \colon \R \to \R$ the function which satisfies for all $t \in \R$ that $\lf t \rf_h = \max( (-\infty, t] \cap \{0, h, -h, 2h, -2h, \ldots\} )$.
We denote by $\und{(\cdot)} \colon \bigl\{ [v]_{\lambda_{(0,1)}, \mathcal{B}(\R)} \in \mathcal{P} \bigl( \mathcal{M}(\mathcal{B}((0,1)), \mathcal{B}(\R)) \bigr)\colon v \in \mathcal{C}( (0,1), \R ) \bigr\} \to \mathcal{C}((0,1), \R)$  the function which satisfies for all $v \in \mathcal{C}( (0,1), \R ) $  that $\und{[v]_{\lambda_{(0,1)}, \mathcal{B}(\R)}}=v$.
For all real numbers $\theta \in (0,1)$ and $p \in [1, \infty)$ we denote by $ \left\| \cdot \right\|_{\W^{\theta, p}((0,1), \R)} \colon \mathcal{M}(\mathcal{B}((0,1)) ,\mathcal{B}(\R)) \to [0,\infty]$  the function which satisfies for all $v \in \mathcal{M}(\mathcal{B}((0,1)) ,\mathcal{B}(\R))$ that
\begin{align}
	\|v\|_{\W^{\theta, p}((0,1), \R)} = \left[ \int_0^1 |v(x)|^p \, dx + \int_0^1 \int_0^1 \frac{|v(x)-v(y)|^p}{|x-y|^{1+ \theta p}} \, dx \, dy\right]^{\nicefrac{1}{p}}.
\end{align}

\section{A priori bounds}
\label{sec:a_priori}

In Proposition~\ref{prop:priori_bound} and Corollary~\ref{cor:a_priori} below we establish general
a priori bounds which will be used in Section~\ref{sec:main} to obtain a~priori bounds for the proposed approximation method.
Before we state Proposition~\ref{prop:priori_bound} and Corollary~\ref{cor:a_priori},
we present in five auxiliary lemmas, Lemmas~\ref{lem:fund:calc}--\ref{lemma:lyapunov3},
elementary  results which are used in the proofs of Proposition~\ref{prop:priori_bound} and Corollary~\ref{cor:a_priori}.
In particular,
we prove a~priori bounds for general Lyapunov-type functions
in Lemmas~\ref{lemma:lyapunov2}--\ref{lemma:lyapunov3},
whereas in Proposition~\ref{prop:priori_bound} and Corollary~\ref{cor:a_priori}
we establish a priori bounds for a particular choice for the Lyapunov-type function.

\subsection{On strong and mild solutions of semilinear evolution equations}
\label{subsec:strong_mild}

The next elementary and well-known result, Lemma~\ref{lem:fund:calc}, presents a version of the fundamental theorem of calculus and the chain rule.
It is employed in the proof of Lemma~\ref{lemma:fund:gen} below.

\begin{lemma}
\label{lem:fund:calc}
Let $(V, \left\| \cdot \right\|_V)$ be a nontrivial $\R$-Banach space, let $(W, \left\|\cdot\right\|_W)$ be an $\R$-Banach space, let $U \subseteq V$ be an open set,  let $f \in \mathcal{C}^1(U,W)$, $a \in \R$, $b \in (a,\infty)$,
let $ x \colon [ a, b ] \to U $ be a function,
and let $y \colon [a,b] \to V$ be a strongly $\mathcal{B}([a,b]) \slash (V, \left\| \cdot \right\|_V)$-measurable function which satisfies for all $t \in [a,b]$ that
 $\int_a^b \|y_s\|_V \, ds < \infty$ and 
\begin{align}
x_t = x_a + \int_a^t y_s \, ds.
\end{align}
Then 
\begin{enumerate}[(i)]
	\item it holds that the function $[a,b] \ni t \mapsto f'(x_t) \, y_t \in W$ is strongly $\mathcal{B}([a,b]) \slash (W, \left\| \cdot \right\|_W)$-measurable,
	\item it holds that $\int_a^b \| f'(x_s) \, y_s \|_W \, ds < \infty$, and 
	\item it holds for all $t_0 \in [a,b]$, $t \in [t_0,b]$ that
	\begin{align}
	f(x_t) = f(x_{t_0}) + \int_{t_0}^t f'(x_s) \, y_s \, ds.
	\end{align}
\end{enumerate}
\end{lemma}

\begin{lemma}
\label{lemma:fund:gen}
Let $(V, \left\|\cdot \right\|_V)$ be a separable $\R$-Banach space,
let $A \in L(V)$, $T \in (0, \infty)$,
let $Y \colon [0,T] \to V $ be a function,
and let $Z \colon [0,T] \to V $ be a $\mathcal{B}([0,T]) $/$ \mathcal{B}(V) $-measurable function which satisfies for all $t \in [0, T]$ that $\sup_{s \in [0,T]} \|Z_s\|_V <\infty$ and $Y_t = \int_0^t e^{(t-s)A} \, Z_s \, ds$. Then 
\begin{enumerate}[(i)]
\item it holds  that $Y$ is continuous and
\item it holds for all $t \in [0,T]$  that  $Y_t = \int_0^t A Y_s + Z_s \, ds$.
\end{enumerate}
\end{lemma}
\begin{proof}[Proof of Lemma~\ref{lemma:fund:gen}]
First, note that for all $t \in [0,T]$ it holds that $Y_t =e^{tA} \int_0^t e^{-sA} \, Z_s \, ds$.  The assumption that $ A \in L(V)$ and the assumption that $\sup_{s \in [0,T]} \|Z_s\|_V <\infty$ hence prove that $Y$ is continuous. Moreover,  Lemma~\ref{lem:fund:calc} (with $V= \R \times V$, $W = V$, $U= \R \times V$,  $f= ( \R \times V \ni (t,v) \mapsto e^{tA} \,v \in V)$, $a=0$, $b=T$, $x = ([0,T] \ni t \mapsto (t, \int_0^t e^{-sA} \, Z_s \, ds) \in \R \times V)$, $ y = ([0,T] \ni t \mapsto (1, e^{-tA} \, Z_t) \in \R \times V)$, $t_0 =0$ in the notation of Lemma~\ref{lem:fund:calc}) ensures for all $t \in [0,T]$ that
\begin{align}
\begin{split}
Y_t = e^{tA} \int_0^t e^{-sA} \, Z_s \, ds = \int_0^t A e^{sA} \int_0^s e^{-uA} \, Z_u \, du \, ds+ \int_0^t e^{sA} \, e^{-sA} \, Z_s \, ds = \int_0^t AY_s +Z_s \, ds.
\end{split}
\end{align}
The proof of Lemma~\ref{lemma:fund:gen} is thus completed. 
\end{proof}

\begin{lemma}
\label{lemma:lyapunov1}
Consider the notation in Subsection~\ref{sec:notation},
let $ ( V,  \left\| \cdot \right\|_V ) $ be a separable $\mathbb{R}$-Banach space,
let $T \in (0, \infty)$, $\eta \in [0, \infty)$, $ h \in (0, T] $, $ A \in L(V)$,
and let $ Z \colon [0, T] \to \R $,
$Y, O, \mathbb{O} \colon [0, T] \to V $,
and $F \colon V \to V $
be functions which satisfy
for all $t \in [0,T]$  that $\eta O \in \mathcal{C}([0, T], V)$,    $ \mathbb{O}_t = O_t - \int_0^t e^{(t-s)(A-\eta)} \, \eta  O_s \, ds$,   and
\begin{equation}
Y_t = e^{ t A}\, ( Y_0 - O_0 )  + \int_0^t e^{ ( t - s ) A } \, Z_{ \fl{s} }  F \big(  Y_{ \fl{s} } \big) \, ds + O_t.
\end{equation}
Then 
\begin{enumerate}[(i)]
\item  it holds that the functions $ [0,T] \ni t \mapsto Y_{t} -\mathbb{O}_t \in V $ and  $ [0,T] \ni t \mapsto \eta Y_{t} \in V $ are continuous and
\item it holds for all  $ t \in [0,T] $ that $ Y_t- \mathbb{O}_t= Y_0 - \mathbb{O}_0 +  \int_0^t (A-\eta) (Y_s - \mathbb{O}_s)+  Z_{ \fl{s} } F\big(Y_{ \fl{s} }\big) +  \eta Y_s  \, ds $. 
\end{enumerate}
\end{lemma}
\begin{proof}[Proof of Lemma~\ref{lemma:lyapunov1}]
Throughout this proof let $  \mathbb{A} \in L(V) $
be the linear operator given by $\mathbb{A}= A-\eta$
and let $\bar{Y}, E \colon [0,T] \to V $
be the functions which satisfy for all $ t \in [0,T] $  that
$\bar{Y}_t = Y_t - \mathbb{O}_t$ and $E_t =  \int_0^t e^{(t-s)\mathbb{A}} \, \eta  O_s \, ds$.
Note that the assumption that $\eta O \in \mathcal{C}([0, T], V)$ and  Lemma~\ref{lemma:fund:gen} (with $V=V$, $A=\mathbb{A}$, $T=T$, $Y=E$, $Z=\eta O$ in the notation of Lemma~\ref{lemma:fund:gen}) prove that $E$ is continuous and that for all $t \in [0, T]$ it holds that 
\begin{align}
\label{eq:diff:E}
E_t = \int_0^t \mathbb{A} E_s + \eta O_s \, ds.
\end{align} 
Moreover, observe that for all $t \in [0,T]$ it holds that 
\begin{align}
Y_t - O_t = e^{ t A } \, (Y_0 -O_0) 
+ \int_0^t e^{  ( t - s ) A} \, Z_{ \fl{s} }  F\big(Y_{ \fl{s} }\big) \, ds .
\end{align}
This and Lemma~\ref{lemma:fund:gen} (with $V=V$, $A=A$, $T=T$, $Y=([0,T] \ni t \mapsto (Y_t-O_t-e^{tA}\, (Y_0-O_0)) \in V)$,
$Z=([0,T] \ni t \mapsto Z_{ \fl{t} }  F(Y_{ \lf t \rf_{h } }) \in V) $
in the notation of Lemma~\ref{lemma:fund:gen})  prove that for all $ t \in [0,T] $ it holds that
$(Y- O) \in \mathcal{C}([0, T], V)$ and
\begin{align}
\label{eq:diff:Y-O}
\begin{split}
Y_t - O_t = Y_0- O_0 + \int_0^t A (Y_s -O_s)+  Z_{ \fl{s} } F\big(Y_{ \fl{s} }\big) \, ds.
\end{split}
\end{align}
This and the fact that  $E, \eta O \in \mathcal{C}([0, T], V)$ ensure that the functions $ [0,T] \ni t \mapsto \bar{Y}_t \in V $ and  $ [0,T] \ni t \mapsto \eta Y_{t} \in V $ are continuous. In the next step we combine \eqref{eq:diff:E} and \eqref{eq:diff:Y-O} to obtain that for all $t \in [0,T]$ it holds that
\begin{align}
\begin{split}
\bar{Y}_t & = \bar{Y}_0 +  \int_0^t A (Y_s -O_s)+  Z_{ \fl{s} } F\big(Y_{ \fl{s} }\big)  \, ds + E_t \\
& = \bar{Y}_0 +  \int_0^t A (Y_s -O_s)+  Z_{ \fl{s} } F\big(Y_{ \fl{s} }\big) + \mathbb{A} E_s + \eta O_s \, ds   \\
& = \bar{Y}_0 +  \int_0^t \mathbb{A} (Y_s -O_s)+  Z_{ \fl{s} } F\big(Y_{ \fl{s} }\big) + \mathbb{A} E_s + \eta O_s  + \eta(Y_s - O_s) \, ds \\
 & = \bar{Y}_0 +  \int_0^t \mathbb{A} (Y_s -O_s+ E_s)+  Z_{ \fl{s} } F\big(Y_{ \fl{s} }\big) +  \eta Y_s  \, ds \\
  & = \bar{Y}_0 +  \int_0^t \mathbb{A} \bar{Y}_s +  Z_{ \fl{s} } F\big(Y_{ \fl{s} }\big) +  \eta Y_s  \, ds.
\end{split}
\end{align}
The proof of Lemma~\ref{lemma:lyapunov1} is thus completed.
\end{proof}


\subsection{General a priori bounds}
\label{subsec:gen_a_priori}

\begin{lemma}
\label{lemma:lyapunov2}
Consider the notation in Subsection~\ref{sec:notation},
let $ ( V,  \left\| \cdot \right\|_V ) $ be a nontrivial separable $\mathbb{R}$-Banach space,
let $T \in (0, \infty)$, $\eta \in [0, \infty)$, $ h \in (0, T] $, $ A \in L(V)$, $\mathbb{V} \in \mathcal{C}^1(V,  \R)$, $F \in \mathcal{C}(V, V)$,
and let
$ Z \colon [0, T] \to \R $,  $Y, O, \mathbb{O} \colon [0, T] \to V $, and $\phi,  f \colon V \to \R $
be functions which 
satisfy for all  $t \in [0,T]$  that $f \circ \mathbb{O} \in \mathcal{C}([0,T], \R)$, $\eta O \in \mathcal{C}([0, T], V)$,    $ \mathbb{O}_t = O_t - \int_0^t e^{(t-s)(A-\eta)} \, \eta  O_s \, ds$,   and
\begin{equation}
Y_t = e^{ t A} \, ( Y_0 - O_0 )  + \int_0^t e^{ ( t - s ) A } \, Z_{ \fl{s} }  F \big(  Y_{ \fl{s} } \big) \, ds + O_t.
\end{equation}
Then 
\begin{enumerate}[(i)]
\item it holds that the functions $ [0,T] \ni t \mapsto Y_{t} -\mathbb{O}_t \in V $ and  $ [0,T] \ni t \mapsto \eta Y_{t} \in V $ are continuous,
\item it holds that $\sup_{s\in [0,T]} \big\|F\big(  Y_s - \mathbb{O}_s +
\mathbb{O}_{ \fl{s} } \big)\big\|_V < \infty$, and 
\item it holds for all $ t \in [0, T]$  that
\begin{align}
&e^{- \int_0^t  \phi( \mathbb{O}_{\fl{s} } ) +  f(\mathbb{O}_s) \, ds} \, \mathbb{V}( Y_t - \mathbb{O}_t )  = \mathbb{V}( Y_0 - \mathbb{O}_0 ) \nonumber \\
&  +  \int_0^t e^{- \int_0^s \phi( \mathbb{O}_{\fl{u} } )  + f(\mathbb{O}_u) \, du} \, \mathbb{V}' ( Y_s - \mathbb{O}_s) \! \left[ (A-\eta) ( Y_s - \mathbb{O}_s) + Z_{ \fl{s} } F\big(  Y_s - \mathbb{O}_s +
\mathbb{O}_{ \fl{s} } \big) + \eta Y_s \right]  ds \nonumber \\
& + \int_0^t  e^{- \int_0^s \phi( \mathbb{O}_{\fl{u} } )  +  f(\mathbb{O}_u) \, du} \, Z_{ \fl{s} } \mathbb{V}' ( Y_s - \mathbb{O}_s)\! \left[
F\big( Y_{ \fl{s} }\big) - F\big(  Y_s - \mathbb{O}_s +
\mathbb{O}_{ \fl{s} } \big)\right]  ds  \\
&- \int_0^t \big[ \phi\big( \mathbb{O}_{\fl{s} } \big) +  f(\mathbb{O}_s)\big]  e^{- \int_0^s \phi( \mathbb{O}_{\fl{u} } )  + f(\mathbb{O}_u)\, du} \, \mathbb{V}( Y_s - \mathbb{O}_s)\, ds \nonumber.
\end{align}
\end{enumerate}
\end{lemma}
\begin{proof}[Proof of Lemma~\ref{lemma:lyapunov2}]
Throughout this proof let $\bar{Y} \colon [0,T] \to V $ be the function which satisfies for all $ t \in [0,T] $  that $\bar{Y}_t = Y_t - \mathbb{O}_t$.
Note that Lemma~\ref{lemma:lyapunov1} (with $V=V$, $T = T$, $\eta = \eta$, $h=h$, $ A = A$, $ Z =Z$,  $Y=Y$, $O=O$, $\mathbb{O}=\mathbb{O}$, $F =F$ in the notation of Lemma~\ref{lemma:lyapunov1}) and the assumption that $F \in \mathcal{C}(V,V)$ establish 
that for all $t \in [0,T]$ it holds that $\bar{Y} \in \mathcal{C}([0,T],V)$, $\eta Y \in \mathcal{C}([0,T],V)$, $\sup_{s\in [0,T]} \big\|F\big(  Y_s - \mathbb{O}_s +
\mathbb{O}_{ \fl{s} } \big)\big\|_V < \infty$, and 
\begin{align}
\begin{split}
\bar{Y}_t  = \bar{Y}_0 +  \int_0^t (A-\eta) \bar{Y}_s +  Z_{ \fl{s} } F\big(Y_{ \fl{s} }\big) +  \eta Y_s  \, ds.
\end{split}
\end{align}
This and Lemma~\ref{lem:fund:calc} (with $V= \R \times V$, $W=\R$, $U= \R \times V$, $f= (\R \times V \ni (t,v) \mapsto e^{-t} \,\mathbb{V}(v) \in \R)$, $a=0$, $b=T$, $x = ([0,T] \ni t \mapsto (\int_0^t  \phi( \mathbb{O}_{\fl{s} } ) +  f(\mathbb{O}_s) \, ds, \bar{Y}_t) \in \R \times V)$, $y = ([0,T] \ni t \mapsto ( \phi( \mathbb{O}_{\fl{t} } ) +  f(\mathbb{O}_t), (A-\eta) \bar{Y}_t +  Z_{ \fl{t} } F(Y_{ \fl{t} }) +  \eta Y_t ) \in \R \times V)$, $t_0=0$ in the notation of Lemma~\ref{lem:fund:calc}) ensure for all $ t \in [0,T] $ that
\begin{align}
\begin{split}
& e^{- \int_0^t  \phi( \mathbb{O}_{\fl{s} } ) +  f(\mathbb{O}_s) \, ds} \, \mathbb{V}( \bar{Y}_t ) \\
&=\mathbb{V}(\bar{Y}_0 ) +  \int_0^t  e^{- \int_0^s \phi( \mathbb{O}_{\fl{u} } )  +  f(\mathbb{O}_u) \, du} \, \mathbb{V}' (\bar{Y}_s) \! \left[ (A-\eta)\bar{Y}_s + Z_{ \fl{s} }
F\big( Y_{ \fl{s} }\big) + \eta Y_s \right]  ds \\
& \quad - \int_0^t \big[ \phi\big( \mathbb{O}_{\fl{s} } \big) +  f(\mathbb{O}_s) \big]  e^{- \int_0^s \phi( \mathbb{O}_{\fl{u} } )  +  f(\mathbb{O}_u) \, du} \, \mathbb{V}(\bar{Y}_s)\, ds \\
& = \mathbb{V}(\bar{Y}_0 ) +  \int_0^t e^{- \int_0^s \phi( \mathbb{O}_{\fl{u} } )  +  f(\mathbb{O}_u) \, du} \, \mathbb{V}' (\bar{Y}_s)\! \left[ (A-\eta) \bar{Y}_s + Z_{ \fl{s} } F\big( \bar{Y}_s +
\mathbb{O}_{ \fl{s} } \big) + \eta Y_s \right]  ds \\
& \quad + \int_0^t  e^{- \int_0^s \phi( \mathbb{O}_{\fl{u} } )  +  f(\mathbb{O}_u) \, du} \, Z_{ \fl{s} } \mathbb{V}' (\bar{Y}_s)\! \left[
F\big( Y_{ \fl{s} }\big) - F\big( \bar{Y}_s +
\mathbb{O}_{ \fl{s} } \big)\right]  ds \\
&\quad - \int_0^t \big[ \phi\big( \mathbb{O}_{\fl{s} } \big) + f(\mathbb{O}_s)\big]   e^{- \int_0^s \phi( \mathbb{O}_{\fl{u} } )  +  f(\mathbb{O}_u)\, du} \, \mathbb{V}(\bar{Y}_s)\, ds .
\end{split}
\end{align}
The proof of Lemma~\ref{lemma:lyapunov2} is thus completed.
\end{proof}

\begin{lemma}
\label{lemma:lyapunov3}
Consider the notation in Subsection~\ref{sec:notation},
let $ ( V, \left\| \cdot \right\|_V ) $ be a nontrivial separable $\mathbb{R}$-Banach space,
let $T \in (0, \infty)$, $\eta \in [0, \infty)$, $ h \in (0, T] $, $ A \in L(V)$,
$\mathbb{V} \in \mathcal{C}^1(V, [0, \infty))$,
$F \in \mathcal{C}(V, V)$,
$\varphi \in \mathcal{C}(V, [0,\infty))$, 
and let
$ Z \colon [0, T] \to [0,1] $,  $Y, O, \mathbb{O} \colon [0, T] \to V $, and  $\phi, \Phi, f,  g \colon V \to  [0, \infty)$
be functions which satisfy for all $v, w \in V $, $t \in [0,T]$  that $f \circ \mathbb{O}$, $ g \circ \mathbb{O} \in \mathcal{C}([0,T], [0, \infty))$, $\eta O \in \mathcal{C}([0, T], V)$, $ \mathbb{V}'(v)F(v+w) \leq \phi(w) \mathbb{V}(v) + \varphi(v)+ \Phi(w)$, $\eta \mathbb{V}'(v) (v+w) \leq f(w) \mathbb{V}(v) + g(w)$,  $ \mathbb{O}_t = O_t - \int_0^t e^{(t-s)(A-\eta)} \, \eta  O_s \, ds$,   and
\begin{equation}
Y_t = e^{ t A}\, ( Y_0 - O_0 )  + \int_0^t e^{ ( t - s ) A } \, Z_{ \fl{s} }  F \big(  Y_{ \fl{s} } \big) \, ds + O_t.
\end{equation}
Then
\begin{enumerate}[(i)]
\item it holds that the function $ [0,T] \ni t \mapsto Y_{t} -\mathbb{O}_t \in V $ is continuous,
\item it holds that $\sup_{s\in [0,T]} \big\|F\big(  Y_s - \mathbb{O}_s +
\mathbb{O}_{ \fl{s} } \big)\big\|_V < \infty$, and 
\item  it holds for all $ t \in [0, T]$  that
\begin{align}
&\mathbb{V}( Y_t - \mathbb{O}_t )  \leq e^{ \int_0^t  \phi( \mathbb{O}_{\fl{s} } ) +  f(\mathbb{O}_s) \, ds} \, \mathbb{V}(Y_0 - \mathbb{O}_0) \nonumber \\
& +   \int_0^t e^{ \int_s^t \phi( \mathbb{O}_{\fl{u} } )  +  f(\mathbb{O}_u) \, du} \left[ \mathbb{V}' (Y_s - \mathbb{O}_s)  (A-\eta)(Y_s - \mathbb{O}_s)+ \varphi( Y_s - \mathbb{O}_s)+ \Phi\big(\mathbb{O}_{ \fl{s} } \big)  + g(\mathbb{O}_s) \right]  ds \nonumber \\
& + \int_0^t  e^{ \int_s^t \phi( \mathbb{O}_{\fl{u} } )  + f(\mathbb{O}_u) \, du} \, Z_{ \fl{s} } \mathbb{V}' (Y_s - \mathbb{O}_s)\! \left[
F\big( Y_{ \fl{s} }\big) - F\big( Y_s - \mathbb{O}_s +
\mathbb{O}_{ \fl{s} } \big)\right]  ds.
\end{align}
\end{enumerate}
\end{lemma}
\begin{proof}[Proof of Lemma~\ref{lemma:lyapunov3}]
Throughout this proof let $\bar{Y} \colon [0,T] \to V $ be the function which satisfies for all $ t \in [0,T] $  that $\bar{Y}_t = Y_t - \mathbb{O}_t$
and let $  \mathbb{A} \in L(V)$  be the linear operator given by $\mathbb{A}= A-\eta$.
Note that Lemma~\ref{lemma:lyapunov2} (with $ V=V$, $T=T$, $\eta=\eta$, $ h=h$, $ A =A$, $\mathbb{V} = (V \ni v \mapsto \mathbb{V}(v) \in \R)$, $F=F$, $ Z=([0,T] \ni t \mapsto Z_t \in \R)$,  $Y=Y$, $O=O$, $\mathbb{O}= \mathbb{O}$, $\phi=(V \ni v \mapsto \phi(v) \in \R)$, $f=(V \ni v \mapsto f(v) \in \R)$ in the notation of Lemma~\ref{lemma:lyapunov2}) ensures that for all $ t \in [0,T] $ it holds that $\bar{Y} \in \mathcal{C}([0,T], V)$, $\sup_{s\in [0,T]} \big\|F\big(  Y_s - \mathbb{O}_s +
\mathbb{O}_{ \fl{s} } \big)\big\|_V < \infty$, and
\begin{align}
\label{eq:lyapunov:fund}
\begin{split}
& e^{- \int_0^t  \phi( \mathbb{O}_{\fl{s} } ) + f(\mathbb{O}_s) \, ds} \, \mathbb{V}( \bar{Y}_t ) \\
& = \mathbb{V}(\bar{Y}_0 ) +  \int_0^t e^{- \int_0^s \phi( \mathbb{O}_{\fl{u} } )  +  f(\mathbb{O}_u) \, du} \, \mathbb{V}' (\bar{Y}_s) \!\left[ \mathbb{A} \bar{Y}_s + Z_{ \fl{s} } F\big( \bar{Y}_s +
\mathbb{O}_{ \fl{s} } \big) + \eta Y_s \right]  ds \\
& \quad + \int_0^t  e^{- \int_0^s \phi( \mathbb{O}_{\fl{u} } )  +  f(\mathbb{O}_u) \, du} \, Z_{ \fl{s} } \mathbb{V}' (\bar{Y}_s) \! \left[
F\big( Y_{ \fl{s} }\big) - F\big( \bar{Y}_s +
\mathbb{O}_{ \fl{s} } \big)\right]  ds \\
& \quad - \int_0^t \big[ \phi\big( \mathbb{O}_{\fl{s} } \big) +  f(\mathbb{O}_s)\big]  e^{- \int_0^s \phi( \mathbb{O}_{\fl{u} } )  +  f(\mathbb{O}_u)\, du} \, \mathbb{V}(\bar{Y}_s)\, ds .
\end{split}
\end{align}
Furthermore, the assumption that $\forall \, v, w \in V \colon \mathbb{V}'(v)F(v+w) \leq \phi(w) \mathbb{V}(v) + \varphi(v)+ \Phi(w)$  implies for all $s \in [0, T]$ that
\begin{align}
\begin{split}
& \mathbb{V}' (\bar{Y}_s) Z_{ \fl{s} } F\big( \bar{Y}_s +
\mathbb{O}_{ \fl{s} } \big) =  Z_{ \fl{s} } \mathbb{V}' (\bar{Y}_s) F\big( \bar{Y}_s +
\mathbb{O}_{ \fl{s} } \big) \\
&\leq Z_{ \fl{s} } \! \left[ \phi\big( 
\mathbb{O}_{ \fl{s} } \big) \mathbb{V}( \bar{Y}_s )  + \varphi( \bar{Y}_s)+ \Phi\big(\mathbb{O}_{ \fl{s} } \big)\right] \leq \phi\big( 
\mathbb{O}_{ \fl{s} } \big) \mathbb{V}( \bar{Y}_s )  + \varphi( \bar{Y}_s)+ \Phi\big(\mathbb{O}_{ \fl{s} } \big).
\end{split}
\end{align}
This together with \eqref{eq:lyapunov:fund} proves for all $t \in [0, T]$ that
\begin{align}
\begin{split}
& e^{- \int_0^t  \phi( \mathbb{O}_{\fl{s} } ) +  f(\mathbb{O}_s) \, ds} \,\mathbb{V}( \bar{Y}_t ) \leq \mathbb{V}(\bar{Y}_0 ) +  \int_0^t e^{- \int_0^s \phi( \mathbb{O}_{\fl{u} } )  +  f(\mathbb{O}_u) \, du} \, \mathbb{V}' (\bar{Y}_s)  \mathbb{A} \bar{Y}_s \, ds  \\
& \quad +  \int_0^t e^{- \int_0^s \phi( \mathbb{O}_{\fl{u} } )  + f(\mathbb{O}_u) \, du} \left[   \phi\big( 
\mathbb{O}_{ \fl{s} } \big) \mathbb{V}( \bar{Y}_s )  + \varphi( \bar{Y}_s)+ \Phi\big(\mathbb{O}_{ \fl{s} } \big) + \eta \mathbb{V}' (\bar{Y}_s) \big( \bar{Y}_s + \mathbb{O}_s\big) \right]  ds \\
&\quad + \int_0^t  e^{- \int_0^s \phi( \mathbb{O}_{\fl{u} } )  +  f(\mathbb{O}_u) \, du} \, Z_{ \fl{s} } \mathbb{V}' (\bar{Y}_s)\! \left[
F\big( Y_{ \fl{s} }\big) - F\big( \bar{Y}_s +
\mathbb{O}_{ \fl{s} } \big)\right]  ds \\
& \quad- \int_0^t \big[ \phi\big( \mathbb{O}_{\fl{s} } \big) +  f(\mathbb{O}_s) \big]   e^{- \int_0^s \phi( \mathbb{O}_{\fl{u} } )  +  f(\mathbb{O}_u)\, du} \, \mathbb{V}(\bar{Y}_s)\, ds.
\end{split}
\end{align}
The assumption that $ \forall \, v, w \in V \colon \eta \mathbb{V}'(v) (v+w) \leq f(w) \mathbb{V}(v) + g(w)$ hence establishes  for all $t \in [0, T]$ that
\begin{align}
\begin{split}
&e^{- \int_0^t  \phi( \mathbb{O}_{\fl{s} } ) +  f(\mathbb{O}_s) \, ds} \, \mathbb{V}( \bar{Y}_t ) \leq \mathbb{V}(\bar{Y}_0 ) +  \int_0^t e^{- \int_0^s \phi( \mathbb{O}_{\fl{u} } )  +  f(\mathbb{O}_u) \, du} \, \mathbb{V}' (\bar{Y}_s)  \mathbb{A} \bar{Y}_s \, ds \\
& \quad +  \int_0^t e^{- \int_0^s \phi( \mathbb{O}_{\fl{u} } )  +  f(\mathbb{O}_u) \, du} \left[   \phi\big( 
\mathbb{O}_{ \fl{s} } \big) \mathbb{V}( \bar{Y}_s )  + \varphi( \bar{Y}_s )+ \Phi\big(\mathbb{O}_{ \fl{s} } \big) + f(\mathbb{O}_s) \mathbb{V} (\bar{Y}_s) +   g (\mathbb{O}_s) \right]  ds \\
& \quad + \int_0^t  e^{- \int_0^s \phi( \mathbb{O}_{\fl{u} } )  +  f(\mathbb{O}_u) \, du} \, Z_{ \fl{s} } \mathbb{V}' (\bar{Y}_s) \!\left[F\big( Y_{ \fl{s} }\big) - F\big( \bar{Y}_s +
\mathbb{O}_{ \fl{s} } \big)\right]  ds \\
& \quad - \int_0^t \big[ \phi\big( \mathbb{O}_{\fl{s} } \big) +  f(\mathbb{O}_s) \big]   e^{- \int_0^s \phi( \mathbb{O}_{\fl{u} } )  + f(\mathbb{O}_u)\, du} \, \mathbb{V}(\bar{Y}_s)\, ds\\
& = \mathbb{V}(\bar{Y}_0 ) +   \int_0^t e^{- \int_0^s \phi( \mathbb{O}_{\fl{u} } )  +  f(\mathbb{O}_u) \, du} \left[ \mathbb{V}' (\bar{Y}_s)  \mathbb{A} \bar{Y}_s + \varphi( \bar{Y}_s )+ \Phi\big(\mathbb{O}_{ \fl{s} } \big)  +  g(\mathbb{O}_s) \right]  ds \\
& \quad + \int_0^t  e^{- \int_0^s \phi( \mathbb{O}_{\fl{u} } )  + f(\mathbb{O}_u) \, du} \, Z_{ \fl{s} } \mathbb{V}' (\bar{Y}_s) \!\left[F\big( Y_{ \fl{s} }\big) - F\big( \bar{Y}_s +\mathbb{O}_{ \fl{s} } \big)\right]  ds.
\end{split}
\end{align}
This assures for all $t \in [0, T]$ that
\begin{align}
\begin{split}
\mathbb{V}( \bar{Y}_t ) & \leq e^{ \int_0^t  \phi( \mathbb{O}_{\fl{s} } ) + f(\mathbb{O}_s) \, ds} \, \mathbb{V}(\bar{Y}_0 )  \\
& \quad +   \int_0^t e^{ \int_s^t \phi( \mathbb{O}_{\fl{u} } )  +  f(\mathbb{O}_u) \, du} \left[ \mathbb{V}' (\bar{Y}_s)  \mathbb{A} \bar{Y}_s + \varphi( \bar{Y}_s )+ \Phi\big(\mathbb{O}_{ \fl{s} } \big)  + g(\mathbb{O}_s) \right]  ds\\
& \quad + \int_0^t  e^{ \int_s^t \phi( \mathbb{O}_{\fl{u} } )  +  f(\mathbb{O}_u) \, du} \, Z_{ \fl{s} } \mathbb{V}' (\bar{Y}_s) \!\left[
F\big( Y_{ \fl{s} }\big) - F\big( \bar{Y}_s +
\mathbb{O}_{ \fl{s} } \big)\right]  ds.
\end{split}
\end{align}
The proof of Lemma~\ref{lemma:lyapunov3} is thus completed.
\end{proof}


\subsection{A priori bounds based on a coercivity-type assumption}

\begin{prop}[A priori bounds]
\label{prop:priori_bound}
Consider the notation in Subsection~\ref{sec:notation},
let $ ( H, \left< \cdot , \cdot \right>_H, \left\| \cdot \right\|_H ) $ be a separable $\mathbb{R}$-Hilbert space,
let $ \mathbb{H} \subseteq H$ be a nonempty orthonormal basis of $ H $,
let $\beta, T \in (0, \infty)$, $\eta, \theta , \vartheta, \kappa, \chi, \varphi \in [0, \infty)$, $\alpha \in \R$, $\rho \in (-\infty,1-\alpha)$, $ \varrho \in [\rho, \rho +1]$, $\psi \in (-\infty, 2-2\varphi)$, $ h \in (0, T] $, $F \in \mathcal{C}(H, H)$, $ A \in L(H) $,
let $\lambda \colon \mathbb{H} \to \R $,
$ Y, O, \mathbb{O} \colon [0, T] \to H $,
and
$\phi, \Phi \colon H \to [0,\infty) $
be functions which satisfy
$\eta O \in \mathcal{C}([0, T], H)$,
$\sup_{b \in \mathbb{H}} \lambda_b < \min\{\eta, \kappa\}$,
and
$ \forall \, b \in \mathbb{H} \colon A b = \lambda_b \,  b $,
let $ ( H_r, \left< \cdot , \cdot \right>_{ H_r }, \left\| \cdot \right\|_{ H_r } ) $, $ r \in \R $, be a family of interpolation spaces associated to $ \kappa- A  $ (cf., e.g., \cite[Section~3.7]{sy02}),
and assume for all $v, w \in H $, $t \in [0,T]$  that
$\left< v, F( v + w ) \right>_H \leq \frac{1}{2} \phi(w) \| v \|^2_H+ \varphi \|(\eta-A)^{\nicefrac{1}{2}} v \|^2_{ H }+ \frac{1}{2}\Phi( w )$,
$\| F(v)\|_{H_{-\alpha}}^2 \leq \theta \max\{ 1, \| v\|_{ H_{\varrho} }^{2 + \vartheta} \} $,
$ \|(\eta-A)^{\nicefrac{-1}{2}} [F(v) - F(w) ]\|_{H}^2\leq \theta \max\{ 1, \|v\|_{ H_{\varrho} }^{\vartheta} \} \|v-w\|_{ H_{\rho} }^2 + \theta \, \|v-w\|^{2 + \vartheta}_{ H_{\rho}}$,
$ \mathbb{O}_t = O_t - \int_0^t e^{(t-s)(A-\eta)} \, \eta  O_s \, ds$,
and
\begin{equation}
\label{eq:scheme_continuous}
Y_t = e^{ t A} ( Y_0 - O_0 )  + \int_0^t e^{ ( t - s ) A } \, \one_{[0, h^{ - \chi }]} \big( \big\|  Y_{ \fl{s} } \big\|_{H_{\varrho}} + \big\|  O_{ \fl{s} } \big\|_{H_{\varrho}}  \big)  F \big(  Y_{ \fl{s} } \big) \, ds + O_t.
\end{equation}
Then
\begin{enumerate}[(i)]
\item it holds that the functions $ [0,T] \ni t \mapsto Y_{t} -\mathbb{O}_t \in H $ and  $ [0,T] \ni t \mapsto \eta \mathbb{O}_{t} \in H $ are continuous and
\item  it holds for all $ t \in [0, T]$ that
\begin{align}
&\| Y_t - \mathbb{O}_t \|_H^2  +  \psi \int_0^t e^{ \int_s^t \phi( \mathbb{O}_{\fl{u} } )  +2\eta (1+\beta) \, du} \, \|(\eta - A)^{\nicefrac{1}{2}}( Y_s - \mathbb{O}_s) \|_{H}^2 \, ds  \nonumber \\
& \leq   e^{ \int_0^t  \phi( \mathbb{O}_{\fl{s} } ) + 2\eta (1+\beta) \, ds} \, \|Y_0 - O_0\|_H^2  +  \int_0^t e^{ \int_s^t \phi( \mathbb{O}_{\fl{u} } )  +2\eta (1+\beta) \, du} \, \Big[  \Phi\big(\mathbb{O}_{ \fl{s} } \big)  + \tfrac{\eta}{2 \beta} \|\mathbb{O}_s\|_H^2 \nonumber   \\
& +  \tfrac{ \theta  e^{h \kappa(2+\vartheta)} [ 1+ (\kappa + \sqrt{\eta} +\sqrt{\eta}|\kappa-\eta|e^{ h\eta} ) \|(\kappa-A)^{ \rho - \varrho } \|_{L(H)} + \sqrt{\theta} \|(\kappa-A)^{ \min\{\alpha+\rho,0\}} \|_{L(H)}+  \sqrt{\eta}  ]^{2+\vartheta}\left|\max \{1 , \smallint\nolimits_{0}^T   \|  \sqrt{\eta} O_u \|_{H_{\varrho}} \, du \}\right|^{2+\vartheta} }{(1- \varphi - \nicefrac{\psi}{2})( 1-\max\{\alpha+\rho,0\})^{2+\vartheta}} \nonumber \\
&  \cdot \max \!\big\{h^2,  h^{2(\varrho - \rho -\chi)},h^{ 2(1-\max\{\alpha+\rho,0\} -( 1 + \nicefrac{\vartheta}{2} ) \chi)} , h \smallint\nolimits_{0}^T   \|  \sqrt{ \eta} O_u \|_{H_{\varrho}}^2 \, du \big\} \\
& \cdot \left|\max \!\big\{ h^{-\chi}, h, h^{ 1-\max\{\alpha+\rho,0\} -( 1 + \nicefrac{\vartheta}{2} ) \chi} , \smallint\nolimits_{0}^T   \|  \sqrt{ \eta} O_u \|_{H_{\varrho}} \, du \big\}\right|^{\vartheta}\Big] \, ds. \nonumber
\end{align}
\end{enumerate}
\end{prop}
\begin{proof}[Proof of Proposition~\ref{prop:priori_bound}]
Throughout this proof let $ Z \colon [0, T] \to [0, 1] $ be the function which satisfies for all $ s \in [0, T]$  that $ Z_s = \one_{[0, h^{ - \chi }]} \! \left( \| Y_{s } \|_{H_{\varrho}} + \| O_{s } \|_{H_{\varrho}}  \right)$, let  $ \mathbb{A} \in L(H) $ be the linear operator given by $\mathbb{A}= A-\eta$, and let $\bar{Y} \colon [0,T] \to H $ be the function which satisfies for all $ t \in [0,T] $  that $\bar{Y}_t = Y_t - \mathbb{O}_t$. Observe that the Cauchy-Schwartz inequality and the fact that $\forall \, a , b \in \R , \, \varepsilon \in (0,\infty) \colon a b \leq \varepsilon  a^2 + \frac{ b^2 }{ 4 \varepsilon }$ prove for all $v, w \in H$ that
\begin{align}
\label{eq:eta:prod}
\begin{split}
2\eta  \langle v, v+w \rangle_H & = 2\eta \|v\|_H^2 +2\eta \langle v, w \rangle_H \leq 2 \eta \|v\|_H^2 +2 \eta \|v\|_H \|w\|_H \\
&\leq 2 \eta \|v\|_H^2 +2 \eta \beta \|v\|_H^2 + \tfrac{\eta}{2\beta} \|w\|_H^2 = 2\eta(1+\beta)\|v\|_H^2 + \tfrac{\eta}{2\beta} \|w\|_H^2 .
\end{split}
\end{align}
In addition, the assumption that  $\eta O \in \mathcal{C}([0, T], H)$ and Lemma~\ref{lemma:fund:gen} (with $V=H$, $A=\mathbb{A}$, $T=T$, $Y= ([0,T] \ni t \mapsto \int_0^t e^{(t-s)\mathbb{A}} \, \eta O_s \, ds \in H)$, $Z= \eta O$ in the notation of Lemma~\ref{lemma:fund:gen}) ensure that $\eta \mathbb{O} \in \mathcal{C}([0,T],H)$. This, \eqref{eq:eta:prod}, and Lemma~\ref{lemma:lyapunov3} (with 
$ V=H$, $T=T$, $\eta=\eta$, $ h=h $, $ A=A$, $\mathbb{V}= (H \ni v \mapsto \|v\|_H^2 \in [0, \infty)) \in \mathcal{C}^1(H, [0, \infty))$, $F =F$, $\varphi=(H \ni v \mapsto 2\varphi \|(\eta-A)^{\nicefrac{1}{2}} v \|^2_{ H } \in [0, \infty)) $, $ Z=Z$,  $Y=Y$, $O=O$, $\mathbb{O}=\mathbb{O}$, $\phi=\phi$, $\Phi=\Phi$, $f= (H \ni v \mapsto 2\eta(1+\beta) \in [0, \infty))$, $g= (H \ni v \mapsto\nicefrac{\eta}{(2 \beta)} \|v\|_H^2 \in [0, \infty))$ in the notation of Lemma~\ref{lemma:lyapunov3}) establish that for all  $ t \in [0, T]$ it holds that $\bar{Y}  \in \mathcal{C}([0,T],H)$ and
\begin{align}
\label{eq:prop:lemma}
&\| \bar{Y}_t \|_H^2  \leq e^{ \int_0^t  \phi( \mathbb{O}_{\fl{s} } ) + 2\eta (1+\beta) \, ds} \, \|\bar{Y}_0 \|_H^2  \nonumber \\
&+   \int_0^t e^{ \int_s^t \phi( \mathbb{O}_{\fl{u} } )  + 2\eta (1+\beta) \, du} \left[ 2 \langle \bar{Y}_s , \mathbb{A} \bar{Y}_s \rangle_H + 2 \varphi \|(\eta-A)^{\nicefrac{1}{2}} \bar{Y}_s \|_{H}^2+ \Phi\big(\mathbb{O}_{ \fl{s} } \big)  + \tfrac{\eta}{2 \beta} \|\mathbb{O}_s\|_H^2 \right]  ds \nonumber\\
& + 2 \int_0^t  e^{ \int_s^t \phi( \mathbb{O}_{\fl{u} } )  + 2\eta (1+\beta) \, du} \, Z_{ \fl{s} } \big<\bar{Y}_s, 
F\big( Y_{ \fl{s} }\big) - F\big( \bar{Y}_s +
\mathbb{O}_{ \fl{s} } \big) \big>_H \,  ds \\
& =  e^{ \int_0^t  \phi( \mathbb{O}_{\fl{s} } ) + 2\eta (1+\beta) \, ds} \, \|\bar{Y}_0\|_H^2 +  \int_0^t e^{ \int_s^t \phi( \mathbb{O}_{\fl{u} } )  +2\eta (1+\beta) \, du} \left[  - \psi \| (-\mathbb{A})^{\nicefrac{1}{2}} \bar{Y}_s \|_{H}^2+ \Phi\big(\mathbb{O}_{ \fl{s} } \big)  + \tfrac{\eta}{2 \beta} \|\mathbb{O}_s\|_H^2 \right]  ds \nonumber  \\
&  +  \int_0^t \! e^{ \int_s^t \phi( \mathbb{O}_{\fl{u} } )  + 2\eta (1+\beta) \, du} \bigl[ -(2- 2\varphi -\psi) \| (-\mathbb{A})^{\nicefrac{1}{2}} \bar{Y}_s\|_{H}^2 +  2 Z_{ \fl{s} } \big\langle \bar{Y}_s,
F\big( Y_{ \fl{s} }\big) - F\big( \bar{Y}_s +
\mathbb{O}_{ \fl{s} } \big)  \big\rangle_H  \bigr]\, ds \nonumber.
\end{align}
Next note that the fact that $2- 2\varphi -\psi > 0$ and the Cauchy-Schwartz inequality  show  for all  $ s \in [0,T] $ that
\begin{align}
& -(2- 2\varphi -\psi) \| (-\mathbb{A})^{\nicefrac{1}{2}} \bar{Y}_s\|_{H}^2 +  2 Z_{ \fl{s} }  \big\langle \bar{Y}_s,
 F\big( Y_{ \fl{s} }\big) - F\big( \bar{Y}_s +
 \mathbb{O}_{ \fl{s} } \big)  \big\rangle_H \nonumber \\
 & \leq    Z_{ \fl{s} } \!\left[-(2- 2\varphi -\psi) \|(-\mathbb{A})^{\nicefrac{1}{2}} \bar{Y}_s\|_{H}^2 + 2  \big< (-\mathbb{A})^{\nicefrac{1}{2}} \bar{Y}_s,  ( - \mathbb{A} )^{\nicefrac{-1}{2}}\big[ F\big(Y_{ \fl{s} }\big)- F\big( \bar{Y}_s +
 \mathbb{O}_{\fl{s} } \big) \big] \big>_H \right] \\
 & \leq Z_{ \fl{s} }\!\left[ -(2- 2\varphi -\psi) \| (-\mathbb{A})^{\nicefrac{1}{2}} \bar{Y}_s\|_{H}^2 + 2  \| (-\mathbb{A})^{\nicefrac{1}{2}} \bar{Y}_s\|_{H}  \big\|(-\mathbb{A})^{\nicefrac{-1}{2}} \big[F\big(Y_{ \fl{s} }\big)- F\big( \bar{Y}_s +
 \mathbb{O}_{\fl{s} } \big) \big]\big\|_{H} \right]. \nonumber
\end{align}
The fact that $\forall \, a , b \in \R , \, \varepsilon \in (0,\infty) \colon 2 a b \leq \varepsilon  a^2 + \frac{ b^2 }{  \varepsilon }$ hence proves for all $ s \in [0, T]$ that
\begin{equation}
\begin{split}
& -(2- 2\varphi -\psi) \| (-\mathbb{A})^{\nicefrac{1}{2}} \bar{Y}_s\|_{H}^2 +  2 Z_{ \fl{s} } \big\langle \bar{Y}_s,
F\big( Y_{ \fl{s} }\big) - F\big( \bar{Y}_s +
\mathbb{O}_{ \fl{s} } \big)  \big\rangle_H \\
&\leq \tfrac{1}{(2- 2\varphi -\psi)} \,  Z_{ \fl{s} } \big\|(-\mathbb{A})^{\nicefrac{-1}{2}} \big[ F\big(Y_{ \fl{s} }\big)- F\big( \bar{Y}_s +
\mathbb{O}_{\fl{s} } \big)  \big]\big\|_{H}^2.
\end{split}
\end{equation}
This together with \eqref{eq:prop:lemma} ensures for all $t \in [0, T]$ that
\begin{align}
\label{eq:estimate0}
\begin{split}
&\| \bar{Y}_t \|_H^2  +  \psi \int_0^t e^{ \int_s^t \phi( \mathbb{O}_{\fl{u} } )  +2\eta (1+\beta) \, du} \,  \| (-\mathbb{A})^{\nicefrac{1}{2}} \bar{Y}_s \|_{H}^2  \, ds \\
& \leq  e^{ \int_0^t  \phi( \mathbb{O}_{\fl{s} } ) + 2\eta (1+\beta) \, ds} \, \|\bar{Y}_0\|_H^2+  \int_0^t e^{ \int_s^t \phi( \mathbb{O}_{\fl{u} } )  +2\eta (1+\beta) \, du} \,\big[  \Phi\big(\mathbb{O}_{ \fl{s} } \big)  + \tfrac{\eta}{2 \beta} \|\mathbb{O}_s\|_H^2 \big] \, ds  \\
& \quad +  \tfrac{1}{(2- 2\varphi -\psi)}  \int_0^t  e^{ \int_s^t \phi( \mathbb{O}_{\fl{u} } )  + 2\eta (1+\beta) \, du} \, Z_{ \fl{s} } \big\|(-\mathbb{A})^{\nicefrac{-1}{2}} \big[ F\big(Y_{ \fl{s} }\big)- F\big( \bar{Y}_s +
\mathbb{O}_{\fl{s} } \big)  \big]\big\|_{H}^2 \, ds.
\end{split}
\end{align}
Furthermore, the assumption that  $ \forall \, v, w \in H \colon \| (-\mathbb{A})^{\nicefrac{-1}{2}} [F(v) - F(w)] \|_{H}^2\leq \theta \max\{ 1, \|v\|_{ H_{\varrho} }^{\vartheta} \}\| v-w\|_{ H_{\rho} }^2 + \theta \, \| v-w\|^{2 + \vartheta}_{ H_{\rho}}$  shows for all $ s \in [0, T]$  that
\begin{equation}
\label{eq:estimate1}
\begin{split}
& Z_{ \fl{s} } \big\|(-\mathbb{A})^{\nicefrac{-1}{2}} \big[ F\big(Y_{ \fl{s} }\big)- F\big( \bar{Y}_s +
\mathbb{O}_{\fl{s} } \big)  \big]\big\|_{H}^2 \\
& \leq Z_{ \fl{s} } \theta  \left[  \max\Big\{ 1, \big\| Y_{ \fl{s} }\big\|_{ H_{\varrho} }^{\vartheta} \Big\} \, \big\|  \bar{Y}_{ \fl{s} } - \bar{Y}_s \big\|_{ H_{\rho}}^2 +   \big\| \bar{Y}_{ \fl{s} } - \bar{Y}_s \big\|_{ H_{\rho} }^{2 + \vartheta} \right] \\
& \leq Z_{ \fl{s} } \theta   \,\big\|  \bar{Y}_{ \fl{s} } - \bar{Y}_s \big\|_{ H_{\rho}}^2 \left[ \max \{1, h^{-\vartheta \chi}  \} +  \big\|  \bar{Y}_{ \fl{s} } - \bar{Y}_s \big\|_{ H_{\rho}}^{\vartheta} \right].
\end{split}
\end{equation}
 Moreover, observe that for all $ s \in [0, T]$ it holds that 
\begin{equation}
\label{eq:priori:Z}
\begin{split}
&Z_{ \fl{s} }  \| Y_{ \fl{s} }- O_{ \fl{s} } - Y_s + O_s \|_{ H_{\rho} } \\
&= Z_{ \fl{s} } \Big\| \left( e^{ ( s - \fl{s} ) A }- \operatorname{Id}_H \right) \big(Y_{ \fl{s} }- O_{ \fl{s} }\big) + \smallint_{\fl{s}}^s e^{  ( s - u ) A } \, F\big( Y_{ \fl{s} } \big) \, du \Big\|_{ H_{\rho}} \\ 
&\leq  Z_{ \fl{s} } \Big[ \big\|\left( e^{ ( s - \fl{s} ) A }- \operatorname{Id}_H \right)  \big(Y_{ \fl{s} }- O_{ \fl{s} }\big) \big\|_{ H_{\rho}} + \smallint_{\fl{s}}^s \big\| e^{  ( s - u ) A } \, F\big( Y_{ \fl{s} } \big) \big\|_{ H_{\rho}} \, du \Big] \\ 
&\leq  Z_{ \fl{s} } \big\| (\kappa-A)^{\rho -\varrho} \left( e^{ ( s - \fl{s} ) A }- \Id_H \right) \! \big\|_{L(H)} \big\| Y_{ \fl{s} }- O_{ \fl{s} }\big\|_{ H_{\varrho}} \\
& \quad + Z_{ \fl{s} } \smallint_{\fl{s}}^s \big\|(\kappa-A)^{\rho+\alpha} e^{  ( s - u ) A } \big\|_{L(H)} \big\| F\big( Y_{ \fl{s} } \big) \big\|_{ H_{-\alpha}} \, du.
\end{split}
\end{equation}
Note that
the fact that
$ \forall \, q \in [ 0, 1 ], \, t \in ( 0, \infty ) \colon
\bigl( 
\| (\kappa-A)^{-q} ( e^{ t (A-\kappa) } - \operatorname{Id}_H ) \|_{L(H)} \leq t^q $
and
$ \| (\kappa-A)^{q} \, e^{ t (A-\kappa) } \|_{L(H)} \leq t^{-q}
\bigr) $
(cf., e.g., Lemma~11.36 in Renardy \& Rogers~\cite{RenardyRogers1993})
and the fact that $\forall \, x \in \R \colon |e^x -1| \leq |x| e^{|x|}$  imply that for all $ s \in [0, T] \backslash \{0, h, 2h, 3h, \ldots\}$, $u \in [\lf s \rf_h, s)$, $r \in \R$ it holds that 
\begin{align}
\label{eq:priori:smooth}
\begin{split}
&\big\| (\kappa-A)^{\rho -\varrho} \left( e^{ ( s - \fl{s} ) (A-r) }- \operatorname{Id}_H \right) \! \big\|_{L(H)} \\
&\leq \big\| (\kappa-A)^{\rho -\varrho} \left( e^{ ( s - \fl{s} ) (A-\kappa) }- \operatorname{Id}_H \right) \! \big\|_{L(H)} + \big\| (\kappa-A)^{\rho -\varrho} \left( e^{ ( s - \fl{s} ) (A-r) }- e^{ ( s - \fl{s} ) (A-\kappa) } \right) \! \big\|_{L(H)} \\
& \leq \left| s - \fl{s} \right|^{ \varrho - \rho } + \big\| (\kappa-A)^{\rho -\varrho} \, e^{ ( s - \fl{s} ) (A-\kappa) }  \big\|_{L(H)} |e^{ ( s - \fl{s} )(\kappa -r)}-1| \\
& \leq h^{ \varrho - \rho } + ( s - \fl{s} )|\kappa -r| e^{ ( s - \fl{s} )|\kappa -r|} \|(\kappa-A)^{ \rho - \varrho } \|_{L(H)} \\
& \leq h^{ \varrho - \rho } +  h |\kappa -r| e^{ h|\kappa -r|} \|(\kappa-A)^{ \rho - \varrho } \|_{L(H)}
\end{split}
\end{align}
and
\begin{align}
\begin{split}
 \big\|(\kappa-A)^{\rho+\alpha} e^{  ( s - u ) A } \big\|_{L(H)} & \leq  \big\|(\kappa-A)^{\rho+\alpha} e^{  ( s - u ) (A-\kappa) } \big\|_{L(H)} e^{(s-u)\kappa}\\
 & \leq ( s - u )^{ - \max\{\alpha+\rho,0\}} \|(\kappa-A)^{ \min\{\alpha+\rho,0\}} \|_{L(H)} e^{h \kappa}.
\end{split}
\end{align}
This, \eqref{eq:priori:Z}, \eqref{eq:priori:smooth}, and the assumption that $\forall \, v \in H  \colon \| F(v)\|_{H_{-\alpha}}^2 \leq \theta \max\{ 1, \| v\|_{ H_{\varrho} }^{2 + \vartheta} \}$  yield that for all $ s \in [0, T] \backslash \{0, h, 2h, 3h, \ldots\}$ it holds that 
\begin{align*}
&Z_{ \fl{s} }  \| Y_{ \fl{s} }- O_{ \fl{s} } - Y_s + O_s \|_{ H_{\rho} }  \leq Z_{ \fl{s} } h^{-\chi} \left[ h^{ \varrho - \rho } + h \kappa e^{h \kappa} \|(\kappa-A)^{ \rho - \varrho } \|_{L(H)}\right] \\
& \quad + Z_{ \fl{s} } e^{h \kappa} \|(\kappa-A)^{ \min\{\alpha+\rho,0\}} \|_{L(H)} \smallint_{ \fl{s}}^s \left( s - u \right)^{ - \max\{\alpha+\rho,0\}} \big\| F\big( Y_{ \fl{s} } \big) \big\|_{ H_{-\alpha}} \, du \\ 
& \leq h^{ \varrho - \rho -\chi } + h^{1-\chi} \kappa e^{h \kappa} \|(\kappa-A)^{ \rho - \varrho } \|_{L(H)}  \\
& \quad + Z_{ \fl{s} } e^{h\kappa} \|(\kappa-A)^{ \min\{\alpha+\rho,0\}} \|_{L(H)} \cdot  \tfrac{\left| s - \fl{s} \right|^{ 1-\max\{\alpha+\rho,0\}}}{1-\max\{\alpha+\rho,0\}} \sqrt{\theta} \max\big\{ 1, \big\| Y_{ \fl{s} }\big\|_{ H_{\varrho} }^{1 + \nicefrac{\vartheta}{2}} \big\} \\
& \leq h^{ \varrho - \rho -\chi } + h^{1-\chi} \kappa e^{h \kappa} \|(\kappa-A)^{ \rho - \varrho } \|_{L(H)}  \\
&\quad +  e^{h\kappa} \|(\kappa-A)^{ \min\{\alpha+\rho,0\}} \|_{L(H)} \cdot  \tfrac{h^{ 1-\max\{\alpha+\rho,0\}}}{1-\max\{\alpha+\rho,0\}} \sqrt{\theta} \max\big\{ 1,h^{-(1+\nicefrac{\vartheta}{2})\chi} \big\} \numberthis  \label{eq:estimate2} \\
& \leq \tfrac{e^{h \kappa}}{ 1-\max\{\alpha+\rho,0\}} \Big[  h^{\varrho - \rho -\chi}+  h^{1-\chi} \kappa  \|(\kappa-A)^{ \rho - \varrho } \|_{L(H)} \\
&\quad + \sqrt{\theta} \|(\kappa-A)^{ \min\{\alpha+\rho,0\}} \|_{L(H)}\, h^{ 1-\max\{\alpha+\rho,0\} } \max\{ 1, h^{-(1+\nicefrac{\vartheta}{2})\chi} \}\Big] \\
& \leq  \frac{ e^{h \kappa}(1+ \kappa  \|(\kappa-A)^{ \rho - \varrho } \|_{L(H)} + \sqrt{\theta} \|(\kappa-A)^{ \min\{\alpha+\rho,0\}} \|_{L(H)} )  }{( 1-\max\{\alpha+\rho,0\})}\\
& \quad \cdot  \max\! \left\{  h^{\varrho - \rho -\chi}, h^{1-\chi}, h^{  1-\max\{\alpha+\rho,0\} } , h^{ 1-\max\{\alpha+\rho,0\} -(1+\nicefrac{\vartheta}{2})\chi}   \right\}  \\
& \leq  \frac{ e^{h \kappa}(1+ \kappa  \|(\kappa-A)^{ \rho - \varrho } \|_{L(H)} + \sqrt{\theta} \|(\kappa-A)^{ \min\{\alpha+\rho,0\}} \|_{L(H)} ) \max\! \left\{ h, h^{\varrho - \rho -\chi}, h^{ 1-\max\{\alpha+\rho,0\} -(1+\nicefrac{\vartheta}{2}) \chi}   \right\}      }{( 1-\max\{\alpha+\rho,0\})} .
\end{align*}
Moreover,
the fact that
$ \forall \, t \in [ 0, \infty ) \colon
\| e^{ t \mathbb{A} } \|_{L(H)} \leq 1 $
and \eqref{eq:priori:smooth} prove  for all $ s \in [0, T]\backslash \{0, h, 2h, 3h, \ldots\}$ that 
\begin{align*}
& \Big\| \smallint_0^s e^{(s-u)\mathbb{A}} \, \eta O_u \, du - \smallint_0^{\lf s \rf_h} e^{(\lf s \rf_h-u)\mathbb{A}} \, \eta O_u \, du  \Big\|_{H_{\rho}}\\
&= \Big\| \left(e^{(s-\fl{s})\mathbb{A}}- \Id_H\right) \smallint_0^{\lf s \rf_h} e^{(\lf s \rf_h-u)\mathbb{A}} \, \eta O_u \, du + \smallint_{\lf s \rf_h}^s e^{(s-u)\mathbb{A}} \, \eta O_u \, du \Big\|_{H_{\rho}} \\
& \leq \|(\kappa-A)^{\rho- \varrho} \left(e^{(s-\fl{s})\mathbb{A}}- \Id_H\right) \big\|_{L(H)} \smallint_0^{\lf s \rf_h} \big\| e^{(\lf s \rf_h-u)\mathbb{A}} \big\|_{L(H)} \, \| \eta O_u \|_{H_{\varrho}} \, du \numberthis \\
& \quad + \smallint_{\lf s \rf_h}^s  \|(\kappa-A)^{\rho-\varrho}  \|_{L(H)} \|e^{(s-u)\mathbb{A}} \|_{L(H)} \|   \eta O_u \|_{H_{\varrho}} \, du  \\
& \leq \big[ h^{ \varrho - \rho } +  h |\kappa -\eta| e^{ h|\kappa -\eta|} \|(\kappa-A)^{ \rho - \varrho } \|_{L(H)} \big] \smallint_0^{\lf s \rf_h} \| \eta O_u \|_{H_{\varrho}} \, du + \|(\kappa-A)^{\rho-\varrho}  \|_{L(H)} \smallint_{\lf s \rf_h}^s   \|   \eta O_u \|_{H_{\varrho}} \, du \\
&\leq \big[ h^{ \varrho - \rho } +  h |\kappa -\eta| e^{ h|\kappa -\eta|} \|(\kappa-A)^{ \rho - \varrho } \|_{L(H)} \big] \smallint_0^T \| \eta O_u \|_{H_{\varrho}} \, du + \|(\kappa-A)^{\rho-\varrho}  \|_{L(H)} \smallint_{\lf s \rf_h}^s   \|   \eta O_u \|_{H_{\varrho}} \, du.
\end{align*}
This and \eqref{eq:estimate2} ensure for all $s \in [0, T]$ that
\begin{align*}
& Z_{ \fl{s} }  \| \bar{Y}_{ \fl{s} } - \bar{Y}_s  \|_{ H_{\rho} } = Z_{ \fl{s} }   \Big\|  Y_{ \fl{s} }- O_{ \fl{s} } + \smallint_0^{\lf s \rf_h} e^{(\lf s \rf_h-u)\mathbb{A}} \, \eta O_u \, du - Y_s + O_s - \smallint_0^s e^{(s-u)\mathbb{A}} \, \eta O_u \, du \Big\|_{ H_{\rho} } \\
& \leq Z_{ \fl{s} }  \| Y_{ \fl{s} }- O_{ \fl{s} } - Y_s + O_s \|_{ H_{\rho} } \\
& \quad + Z_{ \fl{s} } \Big\|\smallint_0^s e^{(s-u)\mathbb{A}} \, \eta O_u \, du - \smallint_0^{\lf s \rf_h} e^{(\lf s \rf_h-u)\mathbb{A}} \, \eta O_u \, du  \Big\|_{H_{\rho}} \\
& \leq \frac{ e^{h \kappa}(1+ \kappa  \|(\kappa-A)^{ \rho - \varrho } \|_{L(H)} + \sqrt{\theta} \|(\kappa-A)^{ \min\{\alpha+\rho,0\}} \|_{L(H)} ) \max\! \left\{ h, h^{\varrho - \rho -\chi}, h^{ 1-\max\{\alpha+\rho,0\} -(1+\nicefrac{\vartheta}{2}) \chi}   \right\}      }{( 1-\max\{\alpha+\rho,0\})}\\
& \quad +  \big[ h^{ \varrho - \rho } +  h |\kappa -\eta| e^{ h|\kappa -\eta|} \|(\kappa-A)^{ \rho - \varrho } \|_{L(H)} \big] \smallint_0^T \| \eta O_u \|_{H_{\varrho}} \, du + \|(\kappa-A)^{\rho-\varrho}  \|_{L(H)} \smallint_{\lf s \rf_h}^s   \|   \eta O_u \|_{H_{\varrho}} \, du\\
& \leq \frac{ e^{h \kappa}(1+ \kappa  \|(\kappa-A)^{ \rho - \varrho } \|_{L(H)} + \sqrt{\theta} \|(\kappa-A)^{ \min\{\alpha+\rho,0\}} \|_{L(H)} )   \max\! \left\{ h, h^{\varrho - \rho -\chi}, h^{ 1-\max\{\alpha+\rho,0\} -(1+\nicefrac{\vartheta}{2}) \chi}   \right\}   }{( 1-\max\{\alpha+\rho,0\})}\\
& \quad + e^{h \kappa} \big[ h^{ \varrho - \rho } +  h |\kappa -\eta| e^{ h\eta} \|(\kappa-A)^{ \rho - \varrho } \|_{L(H)} \big] \smallint_0^T \| \eta O_u \|_{H_{\varrho}} \, du \numberthis \label{eq:bar:Y} \\
& \quad + e^{h \kappa}\sqrt{\eta} \|(\kappa-A)^{\rho-\varrho}  \|_{L(H)} \smallint_{\lf s \rf_h}^s   \|   \sqrt{\eta} O_u \|_{H_{\varrho}} \, du \\
& \leq \tfrac{ e^{h \kappa} [1+ \kappa  \|(\kappa-A)^{ \rho - \varrho } \|_{L(H)} + \sqrt{\theta} \|(\kappa-A)^{ \min\{\alpha+\rho,0\}} \|_{L(H)} + (1+ |\kappa-\eta|e^{ h\eta} \|(\kappa-A)^{\rho-\varrho}  \|_{L(H)})  \smallint_0^T \| \eta O_u \|_{H_{\varrho}} \, du + \sqrt{\eta} \|(\kappa-A)^{\rho-\varrho}  \|_{L(H)} ]   }{( 1-\max\{\alpha+\rho,0\})} \\
&\quad  \cdot  \max \!\big\{h,  h^{\varrho - \rho -\chi},h^{ 1-\max\{\alpha+\rho,0\} -( 1 + \nicefrac{\vartheta}{2} ) \chi} , \smallint\nolimits_{\lf s \rf_h}^s   \|  \sqrt{ \eta} O_u \|_{H_{\varrho}} \, du \big\} .
\end{align*}
Next note that for all $a \in [1, \infty)$, $b,c \in [0, \infty)$ it holds that
\begin{align}
\begin{split}
|ab|^2 (c+ |ab|^{\vartheta}) \leq a^{2+\vartheta}b^2 (c+ b^{\vartheta}) \leq 2 a^{2+\vartheta} b^2 \max\{ c, b^{\vartheta}\}  .
\end{split}
\end{align}
This, \eqref{eq:estimate1}, and \eqref{eq:bar:Y} show for all $s \in [0, T]$ that
\begin{align}
\label{eq:F:term}
\begin{split}
& Z_{ \fl{s} } \big\| (-\mathbb{A})^{\nicefrac{-1}{2}} \big[F\big(Y_{ \fl{s} }\big)- F\big(\bar{Y}_s + \mathbb{O}_{ \fl{s} }\big) \big] \big\|_{ H }^2  \\ 
& \leq \tfrac{2 \theta e^{h \kappa(2+\vartheta)} [ 1+ (\kappa + \sqrt{\eta}) \|(\kappa-A)^{ \rho - \varrho } \|_{L(H)} + \sqrt{\theta} \|(\kappa-A)^{ \min\{\alpha+\rho,0\}} \|_{L(H)} + (1+ |\kappa-\eta|e^{ h\eta} \|(\kappa-A)^{\rho-\varrho}  \|_{L(H)})   \smallint_0^T \| \eta O_u \|_{H_{\varrho}} \, du   ]^{2+\vartheta}   }{( 1-\max\{\alpha+\rho,0\})^{2+\vartheta}}\\
& \quad \cdot \left|\max \!\big\{h,  h^{\varrho - \rho -\chi},h^{ 1-\max\{\alpha+\rho,0\} -( 1 + \nicefrac{\vartheta}{2} ) \chi} , \smallint\nolimits_{\lf s \rf_h}^s   \|  \sqrt{ \eta} O_u \|_{H_{\varrho}} \, du \big\} \right|^2 \\
& \quad \cdot \left|\max \!\big\{1, h^{-\chi}, h,  h^{\varrho - \rho -\chi},h^{ 1-\max\{\alpha+\rho,0\} -( 1 + \nicefrac{\vartheta}{2} ) \chi} , \smallint\nolimits_{\lf s \rf_h}^s   \|  \sqrt{ \eta} O_u \|_{H_{\varrho}} \, du \big\}\right|^{\vartheta}\\
&= \tfrac{2 \theta e^{h \kappa(2+\vartheta)} [ 1+ (\kappa + \sqrt{\eta}) \|(\kappa-A)^{ \rho - \varrho } \|_{L(H)} + \sqrt{\theta} \|(\kappa-A)^{ \min\{\alpha+\rho,0\}} \|_{L(H)} + (1+ |\kappa-\eta|e^{ h\eta} \|(\kappa-A)^{\rho-\varrho}  \|_{L(H)})   \smallint_0^T \| \eta O_u \|_{H_{\varrho}} \, du   ]^{2+\vartheta}   }{( 1-\max\{\alpha+\rho,0\})^{2+\vartheta}}\\
& \quad \cdot \left|\max \!\big\{h,  h^{\varrho - \rho -\chi},h^{ 1-\max\{\alpha+\rho,0\} -( 1 + \nicefrac{\vartheta}{2} ) \chi} , \smallint\nolimits_{\lf s \rf_h}^s   \|  \sqrt{ \eta} O_u \|_{H_{\varrho}} \, du \big\} \right|^2 \\
& \quad \cdot \left|\max \!\big\{ h^{-\chi}, h, h^{ 1-\max\{\alpha+\rho,0\} -( 1 + \nicefrac{\vartheta}{2} ) \chi} , \smallint\nolimits_{\lf s \rf_h}^s   \|  \sqrt{ \eta} O_u \|_{H_{\varrho}} \, du \big\}\right|^{\vartheta}.
\end{split}
\end{align}
Observe that H\"olders inequality implies for all $s \in [0,T]$ that
\begin{align}
\begin{split}
&\left|\max \!\big\{h,  h^{\varrho - \rho -\chi},h^{ 1-\max\{\alpha+\rho,0\} -( 1 + \nicefrac{\vartheta}{2} ) \chi} , \smallint\nolimits_{\lf s \rf_h}^s   \|  \sqrt{ \eta} O_u \|_{H_{\varrho}} \, du \big\} \right|^2 \\
& \cdot \left|\max \!\big\{ h^{-\chi}, h, h^{ 1-\max\{\alpha+\rho,0\} -( 1 + \nicefrac{\vartheta}{2} ) \chi} , \smallint\nolimits_{\lf s \rf_h}^s   \|  \sqrt{ \eta} O_u \|_{H_{\varrho}} \, du \big\}\right|^{\vartheta}\\
&\leq \left| \max \!\Big\{h,  h^{\varrho - \rho -\chi},h^{ 1-\max\{\alpha+\rho,0\} -( 1 + \nicefrac{\vartheta}{2} ) \chi} , \sqrt{h} \sqrt{ \smallint\nolimits_{\lf s \rf_h}^s   \|  \sqrt{ \eta} O_u \|_{H_{\varrho}}^2 \, du} \Big\} \right|^2 \\
&\quad \cdot \left|\max \!\big\{ h^{-\chi}, h, h^{ 1-\max\{\alpha+\rho,0\} -( 1 + \nicefrac{\vartheta}{2} ) \chi} , \smallint\nolimits_{\lf s \rf_h}^s   \|  \sqrt{ \eta} O_u \|_{H_{\varrho}} \, du \big\}\right|^{\vartheta}\\
&\leq \max \!\big\{h^2,  h^{2(\varrho - \rho -\chi)},h^{ 2(1-\max\{\alpha+\rho,0\} -( 1 + \nicefrac{\vartheta}{2} ) \chi)} , h \smallint\nolimits_{0}^T   \|  \sqrt{ \eta} O_u \|_{H_{\varrho}}^2 \, du \big\}  \\
&\quad \cdot \left|\max \!\big\{ h^{-\chi}, h, h^{ 1-\max\{\alpha+\rho,0\} -( 1 + \nicefrac{\vartheta}{2} ) \chi} , \smallint\nolimits_{0}^T   \|  \sqrt{ \eta} O_u \|_{H_{\varrho}} \, du \big\}\right|^{\vartheta}.
\end{split}
\end{align}
This and \eqref{eq:F:term} establish for all $s \in [0, T]$ that
\begin{align}
\label{eq:estimate3}
\begin{split}
&  Z_{ \fl{s} } \big\| (-\mathbb{A})^{\nicefrac{-1}{2}} \big[F\big(Y_{ \fl{s} }\big)- F\big(\bar{Y}_s + \mathbb{O}_{ \fl{s} }\big) \big] \big\|_{ H }^2  \\
& \leq \tfrac{2 \theta  e^{h \kappa(2+\vartheta)} [ 1+ (\kappa + \sqrt{\eta}) \|(\kappa-A)^{ \rho - \varrho } \|_{L(H)} + \sqrt{\theta} \|(\kappa-A)^{ \min\{\alpha+\rho,0\}} \|_{L(H)} + (1+ |\kappa-\eta|e^{ h\eta} \|(\kappa-A)^{\rho-\varrho}  \|_{L(H)})   \smallint_0^T \| \eta O_u \|_{H_{\varrho}} \, du   ]^{2+\vartheta}   }{( 1-\max\{\alpha+\rho,0\})^{2+\vartheta}}\\
& \quad \cdot  \max \!\big\{h^2,  h^{2(\varrho - \rho -\chi)},h^{ 2(1-\max\{\alpha+\rho,0\} -( 1 + \nicefrac{\vartheta}{2} ) \chi)} , h \smallint\nolimits_{0}^T   \|  \sqrt{ \eta} O_u \|_{H_{\varrho}}^2 \, du \big\}  \\
&\quad \cdot \left|\max \!\big\{ h^{-\chi}, h, h^{ 1-\max\{\alpha+\rho,0\}-( 1 + \nicefrac{\vartheta}{2} ) \chi} , \smallint\nolimits_{0}^T   \|  \sqrt{ \eta} O_u \|_{H_{\varrho}} \, du \big\}\right|^{\vartheta}.
\end{split}
\end{align}
Combining \eqref{eq:estimate3} with \eqref{eq:estimate0} yields that for all $t \in [0, T]$ it holds that
\begin{align}
\begin{split}
&\| \bar{Y}_t \|_H^2  +  \psi \int_0^t e^{ \int_s^t \phi( \mathbb{O}_{\fl{u} } )  +2\eta (1+\beta) \, du}\,  \| (-\mathbb{A})^{\nicefrac{1}{2}} \bar{Y}_s \|_{H}^2  \, ds \\
& \leq  e^{ \int_0^t  \phi( \mathbb{O}_{\fl{s} } ) + 2\eta (1+\beta) \, ds} \, \|\bar{Y}_0\|_H^2+  \int_0^t e^{ \int_s^t \phi( \mathbb{O}_{\fl{u} } )  +2\eta (1+\beta) \, du} \left[  \Phi\big(\mathbb{O}_{ \fl{s} } \big)  + \tfrac{\eta}{2 \beta} \|\mathbb{O}_s\|_H^2 \right]  ds  \\
&  + \tfrac{2 \theta  e^{h \kappa(2+\vartheta)} [ 1+ (\kappa + \sqrt{\eta}) \|(\kappa-A)^{ \rho - \varrho } \|_{L(H)} + \sqrt{\theta} \|(\kappa-A)^{ \min\{\alpha+\rho,0\}} \|_{L(H)} + (1+ |\kappa-\eta|e^{ h\eta} \|(\kappa-A)^{\rho-\varrho}  \|_{L(H)})   \smallint_0^T \| \eta O_u \|_{H_{\varrho}} \, du   ]^{2+\vartheta}   }{(2- 2\varphi -\psi)( 1-\max\{\alpha+\rho,0\})^{2+\vartheta}}    \\
&   \cdot \max \!\big\{h^2,  h^{2(\varrho - \rho -\chi)},h^{ 2(1-\max\{\alpha+\rho,0\} -( 1 + \nicefrac{\vartheta}{2} ) \chi)} , h \smallint\nolimits_{0}^T   \|  \sqrt{ \eta} O_u \|_{H_{\varrho}}^2 \, du \big\}  \\
& \cdot \left|\max \!\big\{ h^{-\chi}, h, h^{ 1-\max\{\alpha+\rho,0\} -( 1 + \nicefrac{\vartheta}{2} ) \chi} , \smallint\nolimits_{0}^T   \|  \sqrt{ \eta} O_u \|_{H_{\varrho}} \, du \big\}\right|^{\vartheta} \int_0^t  e^{ \int_s^t \phi( \mathbb{O}_{\fl{u} } )  + 2\eta (1+\beta) \, du} \, ds.
\end{split}
\end{align}
This assures for all $ t \in [0, T]$ that
\begin{align*}
&\| \bar{Y}_t \|_H^2  +  \psi \int_0^t e^{ \int_s^t \phi( \mathbb{O}_{\fl{u} } )  +2\eta (1+\beta) \, du} \, \| (-\mathbb{A})^{\nicefrac{1}{2}} \bar{Y}_s \|_{H}^2  \, ds \\
&\leq  e^{ \int_0^t  \phi( \mathbb{O}_{\fl{s} } ) + 2\eta (1+\beta) \, ds} \, \|\bar{Y}_0\|_H^2  +  \int_0^t e^{ \int_s^t \phi( \mathbb{O}_{\fl{u} } )  +2\eta (1+\beta) \, du} \, \Big[  \Phi\big(\mathbb{O}_{ \fl{s} } \big)  + \tfrac{\eta}{2 \beta} \|\mathbb{O}_s\|_H^2    \\
&  +  \tfrac{ \theta  e^{h \kappa(2+\vartheta)} [ 1+ (\kappa + \sqrt{\eta}) \|(\kappa-A)^{ \rho - \varrho } \|_{L(H)} + \sqrt{\theta}\|(\kappa-A)^{ \min\{\alpha+\rho,0\}} \|_{L(H)} + (1+ |\kappa-\eta|e^{ h\eta} \|(\kappa-A)^{\rho-\varrho}  \|_{L(H)})  \sqrt{\eta}  ]^{2+\vartheta} }{(1- \varphi -\nicefrac{\psi}{2})( 1-\max\{\alpha+\rho,0\})^{2+\vartheta}}  \\
& \cdot \left|\max \{1 , \smallint\nolimits_{0}^T   \|  \sqrt{\eta} O_u \|_{H_{\varrho}} \, du \}\right|^{2+\vartheta}  \max \!\big\{h^2,  h^{2(\varrho - \rho -\chi)},h^{ 2(1-\max\{\alpha+\rho,0\} -( 1 + \nicefrac{\vartheta}{2} ) \chi)} , h \smallint\nolimits_{0}^T   \|  \sqrt{ \eta} O_u \|_{H_{\varrho}}^2 \, du \big\}  \\
&  \cdot \left|\max \!\big\{ h^{-\chi}, h, h^{ 1-\max\{\alpha+\rho,0\} -( 1 + \nicefrac{\vartheta}{2} ) \chi} , \smallint\nolimits_{0}^T   \|  \sqrt{ \eta} O_u \|_{H_{\varrho}} \, du \big\}\right|^{\vartheta}\Big] \, ds \numberthis \\
&=  e^{ \int_0^t  \phi( \mathbb{O}_{\fl{s} } ) + 2\eta (1+\beta) \, ds} \, \|\bar{Y}_0\|_H^2  +  \int_0^t e^{ \int_s^t \phi( \mathbb{O}_{\fl{u} } )  +2\eta (1+\beta) \, du} \, \Big[  \Phi\big(\mathbb{O}_{ \fl{s} } \big)  + \tfrac{\eta}{2 \beta} \|\mathbb{O}_s\|_H^2    \\
&  +  \tfrac{ \theta  e^{h \kappa(2+\vartheta)} [ 1+ (\kappa + \sqrt{\eta} +\sqrt{\eta}|\kappa-\eta|e^{ h\eta} ) \|(\kappa-A)^{ \rho - \varrho } \|_{L(H)} + \sqrt{\theta}\|(\kappa-A)^{ \min\{\alpha+\rho,0\}} \|_{L(H)} +  \sqrt{\eta}  ]^{2+\vartheta}\left|\max \{1 , \smallint\nolimits_{0}^T   \|  \sqrt{\eta} O_u \|_{H_{\varrho}} \, du \}\right|^{2+\vartheta} }{(1- \varphi -\nicefrac{\psi}{2})( 1-\max\{\alpha+\rho,0\})^{2+\vartheta}}  \\
& \cdot \max \!\big\{h^2,  h^{2(\varrho - \rho -\chi)},h^{ 2(1-\max\{\alpha+\rho,0\} -( 1 + \nicefrac{\vartheta}{2} ) \chi)} , h \smallint\nolimits_{0}^T   \|  \sqrt{ \eta} O_u \|_{H_{\varrho}}^2 \, du \big\}  \\
&  \cdot \left|\max \!\big\{ h^{-\chi}, h, h^{ 1-\max\{\alpha+\rho,0\} -( 1 + \nicefrac{\vartheta}{2} ) \chi} , \smallint\nolimits_{0}^T   \|  \sqrt{ \eta} O_u \|_{H_{\varrho}} \, du \big\}\right|^{\vartheta}\Big] \, ds.
\end{align*}
The proof of Proposition~\ref{prop:priori_bound} is thus completed.
\end{proof}

The next result, Corollary~\ref{cor:a_priori}, follows immediately from Proposition~\ref{prop:priori_bound} above.

\begin{cor}
\label{cor:a_priori}
Consider the notation in Subsection~\ref{sec:notation},
let $ ( H, \left< \cdot , \cdot \right>_H, \left\| \cdot \right\|_H ) $ be a separable $\mathbb{R}$-Hilbert space,
let $ \mathbb{H} \subseteq H$ be a nonempty orthonormal basis of $ H $,
let $\beta, T \in (0, \infty)$, $\eta, \theta , \vartheta, \kappa \in [0, \infty)$, $ \varphi \in [0,1)$, $\alpha \in \R$, $\rho \in [-\alpha,1-\alpha)$, $ \varrho \in [\rho, \rho +1]$, $\chi \in [0, \nicefrac{(2-2\alpha-2\rho)}{(1+\vartheta)} ]$, 
$ h \in (0, \min\{1,T\}] $,
$F \in \mathcal{C}(H, H)$,
$ A \in L(H) $,
let $\lambda \colon \mathbb{H} \to \R $,
$ Y, O, \mathbb{O} \colon [0, T] \to H $,
and
$\phi, \Phi \colon H \to [0,\infty) $
be functions which satisfy
$\eta O \in \mathcal{C}([0, T], H)$,
$\sup_{b \in \mathbb{H}} \lambda_b < \min\{\eta, \kappa\}$,
and
$ \forall \, b \in \mathbb{H} \colon A b = \lambda_b \,  b $,
let $ ( H_r, \left< \cdot , \cdot \right>_{ H_r }, \left\| \cdot \right\|_{ H_r } ) $, $ r \in \R $, be a family of interpolation spaces associated to $ \kappa- A  $,
and assume for all $v, w \in H $, $t \in [0,T]$  that
$\left< v, F( v + w ) \right>_H \leq \frac{1}{2} \phi(w) \| v \|^2_H+ \varphi \|(\eta-A)^{\nicefrac{1}{2}} v \|^2_{ H }+ \frac{1}{2}\Phi( w )$,
$\| F(v)\|_{H_{-\alpha}}^2 \leq \theta \max\{ 1, \| v\|_{ H_{\varrho} }^{2 + \vartheta} \} $,
$ \|(\eta-A)^{\nicefrac{-1}{2}} [F(v) - F(w) ]\|_{H}^2\leq \theta \max\{ 1, \|v\|_{ H_{\varrho} }^{\vartheta} \} \|v-w\|_{ H_{\rho} }^2 + \theta \, \|v-w\|^{2 + \vartheta}_{ H_{\rho}}$,
$ \mathbb{O}_t = O_t - \int_0^t e^{(t-s)(A-\eta)} \, \eta  O_s \, ds$,
and
\begin{equation}
Y_t =  \int_0^t e^{ ( t - s ) A } \, \one_{[0, h^{ - \chi }]} \big( \big\|  Y_{ \fl{s} } \big\|_{H_{\varrho}} + \big\|  O_{ \fl{s} } \big\|_{H_{\varrho}}  \big)  F \big(  Y_{ \fl{s} } \big) \, ds + O_t.
\end{equation}
Then it holds for all $ t \in [0, T]$ that $\eta \mathbb{O} \in \mathcal{C}([0,T],H)$ and
\begin{align}
\begin{split}
&\| Y_t - \mathbb{O}_t \|_H^2  \leq    \int_0^t e^{ \int_s^t \phi( \mathbb{O}_{\fl{u} } )  +2\eta (1+\beta) \, du} \, \Big[  \Phi\big(\mathbb{O}_{ \fl{s} } \big)  + \tfrac{\eta}{2 \beta} \|\mathbb{O}_s\|_H^2    \\
&\quad + \tfrac{ \theta  e^{ \kappa(2+\vartheta)} [ 1+ (\kappa + \sqrt{\eta} +\sqrt{\eta}|\kappa-\eta|e^{ \eta} ) \|(\kappa-A)^{ \rho - \varrho } \|_{L(H)} + \sqrt{\theta} +  \sqrt{\eta} ]^{2+\vartheta} \left|\max \{1 , \smallint\nolimits_{0}^T   \|  \sqrt{\eta} O_u \|_{H_{\varrho}} \, du \}\right|^{2+\vartheta}}{(1- \varphi)( 1-\alpha - \rho)^{2+\vartheta}}   \\
& \quad \cdot  \max \!\big\{  h^{2(\varrho - \rho -\chi)},h^{ 2(1-\alpha - \rho -( 1 + \nicefrac{\vartheta}{2} ) \chi)} , h \smallint\nolimits_{0}^T   \|  \sqrt{ \eta} O_u \|_{H_{\varrho}}^2 \, du \big\} \left|\max \!\big\{ h^{-\chi},  \smallint\nolimits_{0}^T   \|  \sqrt{ \eta} O_u \|_{H_{\varrho}} \, du \big\}\right|^{\vartheta}\Big] ds.
\end{split}
\end{align}
\end{cor}

\section{Main result}
\label{sec:main}

In the main result of this article, Theorem~\ref{thm:strong} below,
we establish strong convergence for an explicit space-time discrete numerical approximation scheme
for a large class of SEEs.
Before presenting Theorem~\ref{thm:strong},
we provide a few elementary and well-known results in Lemmas~\ref{lem:fast_convergence}--\ref{lem:subsequence} below.
These auxiliary lemmas as well as
a pathwise convergence result (see Proposition~\ref{prop:main_det} below)
and pathwise a~priori bounds (see Proposition~\ref{prop:main_det:2} below)
are required in the proof of Theorem~\ref{thm:strong}.

\subsection{Fast convergence in probability}

\begin{lemma}
\label{lem:fast_convergence}
Let $ ( \Omega, \F, \P ) $ be a probability space, let $ ( E, d ) $ be a metric space, and let $ X_n \colon \Omega \to E $, $ n \in \N_0 $, be strongly $ \F $/$ ( E, d ) $-measurable functions which satisfy $ \sum_{ n = 1 }^{ \infty } \E\big[\! \min\!\big\{ 1 , d( X_n, X_0 ) \big\} \big] < \infty $. Then it holds that $\{ \limsup_{n \to \infty} d(X_n, X_0)=0\} \in \F$ and $\P \bigl( \limsup_{n \to \infty} d(X_n, X_0)=0 \bigr)=1$. 
\end{lemma} 
\begin{proof}[Proof of Lemma~\ref{lem:fast_convergence}] 
Note that the assumption that $ \sum_{ n = 1 }^{ \infty } \E\big[ \!\min\!\big\{ 1 , d( X_n, X_0 ) \big\} \big] < \infty $ and Markov's inequality ensure  for all $ \varepsilon \in ( 0, 1 ] $ that 
\begin{equation} 
\begin{split} 
\sum_{ n = 1 }^{ \infty } \P\bigl( d( X_n, X_0 ) \geq \varepsilon \bigr) & = \sum_{ n = 1 }^{ \infty } \P\bigl( \min\{ 1, d( X_n, X_0 ) \} \geq \varepsilon \bigr) \\ 
& \leq \frac{1}{\varepsilon} \sum_{ n = 1 }^{ \infty }  \E\big[ \!\min\{ 1, d( X_n, X_0 ) \} \big]  < \infty . 
\end{split} 
\end{equation} 
The Borel-Cantelli lemma  hence implies for all $ \varepsilon \in ( 0, 1 ] $  that 
\begin{equation} 
\P\bigg( \bigcap_{n=1}^{\infty} \bigcup_{m=n}^{\infty} \left\{ d( X_m , X_0 ) \geq \varepsilon \right\} \bigg) = 0 . 
\end{equation} 
This proves for all $ \varepsilon \in ( 0, 1 ] $  that 
\begin{equation} 
\label{eq:fast_conv}
\P\bigg( \bigcup_{n=1}^{\infty} \bigcap_{m=n}^{\infty} \left\{ d( X_m , X_0 ) < \varepsilon \right\}  \bigg)  = 1 . 
\end{equation} 
Moreover, note that
\begin{align}
\begin{split}
&\bigg\{\! \limsup_{n \to \infty} d(X_n, X_0)=0\bigg\} = \bigcap_{\varepsilon \in (0,\infty) \cap \mathbb{Q}} \bigcup_{n=1}^{\infty} \bigcap_{m=n}^{\infty} \left\{ d( X_m , X_0 ) < \varepsilon \right\} \in  \mathcal{F}.
\end{split}
\end{align}
Equation~\eqref{eq:fast_conv} hence shows that 
\begin{equation} 
\begin{split} 
\P\biggl( \limsup_{n \to \infty} d(X_n, X_0)=0 \biggr) & = \P\bigg( \bigcap_{\varepsilon \in (0,\infty) \cap \mathbb{Q}} \bigcup_{n=1}^{\infty} \bigcap_{m=n}^{\infty} \left\{ d( X_m , X_0 ) < \varepsilon \right\} \bigg) \\ 
& = \lim_{ \varepsilon \searrow 0   } \P\bigg( \bigcup_{n=1}^{\infty} \bigcap_{m=n}^{\infty} \left\{ d( X_m , X_0 ) < \varepsilon \right\} \bigg) = 1 . 
\end{split} 
\end{equation} 
The proof of Lemma~\ref{lem:fast_convergence} is thus completed. 
\end{proof}

\subsection{A characterization of convergent sequences in topological spaces}
\begin{lemma}
\label{lem:subsequence}
Let $ ( E, \mathcal{E} ) $ be a topological space and let $ e \colon \N_0 \to E $ be a function. Then the following seven statements are equivalent: \begin{enumerate}[(i)] 
\item \label{item:con} It holds that $ e_n \in E $, $ n \in \N $, converges in $ ( E, \mathcal{E} ) $ to $  e_0 $. 
\item \label{item:infty:inc} For every function $ k \colon \N \to \N $ with $  \liminf_{ n \to \infty } k(n) = \infty $ there exists a strictly increasing function $  l \colon \N \to \N $  such that $  e_{ k( l(n) ) } \in E $, $ n \in \N $, converges in $ ( E, \mathcal{E} ) $ to $  e_0 $.
\item \label{item:infty:infty} For every function $ k \colon \N \to \N $ with $  \liminf_{ n \to \infty } k(n) = \infty $ there exists a function $  l \colon \N \to \N $ with $  \liminf_{ n \to \infty } l(n) = \infty $ such that $  e_{ k( l(n) ) } \in E $, $ n \in \N $, converges in $ ( E, \mathcal{E} ) $ to $  e_0 $.
\item \label{item:infty:exist} For every function $ k \colon \N \to \N $ with $  \liminf_{ n \to \infty } k(n) = \infty $ there exists a function $  l \colon \N \to \N $ such that $  e_{ k( l(n) ) } \in E $, $ n \in \N $, converges in $ ( E, \mathcal{E} ) $ to $  e_0 $. 
\item \label{item:inc:inc} For every strictly increasing function $ k \colon \N \to \N $ there exists a strictly increasing function $  l \colon \N \to \N $ such that $  e_{ k( l(n) ) } \in E $, $ n \in \N $, converges in $ ( E, \mathcal{E} ) $ to $  e_0 $. 
\item \label{item:inc:infty} For every strictly increasing function $ k \colon \N \to \N $ there exists a function $  l \colon \N \to \N $ with $  \liminf_{ n \to \infty }$ $l(n) = \infty $ such that $  e_{ k( l(n) ) } \in E $, $ n \in \N $, converges in $ ( E, \mathcal{E} ) $ to $  e_0 $. 
\item \label{item:inc:exist} For every strictly increasing function $ k \colon \N \to \N $ there exists a function $  l \colon \N \to \N $ such that $  e_{ k( l(n) ) } \in E $, $ n \in \N $, converges in $ ( E, \mathcal{E} ) $ to $  e_0 $. 
\end{enumerate} 
\end{lemma}
\begin{proof}[Proof of Lemma~\ref{lem:subsequence}]
Consider the notation in Subsection~\ref{sec:notation}.
It is clear that $ (\eqref{item:con}  \Rightarrow \eqref{item:infty:inc}) $, $ (\eqref{item:infty:inc} \Rightarrow \eqref{item:infty:infty}) $, $ (\eqref{item:infty:infty} \Rightarrow \eqref{item:infty:exist}) $,   $ (\eqref{item:inc:inc} \Rightarrow \eqref{item:inc:infty}) $,  and $ (\eqref{item:inc:infty} \Rightarrow \eqref{item:inc:exist}) $. The fact that $ (\eqref{item:infty:exist} \Rightarrow \eqref{item:inc:exist})$ and the fact that $(\eqref{item:infty:inc} \Rightarrow \eqref{item:inc:inc}) $ hence ensure that it is sufficient to prove that $  ( \eqref{item:inc:exist} \Rightarrow \eqref{item:con}  ) $ in order  to complete the proof of Lemma~\ref{lem:subsequence}. We show $  ( \eqref{item:inc:exist} \Rightarrow \eqref{item:con}  ) $ by a contradiction and for this we assume $   \big(     \eqref{item:inc:exist} \wedge     (\neg \eqref{item:con} )   \big) $ in the following.
Observe that $  (\neg \eqref{item:con}) $ assures that there exists a set $ A \in \mathcal{E} $ with $ e_0 \in A $ and $  \#_{ \{ n \in \N \colon e_n \notin A \} } = \infty $.
This implies that there exists a strictly increasing function $ k \colon \N \to \N $ such that for all $ n \in \N $ it holds that 
\begin{equation} \label{eq:subsequence_contra}  
e_{ k(n) } \notin A  . 
\end{equation} 
Next note that $ \eqref{item:inc:exist} $ ensures that there exists a function $ l \colon \N \to \N $ such that $  e_{ k( l(n) ) } \in E $, $ n \in \N $, converges in $ ( E, \mathcal{E} ) $ to $ e_0 $. This proves that there exists a natural number $ N \in \N $ such that for all $ n \in \{ N , N + 1, \dots \} $ it holds that $  e_{ k( l(n) ) } \in A $. In particular, we obtain that $  e_{ k( l(N) ) } \in A $. This contradicts to \eqref{eq:subsequence_contra}. The proof of Lemma~\ref{lem:subsequence} is thus completed.
\end{proof}

\subsection{Pathwise convergence}

\begin{prop}\label{prop:main_det}
Consider the notation in Subsection~\ref{sec:notation},
let $ ( H, \left< \cdot , \cdot \right>_H, \left\| \cdot \right\|_H ) $ be a separable $\mathbb{R}$-Hilbert space,
let $ \mathbb{H} \subseteq H$ 
be a nonempty orthonormal basis of $ H $,
let $ \kappa \in [0, \infty)$,
let
$\lambda \colon \mathbb{H} \to \R $
be a function which satisfies
$\inf_{b \in \mathbb{H}} \lambda_b > -\kappa$,
let $  A \colon D(A) \subseteq H \to H $ be the linear operator which satisfies
$ D(A) = \{ v \in H \colon \sum_{b \in \mathbb{H}} | \lambda_b \langle b , v \rangle_H |^2 < \infty \} $
and
$ \forall \, v \in D(A) \colon A v = \sum_{b \in \mathbb{H}} - \lambda_b \langle b , v \rangle_H b$,
let $ ( H_r, \left< \cdot , \cdot \right>_{ H_r }, \left\| \cdot \right\|_{ H_r } ) $, $ r \in \R $,
be a family of interpolation spaces associated to $ \kappa- A $,
let $ \alpha \in [ 0, 1)$,   $\varrho \in [0, 1-\alpha)$, $ T , \chi \in (0,\infty)$,   $F \in \mathcal{C}(H_{\varrho}, H_{-\alpha})$, $ X, O \in \mathcal{C}(  [0, T], H_{ \varrho}) $,
let
$ ( \mathbb{H}_n )_{ n \in \N } \colon \N \to \mathcal{P}_0( \mathbb{H} ) $,
$ ( P_n )_{ n \in \N } \colon \N \to L( H ) $,
$(h_n)_{n \in \N} \colon \N \to (0, T] $,
and
$ \Y^n, \mathcal{O}^n \colon [0,T] \to H_{\varrho} $, $ n \in \N $,
be functions,
and assume for all  $ v \in H $, $ n \in \N $, $ t \in [0,T] $, $r \in (0,\infty)$ that
$P_n(v) = \sum_{ b \in \mathbb{H}_n } \langle b, v \rangle_H b$,
$\sup (\{\|F(x)-F(y)\|_{H_{-\alpha}}/\|x-y\|_{H_{\varrho}} \colon x, y \in H_{\varrho}, x\neq y, \max\{\|x\|_{H_{\varrho}},\|y\|_{H_{\varrho}}\} \leq r \}) < \infty$,
$ \liminf_{ m \to \infty } \inf( \{\lambda_b \colon b \in \mathbb{H} \backslash \mathbb{H}_m \} \cup \{\infty\}  ) = \infty $,
$ \limsup_{m\to \infty}  ( h_m + \sup_{s \in [0, T]} \| O_s - \mathcal{O}_s^m\|_{H_{\varrho}} ) = 0$,
$ X_t = \int_0^t e^{ ( t - s ) A } \, F( X_s ) \, ds + O_t$,
and 
\begin{align}
\Y_t^n  = \int_0^t P_n \,  e^{  ( t - s ) A } \, \one_{[0,|h_n|^{-\chi}]} \big(  \big\| \Y_{ \lf s \rf_{h_n} }^n \big\|_{ H_{\varrho} } +  \big\|  \mathcal{O}_{ \lf s \rf_{h_n} }^n \big\|_{ H_{\varrho} }  \big) \, F \big(  \Y_{ \lf s \rf_{ h_n } }^n \big) \, ds  +\mathcal{O}_t^n.
\end{align}
Then  it holds that $ \limsup_{n \to \infty} \sup_{t \in [0, T]} \| X_t- \Y_t^n \|_{H_\varrho}=0 $.
\end{prop}

\begin{proof}[Proof of Proposition~\ref{prop:main_det}]
Note that
the fact that
$ \forall \, r \in [ 0, 1 ], \, t \in ( 0, \infty ) \colon
\| (\kappa-A)^{r} \, e^{ t (A-\kappa) } \|_{L(H)} \leq t^{-r} $
(cf., e.g., Lemma~11.36 in Renardy \& Rogers~\cite{RenardyRogers1993})
proves for all $ n \in \N $, $ t \in [0, T] $, $ \varepsilon \in [0, 1 - \varrho - \alpha) $  that 
\begin{align} 
\label{eq:sup_s_alpha_finite} 
\begin{split} 
\sup_{ s \in (0, T] } \big( s^{ ( \varrho + \alpha ) } \| e^{ s A } \|_{ L( H_{ - \alpha }, H_{ \varrho } ) } \big) & = \sup_{ s \in (0, T] } \big( s^{ ( \varrho + \alpha ) } \| (\kappa-A)^{ ( \varrho + \alpha ) } \, e^{ s A } \|_{ L(H) } \big) \\ 
& \leq e^{T \kappa} \cdot \sup_{ s \in (0, T] } \big( s^{ ( \varrho + \alpha ) } \| (\kappa-A)^{ ( \varrho + \alpha ) } \, e^{ s (A-\kappa) } \|_{ L(H) } \big) \leq e^{T \kappa} < \infty 
\end{split} 
\end{align} 
and 
\begin{align} 
\begin{split} 
& \int_0^t \| ( \Id_{H_{\varrho}} - P_n ) \, e^{ s A } \|_{ L( H_{ - \alpha } , H_{ \varrho } ) } \, ds \leq \int_0^t \| \Id_{H_{\varrho+\varepsilon}} - P_n|_{H_{\varrho+\varepsilon}} \|_{ L( H_{ \varrho + \varepsilon } , H_{ \varrho } ) } \, \| e^{ s A } \|_{ L( H_{ - \alpha }, H_{ \varrho + \varepsilon } ) } \, ds \\ 
& = \| ( \kappa- A )^{ - \varepsilon } ( \Id_H - P_n ) \|_{ L( H ) } \int_0^t \| (\kappa - A )^{ ( \varrho + \varepsilon + \alpha ) } \, e^{ s A } \|_{ L(H) } \, ds \\ 
& = \| ( \kappa- A )^{ -1 } ( \Id_H - P_n ) \|_{ L( H ) }^{\varepsilon} \int_0^t e^{s \kappa} \,  \| (\kappa - A )^{( \varrho + \varepsilon + \alpha ) }\, e^{ s (A-\kappa) } \|_{ L(H) } \, ds \\ 
& \leq e^{T \kappa} \, \| (\kappa- A )^{ - 1 } ( \Id_H - P_n ) \|_{ L( H ) }^{ \varepsilon }  \int_0^t s^{ - ( \varrho + \varepsilon + \alpha ) } \, ds \\ 
& \leq \frac{  e^{T \kappa} \,  \| ( \kappa - A )^{ - 1 } ( \Id_H - P_n ) \|_{ L( H ) }^{ \varepsilon }  \, T^{ ( 1 - \varrho - \varepsilon - \alpha ) } }{ ( 1 - \varrho -  \varepsilon -\alpha) } . 
\end{split} 
\end{align} 
This and the assumption that $ \liminf_{ n \to \infty} \inf( \{\lambda_b \colon b \in \mathbb{H} \backslash \mathbb{H}_n \} \cup \{\infty\} ) = \infty $ imply that
\begin{align}
\label{eq:sup_Id_H_0} 
\limsup_{ n \to \infty } \left( \int_0^T \| (\Id_{H_{\varrho}} - P_n ) e^{ s A } \|_{ L( H_{ - \alpha } , H_{ \varrho } ) } \, ds \right) = 0 . 
\end{align} 
Combining the fact that $\limsup_{ n \to \infty} h_n =0 $, the fact that $ \limsup_{n\to \infty} \sup_{t \in [0, T]} \| O_t - \mathcal{O}_t^n\|_{H_{\varrho}} = 0$, \eqref{eq:sup_s_alpha_finite}, \eqref{eq:sup_Id_H_0}, and the fact that $ \limsup_{ n \to \infty } \| P_n |_{H_{\varrho}} \|_{ L( H_{ \varrho } ) } = 1 < \infty $ with, e.g., Proposition~3.3 in Hutzenthaler et al.~\cite{Salimova2016} (with $ V = H_{ \varrho } $, $ W = H_{ - \alpha } $, $ T = T $, $ \chi = \chi $, $\Upsilon = \sup_{t \in (0,T]}(t^{(\varrho+\alpha)}\|e^{tA}\|_{L(H_{-\alpha},H_{\varrho})})$, $ \alpha = \varrho + \alpha $, $ (P_n)_{n \in \N} = ( H_{ \varrho } \ni v \mapsto P_n( v ) \in H_{ \varrho } )_{n \in \N} $, $ (h_n)_{n \in \N} = (h_n)_{n \in \N} $, $ F = F $, $\Psi= ([0,\infty] \ni r \mapsto \sup (\{ 0 \} \cup \{\|F(x)-F(y)\|_{H_{-\alpha}}/\|x-y\|_{H_{\varrho}} \colon x, y \in H_{\varrho}, x\neq y, \max\{\|x\|_{H_{\varrho}},\|y\|_{H_{\varrho}}\} \leq r \}) \in [0,\infty])$, $X=X$, $O=O$, $(\mathcal{X}_n)_{n \in \N} = (\mathcal{X}_n)_{n \in \N}$, $(\mathcal{O}_n)_{n \in \N} = (\mathcal{O}_n)_{n \in \N}$, $ S = ((0,T] \ni t \mapsto ( H_{ - \alpha } \ni v \mapsto e^{ t A } v \in H_{ \varrho } ) \in L(H_{-\alpha},H_{\varrho}) ) $  in the notation of Proposition~3.3 in Hutzenthaler et al.~\cite{Salimova2016}) shows that $ \limsup_{n \to \infty} \sup_{t \in [0, T]} \| X_t- \Y_t^n \|_{H_\varrho}=0 $. The proof of Proposition~\ref{prop:main_det} is thus completed. 
\end{proof}

\subsection{Pathwise a priori bounds}

\begin{prop}\label{prop:main_det:2}
Consider the notation in Subsection~\ref{sec:notation},
let $ ( H, \left< \cdot , \cdot \right>_H, \left\| \cdot \right\|_H ) $ be a separable $\mathbb{R}$-Hilbert space,
let $ \mathbb{H} \subseteq H$ be a nonempty orthonormal basis of $ H $,
let $\eta , \theta , \vartheta,  \kappa \in [0, \infty)$,
let $\lambda \colon \mathbb{H} \to \R $
be a function which satisfies
$\inf_{b \in \mathbb{H}} \lambda_b > -\min\{\eta, \kappa\}$,
let $ A \colon D(A) \subseteq H \to H $ be the linear operator which satisfies
$ D(A) = \{ v \in H \colon \sum_{b \in \mathbb{H}} | \lambda_b \langle b , v \rangle_H |^2 < \infty \} $
and
$ \forall \, v \in D(A) \colon  A v = \sum_{b \in \mathbb{H}} - \lambda_b \langle b , v \rangle_H b$,
let $ ( H_r, \left< \cdot , \cdot \right>_{ H_r }, \left\| \cdot \right\|_{ H_r } ) $, $ r \in \R $,
be a family of interpolation spaces associated to $ \kappa- A $,
let $\varphi \in [0,1)$, $ \alpha \in [ 0, \nicefrac{1}{2}]$,  $\rho \in [0, 1- \alpha)$, $\varrho \in (\rho, 1-\alpha)$, $ \beta, T \in (0,\infty)$, $ \chi \in ( 0, \min\{  \nicefrac{ ( 1 - \alpha - \rho ) }{( 1 +  \vartheta ) }, \nicefrac{( \varrho - \rho ) }{( 1 + \nicefrac{\vartheta}{2} ) } \}] $, $p \in [2, \infty)$, $F \in \mathcal{C}(H_{\varrho}, H_{-\alpha})$,
let
$ \phi, \Phi \colon H_1 \to  [0,\infty) $,
$ ( \mathbb{H}_n )_{ n \in \N } \colon \N \to \mathcal{P}_0( \mathbb{H} ) $,
$ ( P_n )_{ n \in \N } \colon \N \to L( H_{ -1 } ) $,
and
$(h_n)_{n \in \N} \colon \N \to (0, T] $
be functions,
let
$ \Y^n, \mathbb{O}^n \colon [0,T] \to H_{\varrho} $, $ n \in \N $,
and
$ \mathcal{O}^n \colon [0,T] \to P_n(H) $, $ n \in \N $,
be functions,
and assume
for all $u \in H$, $ n \in \N $, $ v, w \in P_n(H) $, $ t \in [0,T] $ that
$P_n(u) = \sum_{ b \in \mathbb{H}_n } \langle b, u \rangle_H b$,
$ \left< v, P_n F( v + w ) \right>_H \leq \phi( w ) \| v \|^2_H + \varphi \| (\eta-A)^{\nicefrac{1}{2}} v \|^2_{ H } + \Phi( w ) $,
$\| F(v)\|_{H_{-\alpha}}^2 \leq \theta \max\{ 1, \| v\|_{ H_{\varrho} }^{2 + \vartheta} \} $,
$ \|(\eta-A)^{\nicefrac{-1}{2}} [F(v) - F(w) ]\|_{H}^2\leq \theta \max\{ 1, \|v\|_{ H_{\varrho} }^{\vartheta} \} \|v-w\|_{ H_{\rho} }^2 + \theta \, \|v-w\|^{2 + \vartheta}_{ H_{\rho}}$,
$\eta \mathcal{O}^n \in \mathcal{C}([0, T], P_n( H ) )$,
$\mathbb{O}_t^n = \mathcal{O}_t^n - \int_0^t e^{(t-s)(A-\eta)} \, \eta \mathcal{O}_s^n \, ds$,
and 
\begin{align}
\Y_t^n  = \int_0^t P_n \,  e^{  ( t - s ) A } \, \one_{[0,|h_n|^{-\chi}]} \big(  \big\| \Y_{ \lf s \rf_{h_n} }^n \big\|_{ H_{\varrho} } +  \big\|  \mathcal{O}_{ \lf s \rf_{h_n} }^n \big\|_{ H_{\varrho} }  \big) \, F \big(  \Y_{ \lf s \rf_{ h_n } }^n \big) \, ds  +\mathcal{O}_t^n.
\end{align}
Then
\begin{enumerate}[(i)]
\item \label{item:Pn} it holds for all $n \in \N$ that  $ \Y^n( [0,T] ) \cup \mathbb{O}^n( [0,T] ) \subseteq P_n( H )$ and
\item  \label{item:bound} it holds for all $t \in [0,T]$, $n \in \N$ with $ h_n \leq 1 $  that $\eta \mathbb{O}^n \in \mathcal{C}([0, T], H_{\varrho})$ and 
\begin{align*}
\| \Y^n_t \|_H^p 
&\leq 2^{p-1}\| \mathbb{O}^n_t \|_H^p  +  2^{p-1}  t^{\nicefrac{p}{2}-1} \bigg[1 +  \tfrac{ \theta  e^{ \kappa(2+ \vartheta)} [ 1+ (\kappa + \sqrt{\eta} +\sqrt{\eta}|\kappa-\eta|e^{ \eta} ) \|(\kappa-A)^{ \rho - \varrho } \|_{L(H)} + \sqrt{\theta} +  \sqrt{\eta} ]^{2+\vartheta}  }{(1- \varphi)( 1-\alpha - \rho)^{2+\vartheta}}\bigg]^{\nicefrac{p}{2}}     \\
& \quad \cdot \int_0^t e^{ \int_s^t p \, \phi( \mathbb{O}_{\lf u \rf_{h_n} }^n )  +p\eta (1+\beta) \, du}  \Big[  2 \Phi\big(\mathbb{O}_{ \lf s \rf_{h_n} }^n \big)  + \tfrac{\eta}{2 \beta} \|\mathbb{O}_s^n\|_H^2 \numberthis \\
& \quad + \left|\max \!\big\{ 1, \smallint\nolimits_{0}^T   \|   \sqrt{\eta} \mathcal{O}_u^n \|_{H_{\varrho}} \, du \big\}\right|^{2+ 2\vartheta} \max \!\big\{ 1,  \smallint\nolimits_{0}^T   \|  \sqrt{ \eta} \mathcal{O}_u^n \|_{H_{\varrho}}^2 \, du \big\}  \Big]^{\nicefrac{p}{2}} ds.
\end{align*}
\end{enumerate}

\end{prop}

\begin{proof}[Proof of Proposition~\ref{prop:main_det:2}]
Without loss of generality we assume for all $n \in \N$ that $\mathbb{H}_n \neq \emptyset$.
Throughout this proof let $ \tilde{\phi}_n \colon P_n( H ) \to [0,\infty) $, $ n \in \N $, and $\tilde{\Phi}_n \colon P_n( H ) \to [0,\infty) $, $ n \in \N $, be the functions which satisfy for all  $ n \in \N $, $ v \in P_n( H ) $ that $\tilde{\phi}_n( v ) = 2 \cdot \phi( v )$ and  $\tilde{\Phi}_n( v ) = 2 \cdot \Phi( v )$.
Note that for all $n \in \N$ it holds that $P_n(H)$ is a finite-dimensional $\R$-vector space and $ \Y^n( [0,T] ) \cup \mathbb{O}^n( [0,T] ) \subseteq P_n( H )$. Corollary~\ref{cor:a_priori} (with 
$ H = P_n( H ) $,
$ \mathbb{H} = \mathbb{H}_n $,
$\beta=\beta$, $T =T$,  $\eta=\eta$, $ \theta= \theta$, $\vartheta = \vartheta$, $\kappa= \kappa$,
$ \varphi= \varphi$, $\alpha = \alpha$,  $\rho= \rho$, $ \varrho = \varrho$,
$\chi = \chi $, 
$ h = h_n$,
$ F = ( P_n( H ) \ni v \mapsto P_n F( v ) \in P_n(H) ) \in \mathcal{C}( P_n( H ), P_n( H ) ) $,
$ A = ( P_n( H ) \ni v \mapsto A v \in P_n( H ) ) \in L( P_n( H ) ) $, 
$ Y = ( [0,T] \ni t\mapsto \Y^n_t \in P_n( H ) ) $, 
$ O = \mathcal{O}^n $, 
$ \mathbb{O} = ( [0,T] \ni t \mapsto \mathbb{O}^n_t \in P_n( H ) ) $,
$ \phi = \tilde{\phi}_n $, $ \Phi = \tilde{\Phi}_n $
for  $ n \in \{ m \in \N \colon h_m \leq 1 \}$  in the notation of Corollary~\ref{cor:a_priori}) hence proves that for all  $t \in [0, T]$, $ n \in \N  $ with $h_n \leq 1$ it holds that $\eta \mathbb{O}^n \in \mathcal{C}([0, T], H_{\varrho})$ and
\begin{align*}
\label{eq:thm_application}
& \| \Y^n_t - \mathbb{O}^n_t \|_H^2  \leq \int_0^t e^{ \int_s^t \tilde{\phi}_n( \mathbb{O}_{\lf u \rf_{h_n} }^n )  +2\eta (1+\beta) \, du} \Big[  \tilde{\Phi}_n\big(\mathbb{O}_{ \lf s \rf_{h_n} }^n \big)  + \tfrac{\eta}{2 \beta} \|\mathbb{O}_s^n\|_H^2    \\
& +  \tfrac{ \theta  e^{\kappa(2+\vartheta)} [ 1+ (\kappa + \sqrt{\eta} +\sqrt{\eta}|\kappa-\eta|e^{ \eta} ) \|(\kappa-A)^{ \rho - \varrho } \|_{L(H)} + \sqrt{\theta} +  \sqrt{\eta} ]^{2+\vartheta} \left|\max \{ 1, \smallint\nolimits_{0}^T   \|   \sqrt{\eta} \mathcal{O}_u^n \|_{H_{\varrho}} \, du \}\right|^{2+ \vartheta}}{(1- \varphi )( 1-\alpha - \rho)^{2+\vartheta}}   \\
& \cdot  \max \!\big\{  |h_n|^{2(\varrho - \rho -\chi)},|h_n|^{ 2(1-\alpha - \rho -( 1 + \nicefrac{\vartheta}{2} ) \chi)} , h_n \smallint\nolimits_{0}^T   \|  \sqrt{ \eta} \mathcal{O}_u^n \|_{H_{\varrho}}^2 \, du \big\}   \left|\max \!\big\{ |h_n|^{-\chi}, \smallint\nolimits_{0}^T   \|  \sqrt{ \eta} \mathcal{O}_u^n \|_{H_{\varrho}} \, du \big\}\right|^{\vartheta}\Big] ds \\
& = \int_0^t e^{ \int_s^t 2\phi( \mathbb{O}_{\lf u \rf_{h_n} }^n )  +2\eta (1+\beta) \, du} \Big[  2 \Phi\big(\mathbb{O}_{ \lf s \rf_{h_n} }^n \big)  + \tfrac{\eta}{2 \beta} \|\mathbb{O}_s^n\|_H^2  \numberthis  \\
& +  \tfrac{\theta  e^{ \kappa(2+ \vartheta)} [ 1+ (\kappa + \sqrt{\eta} +\sqrt{\eta}|\kappa-\eta|e^{ \eta} ) \|(\kappa-A)^{ \rho - \varrho } \|_{L(H)} + \sqrt{\theta} +  \sqrt{\eta} ]^{2+\vartheta} \left|\max \{ 1, \smallint\nolimits_{0}^T   \|   \sqrt{\eta} \mathcal{O}_u^n \|_{H_{\varrho}} \, du \}\right|^{2+ \vartheta} }{(1- \varphi)( 1-\alpha - \rho)^{2+\vartheta}}  \\
& \cdot \max \!\big\{ |h_n|^{2(\varrho - \rho -\chi)},|h_n|^{ 2(1-\alpha - \rho -( 1 + \nicefrac{\vartheta}{2} ) \chi)} , h_n \smallint\nolimits_{0}^T   \|  \sqrt{ \eta} \mathcal{O}_u^n \|_{H_{\varrho}}^2 \, du \big\}  \left|\max \!\big\{ |h_n|^{-\chi}, \smallint\nolimits_{0}^T   \|  \sqrt{ \eta} \mathcal{O}_u^n \|_{H_{\varrho}} \, du \big\}\right|^{ \vartheta}\Big] ds.
\end{align*}
Moreover, the assumption that $ \chi \in ( 0, \min\{  \nicefrac{ ( 1 - \alpha - \rho ) }{( 1 +  \vartheta ) }, \nicefrac{( \varrho - \rho ) }{( 1 + \nicefrac{\vartheta}{2} ) } \}] $ implies that
\begin{equation}
1 - \alpha - \rho -( 1 +  \vartheta ) \chi\geq 0, \qquad 1-\vartheta \chi \geq 0, \qquad \text{and} \qquad   \varrho - \rho -( 1 + \nicefrac{\vartheta}{2} ) \chi\geq 0.
\end{equation}
This ensures for all $ n \in \N $  with $ h_n \leq 1 $ that
\begin{equation}
\begin{split}
& \max \!\big\{ |h_n|^{2(\varrho - \rho -\chi)},|h_n|^{ 2(1-\alpha - \rho -( 1 + \nicefrac{\vartheta}{2} ) \chi)} , h_n \smallint\nolimits_{0}^T   \|  \sqrt{ \eta} \mathcal{O}_u^n \|_{H_{\varrho}}^2 \, du \big\}  \left|\max \!\big\{ |h_n|^{-\chi},  \smallint\nolimits_{0}^T   \|  \sqrt{ \eta} \mathcal{O}_u^n \|_{H_{\varrho}} \, du \big\}\right|^{ \vartheta}\\
& = \max \!\big\{ |h_n|^{2(\varrho - \rho -(1+\nicefrac{\vartheta}{2})\chi)},|h_n|^{ 2(1-\alpha - \rho -( 1 + \vartheta ) \chi)} , |h_n|^{1-\vartheta \chi} \smallint\nolimits_{0}^T   \|  \sqrt{ \eta} \mathcal{O}_u^n \|_{H_{\varrho}}^2 \, du \big\} \\
& \quad \cdot \left|\max \!\big\{ 1, |h_n|^{\chi} \smallint\nolimits_{0}^T   \|  \sqrt{ \eta} \mathcal{O}_u^n \|_{H_{\varrho}} \, du \big\}\right|^{ \vartheta} \\
& \leq \max \!\big\{ 1,  \smallint\nolimits_{0}^T   \|  \sqrt{ \eta} \mathcal{O}_u^n \|_{H_{\varrho}}^2 \, du \big\}  \left|\max \!\big\{ 1, \smallint\nolimits_{0}^T   \|  \sqrt{ \eta} \mathcal{O}_u^n \|_{H_{\varrho}} \, du \big\}\right|^{ \vartheta}.
\end{split}
\end{equation}
The fact that $\forall \, x,y \in H \colon \|x+y\|_H^2 \leq 2\|x\|_H^2 + 2\|y\|_H^2$ and \eqref{eq:thm_application} hence show that for all $t \in [0,T]$, $n \in \N$ with $ h_n \leq 1 $ it holds that
\begin{align}
\begin{split}
\| \Y^n_t \|_H^2  
&\leq 2 \, \| \mathbb{O}^n_t \|_H^2   +  2\int_0^t e^{ \int_s^t 2\phi( \mathbb{O}_{\lf u \rf_{h_n} }^n )  +2\eta (1+\beta) \, du} \Big[  2 \Phi\big(\mathbb{O}_{ \lf s \rf_{h_n} }^n \big)  + \tfrac{\eta}{2 \beta} \|\mathbb{O}_s^n\|_H^2    \\
&\quad +  \tfrac{ \theta  e^{ \kappa(2+ \vartheta)} [ 1+ (\kappa + \sqrt{\eta} +\sqrt{\eta}|\kappa-\eta|e^{ \eta} ) \|(\kappa-A)^{ \rho - \varrho } \|_{L(H)} + \sqrt{\theta} +  \sqrt{\eta}]^{2+\vartheta} \left|\max \{ 1, \smallint\nolimits_{0}^T   \|   \sqrt{\eta} \mathcal{O}_u^n \|_{H_{\varrho}} \, du \}\right|^{2+ \vartheta} }{(1- \varphi)( 1-\alpha - \rho)^{2+\vartheta}} \\
& \quad \cdot  \max \!\big\{ 1,  \smallint\nolimits_{0}^T   \|  \sqrt{ \eta} \mathcal{O}_u^n \|_{H_{\varrho}}^2 \, du \big\}  \left|\max \!\big\{ 1, \smallint\nolimits_{0}^T   \|  \sqrt{ \eta} \mathcal{O}_u^n \|_{H_{\varrho}} \, du \big\}\right|^{ \vartheta}\Big] ds\\
& = 2 \, \| \mathbb{O}^n_t \|_H^2   +  2\int_0^t e^{ \int_s^t 2\phi( \mathbb{O}_{\lf u \rf_{h_n} }^n )  +2\eta (1+\beta) \, du} \Big[  2 \Phi\big(\mathbb{O}_{ \lf s \rf_{h_n} }^n \big)  + \tfrac{\eta}{2 \beta} \|\mathbb{O}_s^n\|_H^2    \\
& \quad +  \tfrac{ \theta  e^{ \kappa(2+ \vartheta)} [ 1+ (\kappa + \sqrt{\eta} +\sqrt{\eta}|\kappa-\eta|e^{ \eta} ) \|(\kappa-A)^{ \rho - \varrho } \|_{L(H)} + \sqrt{\theta} +  \sqrt{\eta} ]^{2+\vartheta}  }{(1- \varphi)( 1-\alpha - \rho)^{2+\vartheta}} \\
& \quad  \cdot \left|\max \!\big\{ 1, \smallint\nolimits_{0}^T   \|   \sqrt{\eta} \mathcal{O}_u^n \|_{H_{\varrho}} \, du \big\}\right|^{2+ 2\vartheta} \max \!\big\{ 1,  \smallint\nolimits_{0}^T   \|  \sqrt{ \eta} \mathcal{O}_u^n \|_{H_{\varrho}}^2 \, du \big\}  \Big] ds \\
& \leq 2 \, \| \mathbb{O}^n_t \|_H^2   +  2 \bigg[1 +  \tfrac{ \theta  e^{ \kappa(2+ \vartheta)} [ 1+ (\kappa + \sqrt{\eta} +\sqrt{\eta}|\kappa-\eta|e^{ \eta} ) \|(\kappa-A)^{ \rho - \varrho } \|_{L(H)} + \sqrt{\theta} +  \sqrt{\eta}]^{2+\vartheta}  }{(1- \varphi)( 1-\alpha - \rho)^{2+\vartheta}}\bigg]     \\
& \quad \cdot \int_0^t e^{ \int_s^t 2\phi( \mathbb{O}_{\lf u \rf_{h_n} }^n )  +2\eta (1+\beta) \, du} \Big[  2 \Phi\big(\mathbb{O}_{ \lf s \rf_{h_n} }^n \big)  + \tfrac{\eta}{2 \beta} \|\mathbb{O}_s^n\|_H^2 \\
& \quad + \left|\max \!\big\{ 1, \smallint\nolimits_{0}^T   \|   \sqrt{\eta} \mathcal{O}_u^n \|_{H_{\varrho}} \, du \big\}\right|^{2+ 2\vartheta} \max \!\big\{ 1,  \smallint\nolimits_{0}^T   \|  \sqrt{ \eta} \mathcal{O}_u^n \|_{H_{\varrho}}^2 \, du \big\}  \Big] ds.
\end{split}
\end{align}
This, the assumption that $p \in [2, \infty)$, the fact that $\forall \, a, b \in \R \colon |a+b|^{\nicefrac{p}{2}} \leq 2^{\nicefrac{p}{2}-1} |a|^{\nicefrac{p}{2}} + 2^{\nicefrac{p}{2}-1} |b|^{\nicefrac{p}{2}} $, and H\"older's inequality prove  for all $t \in [0,T]$, $n \in \N$ with $ h_n \leq 1 $   that
\begin{align}
\label{eq:thm:bound}
\begin{split}
&\| \Y^n_t \|_H^p 
\leq 2^{p-1}\| \mathbb{O}^n_t \|_H^p   +  2^{p-1} \bigg[1 +  \tfrac{ \theta  e^{ \kappa(2+ \vartheta)} [ 1+ (\kappa + \sqrt{\eta} +\sqrt{\eta}|\kappa-\eta|e^{ \eta} ) \|(\kappa-A)^{ \rho - \varrho } \|_{L(H)} + \sqrt{\theta} +  \sqrt{\eta} ]^{2+\vartheta}  }{(1- \varphi)( 1-\alpha - \rho)^{2+\vartheta}}\bigg]^{\nicefrac{p}{2}}     \\
&  \quad \cdot \bigg[\int_0^t e^{ \int_s^t 2\phi( \mathbb{O}_{\lf u \rf_{h_n} }^n )  +2\eta (1+\beta) \, du} \Big[  2 \Phi\big(\mathbb{O}_{ \lf s \rf_{h_n} }^n \big)  + \tfrac{\eta}{2 \beta} \|\mathbb{O}_s^n\|_H^2 \\
& \quad  + \left|\max \!\big\{ 1, \smallint\nolimits_{0}^T   \|   \sqrt{\eta} \mathcal{O}_u^n \|_{H_{\varrho}} \, du \big\}\right|^{2+ 2\vartheta} \max \!\big\{ 1,  \smallint\nolimits_{0}^T   \|  \sqrt{ \eta} \mathcal{O}_u^n \|_{H_{\varrho}}^2 \, du \big\}  \Big] ds \bigg]^{\nicefrac{p}{2}}\\
&\leq 2^{p-1}\| \mathbb{O}^n_t \|_H^p   +  2^{p-1}  t^{\nicefrac{p}{2}-1} \bigg[1 +  \tfrac{ \theta  e^{ \kappa(2+ \vartheta)} [ 1+ (\kappa + \sqrt{\eta} +\sqrt{\eta}|\kappa-\eta|e^{ \eta} ) \|(\kappa-A)^{ \rho - \varrho } \|_{L(H)} + \sqrt{\theta} +  \sqrt{\eta} ]^{2+\vartheta}  }{(1- \varphi)( 1-\alpha - \rho)^{2+\vartheta}}\bigg]^{\nicefrac{p}{2}}     \\
& \quad \cdot \int_0^t e^{ \int_s^t p \, \phi( \mathbb{O}_{\lf u \rf_{h_n} }^n )  +p\eta (1+\beta) \, du} \Big[  2 \Phi\big(\mathbb{O}_{ \lf s \rf_{h_n} }^n \big)  + \tfrac{\eta}{2 \beta} \|\mathbb{O}_s^n\|_H^2 \\
& \quad + \left|\max \!\big\{ 1, \smallint\nolimits_{0}^T   \|   \sqrt{\eta} \mathcal{O}_u^n \|_{H_{\varrho}} \, du \big\}\right|^{2+ 2\vartheta} \max \!\big\{ 1,  \smallint\nolimits_{0}^T   \|  \sqrt{ \eta} \mathcal{O}_u^n \|_{H_{\varrho}}^2 \, du \big\}  \Big]^{\nicefrac{p}{2}} ds .
\end{split}
\end{align}
The proof of Proposition~\ref{prop:main_det:2} is thus completed.
\end{proof}

\subsection{Strong convergence}
\label{subsec:thm:strong}

\begin{theorem}
\label{thm:strong}
Consider the notation in Subsection~\ref{sec:notation},
let $ ( H, \left< \cdot , \cdot \right>_H, \left\| \cdot \right\|_H ) $ be a separable $\mathbb{R}$-Hilbert space,
let $ ( \Omega, \F, \P ) $ be a probability space,
let $ \mathbb{H} \subseteq H$ be a nonempty orthonormal basis of $ H $,
let $\eta, \theta,  \kappa \in [0, \infty)$,
let $\lambda \colon \mathbb{H} \to \R $
be a function which satisfies $\inf_{b \in \mathbb{H}} \lambda_b > -\min\{\eta, \kappa\}$,
let $ A \colon D(A) \subseteq H \to H $
be the linear operator which satisfies
$ D(A) = \{ v \in H \colon \sum_{b \in \mathbb{H}} | \lambda_b \langle b , v \rangle_H |^2 < \infty \} $
and
$ \forall \, v \in D(A) \colon A v = \sum_{b \in \mathbb{H}} - \lambda_b \langle b , v \rangle_H b$,
let $ ( H_r, \left< \cdot , \cdot \right>_{ H_r }, \left\| \cdot \right\|_{ H_r } ) $, $ r \in \R $,
be a family of interpolation spaces associated to $ \kappa- A $,
let $\varphi \in [0,1)$, $ \alpha \in [ 0, \nicefrac{1}{2}]$,  $\rho \in [0, 1- \alpha)$, $\varrho \in (\rho, 1-\alpha)$, $  \vartheta,T \in (0,\infty)$, $ \chi \in ( 0, \min\{  \nicefrac{ ( 1 - \alpha - \rho ) }{( 1 + 2 \vartheta ) }, \nicefrac{( \varrho - \rho ) }{( 1 + \vartheta) } \}] $, $p \in [2, \infty)$,  $F \in \mathcal{C}(H_{\varrho}, H_{-\alpha} )$,
let
$ \phi, \Phi \colon H_1 \to  [0,\infty) $,
$ ( \mathbb{H}_n )_{ n \in \N } \colon \N \to \mathcal{P}_0( \mathbb{H} ) $,
$ ( P_n )_{ n \in \N } \colon \N \to L( H_{ -1 } ) $,
$(h_n)_{n \in \N} \colon \N \to (0, T] $,
and
$ \mathbb{X}^n, \mathbb{O}^n \colon [0,T]\times \Omega \to H_{\varrho} $, $ n \in \N $,
be functions,
let $ \Y^n \colon [0,T] \times \Omega \to H_{ \varrho } $, $ n \in \N $, be stochastic processes,
let $ \mathcal{O}^n \colon [0,T] \times \Omega \to P_n( H ) $, $ n \in \N $,
and
$ X, O \colon [0, T] \times \Omega \to H_{ \varrho} $
be stochastic processes with continuous sample paths,
and assume
for all $u \in H$, $n \in \N$, $v, w \in P_n(H)$, $ t \in [0,T] $ that
$ P_n(u) = \sum_{ b \in \mathbb{H}_n } \langle b, u \rangle_H b $,
$ \left< v, P_n F( v + w ) \right>_H \leq \phi( w ) \| v \|^2_H + \varphi \| (\eta-A)^{\nicefrac{1}{2}} v \|^2_{ H} + \Phi( w ) $,
$ \left\| F(v) - F(w) \right\|_{ H_{ - \alpha } } \leq \theta \, ( 1 + \| v \|_{ H_{ \rho } }^{ \vartheta } + \|w\|_{H_{\rho}}^{\vartheta}) \, \|v-w\|_{H_{\rho}} $,
$ \liminf_{ m \to \infty } \inf( \{\lambda_b \colon b \in \mathbb{H} \backslash \mathbb{H}_m \} \cup \{\infty\} ) = \infty $,
$\limsup_{m \to \infty}  \E \big[\!\min\!\big\{ 1, \sup_{s \in [0,T]}\| O_s- \mathcal{O}_s^{m} \|_{ H_{ \varrho } } \big\} + h_m \big] = 0$,
$\mathbb{O}_t^n = \mathcal{O}_t^n - \int_0^t e^{(t-s)(A-\eta)} \, \eta \mathcal{O}_s^n \, ds$,
$ \mathbb{X}_t^n  =  \smallint\nolimits_0^t P_n \,  e^{  ( t - s ) A } \, \one_{ \{ \| \Y_{ \lf s \rf_{h_n} }^n \|_{ H_{ \varrho } } +  \| \mathcal{O}_{ \lf s \rf_{h_n} }^n \|_{ H_{ \varrho } }  \leq | h_n|^{ - \chi } \}} \, F \big(  \Y_{ \lf s \rf_{ h_n } }^n \big) \, ds  +\mathcal{O}_t^n$,
$ \limsup_{ m \to \infty} \sup_{ s \in [0,T]} \E \bigl[ \| \mathbb{O}_s^m \|_H^p \bigr] < \infty$,
$ \P\big( X_t = \int_0^t e^{ ( t - s ) A } \, F( X_s ) \, ds + O_t \big)=\P (\mathbb{X}_t^n  = \Y_t^n )=1$,
and
\begin{equation}
\label{eq:integrability}
 \limsup_{ m \to \infty} \E\biggl[ \int_0^T e^{ \int_s^T p\, \phi( \mathbb{O}_{\lf u \rf_{h_m} }^m )  \, du} \max\bigl\{ 1, |\Phi(\mathbb{O}_{ \lf s \rf_{h_m} }^m )|^{\nicefrac{p}{2}}, \|\mathbb{O}_s^m\|_H^p, \smallint\nolimits_{0}^T \| \mathcal{O}_u^m \|_{H_{\varrho}}^{2p+ 2p\vartheta} \, du \bigr\} \, ds\biggr] < \infty.
\end{equation}
Then
\begin{enumerate}[(i)]
\item \label{item:exp} it holds that $\limsup_{n \to \infty}  \E \big[\!\min\!\big\{ 1, \sup_{t \in [0,T]}\| X_t- \mathbb{X}_t^{n} \|_{ H_{ \varrho } } \big\} \big] = 0$, 
\item \label{item:pnorm} it holds that $ \limsup_{ n \to \infty } \sup_{ t \in [0,T] } \E\big[ \| X_t \|^p_H + \| \Y^n_t \|_H^p \big] < \infty $, and
\item \label{item:strong} it holds for all $ q \in (0, p) $ that $ \limsup_{ n \to \infty } \sup_{ t \in [0,T] } \E\big[ \| X_t - \Y_t^n \|_H^q \big] = 0 $.
\end{enumerate}
\end{theorem}

\begin{proof}[Proof of Theorem~\ref{thm:strong}]
Throughout this proof let $ \tilde{\Omega} $ be the set given by
\begin{equation}
\begin{split}
\tilde{\Omega} & =
\Bigl\{\omega \in \Omega \colon  \Bigl(\forall \, m \in \N, s \in  [0, T] \colon \mathbb{X}_{\lf s \rf_{h_m}}^m(\omega)= \Y_{\lf s \rf_{h_m}}^m(\omega) \Bigr)\Bigr\}
\\ & \quad
\cap
\biggl\{  \omega \in \Omega \colon \biggl( \forall \, s \in [0,T] \colon  X_s(\omega) = \int_0^s e^{ ( s - u ) A } \, F( X_u(\omega) ) \, du + O_s(\omega)\biggr)\biggr\},
\end{split}
\end{equation}
let $ \tilde{\Y}^n \colon [0,T] \times \Omega \to H_{\varrho} $, $ n \in \N $,
be the sequence which satisfies for all $n \in \N$, $t \in [0, T]$  that
\begin{align}
\label{eq:tilde_X}
\tilde{\Y}_t^n  =  \int_0^t P_n \,  e^{  ( t - s ) A } \, \one_{ \{ \|\tilde{\Y}_{ \lf s \rf_{h_n} }^n \|_{ H_{ \varrho } } +  \| \mathcal{O}_{ \lf s \rf_{h_n} }^n \|_{ H_{ \varrho } }  \leq | h_n|^{ - \chi } \}} \, F \big(  \tilde{\Y}_{ \lf s \rf_{ h_n } }^n \big) \, ds  +\mathcal{O}_t^n,
\end{align} 
and let $ \tilde{\theta} \in [0, \infty)$, $\tilde{\vartheta} \in (0,\infty)$ 
be the real numbers given by $ \tilde{\vartheta} = 2 \vartheta $ 
and 
\begin{multline}
\tilde{\theta} = 
\max\{1, \|(\eta-A)^{-1} (\kappa-A)\|_{L(H)}\}  \max\!\bigg\{ \big(8 \theta^2 + 2 \, \| F(0) \|_{ H_{ - \alpha } }^2 \big) \max\!\bigg\{ 1, \sup_{
	u \in H_{ \varrho } \backslash \{ 0 \} }\tfrac{ \| u \|_{ H_{ \rho } }^{ 2 + 2 \vartheta } 
}{\| u \|_{ H_{ \varrho } }^{ 2 + 2 \vartheta } } \bigg\}  ,\\
3 \, \theta^2  \bigg[\sup_{ u \in H_{ - \alpha } \backslash\{ 0 \} }\tfrac{ \| u \|_{ H_{  \nicefrac{-1}{2} } }^2 }{\| u \|_{ H_{ - \alpha } }^2 } \bigg] \bigg[1 + \sup_{ u \in H_{ \varrho } \backslash \{ 0 \} }\tfrac{
	\| u \|_{ H_{ \rho } }^{ 2 \vartheta } }{\| u \|_{H_{\varrho}}^{2 \vartheta}} \bigg]
\big(1+2^{\max\{2\vartheta-1, 0\}}\big)  \! \bigg\}.
\end{multline}
Observe that for all $n \in \N$ it holds that $X$, $O$, $\mathbb{X}^n$, $\mathcal{O}^n$, $\mathbb{O}^n$,  $\tilde{\Y}^n$ are stochastic processes with continuous sample paths.
In addition, note that the assumption that $X, O \colon [0, T] \times \Omega \to H_{\varrho}$ are stochastic processes with continuous sample paths
and
the assumption that
$\forall \, t \in [0, T] \colon  \P\big(X_t = \int_0^t e^{ ( t - s ) A } \, F( X_s ) \, ds + O_t\big)=1$
show that
\begin{equation}
\biggl\{  \omega \in \Omega \colon \biggl(\forall \, t \in [0,T] \colon  X_t(\omega) = \int_0^t e^{ ( t - s ) A } \, F( X_s(\omega) ) \, ds + O_t(\omega) \biggr)\biggr\} \in \F
\end{equation}
and
\begin{equation}
\P\biggl(\forall \, t \in [0,T] \colon  X_t = \int_0^t e^{ ( t - s ) A } \, F( X_s ) \, ds + O_t\biggr)=1.
\end{equation}
This and the assumption that
$\forall \, n\in \N, \, t \in [0, T] \colon \P(\mathbb{X}_t^n = \Y_t^n)=1$
yield that $\tilde{\Omega} \in \F$ and $\P(\tilde{\Omega})=1$.
In the next step let $k \colon \N \to \N $ be a strictly increasing function.
The fact that
$\limsup_{n \to \infty}  \E \bigl[\min\bigl\{ 1, \sup_{t \in [0,T]}\| O_t- \mathcal{O}_t^{n} \|_{ H_{ \varrho } } \bigr\} \bigr] = 0$
assures that
\begin{equation}
\limsup_{n \to \infty}  \E \biggl[ \min \biggl\{ 1, \sup_{t \in [0,T]}\| O_t- \mathcal{O}_t^{k(n)} \|_{ H_{ \varrho } } \biggr\} \biggr] = 0.
\end{equation}
This implies that there exists a strictly increasing function $l \colon \N \to \N $ such that
\begin{equation}
\sum_{n=1}^{\infty}  \E \biggl[ \min \biggl\{ 1, \sup_{t \in [0,T]}\| O_t- \mathcal{O}_t^{k(l(n))} \|_{ H_{ \varrho } } \biggr\} \biggr] < \infty.
\end{equation}
Lemma~\ref{lem:fast_convergence}
(with
$(\Omega, \mathcal{F}, \P)= (\Omega, \mathcal{F}, \P)$,
$E=\R$,
$d= ( \R \times \R \ni (x,y) \mapsto | x - y | \in [0,\infty) )$,
$(X_n)_{n \in \N} = ( \Omega \ni \omega \mapsto
\sup_{t \in [0,T]}\| O_t( \omega ) - \mathcal{O}_t^{k(l(n))}( \omega ) \|_{ H_{ \varrho } }  \in \R )_{n \in \N}$,
$X_0 = ( \Omega \ni \omega \mapsto 0 \in \R ) $
in the notation of Lemma~\ref{lem:fast_convergence})
hence proves that
\begin{equation}
\P\biggl(\limsup_{n  \to \infty} \sup_{t \in [0,T]}\| O_t- \mathcal{O}_t^{k(l(n))} \|_{ H_{ \varrho } } = 0\biggr)=1.
\end{equation}
Combining this,
\eqref{eq:tilde_X},
the fact that $\forall \, \omega \in \tilde{\Omega}, \, t \in [0,T] \colon  X_t(\omega) = \int_0^t e^{ ( t - s ) A } \, F( X_s(\omega) ) \, ds + O_t(\omega)$,
and
the fact that $\P(\tilde{ \Omega })=1$
with Proposition~\ref{prop:main_det}
(with $ H = H $,
$ \mathbb{H}= \mathbb{H}$,
$\kappa=\kappa$,
$A=A$,
$\alpha = \alpha$,
$ \varrho = \varrho$,
$T =T$,
$\chi = \chi $,
$ F = F $,
$X=([0,T] \ni t \mapsto X_t(\omega) \in H_{\varrho})$,
$O= ([0,T] \ni t \mapsto O_t(\omega) \in H_{\varrho})$,
$ ( \mathbb{H}_n )_{ n \in \N } =  ( \mathbb{H}_{k(l(n))} )_{ n \in \N }$,
$ ( P_n )_{ n \in \N } = ( H \ni v \mapsto P_{k(l(n))}(v) \in H )_{ n \in \N }$,
$ (h_n)_{n \in \N} = (h_{k(l(n))})_{n \in \N}$,
$ (\Y^n)_{n \in \N} = ( [0,T] \ni t \mapsto \tilde{\Y}^{k(l(n))}_t( \omega ) \in H_{\varrho} )_{n \in \N} $,
$(\mathcal{O}^n)_{n \in \N} = ( [0,T] \ni t \mapsto \mathcal{O}^{k(l(n))}_t( \omega ) \in H_{\varrho} )_{n \in \N} $
for  $\omega \in \{ \varpi \in \Omega \colon \limsup_{n \to \infty} \sup_{t \in [0,T]}\| O_t(\varpi)- \mathcal{O}_t^{k(l(n))}(\varpi) \|_{ H_{ \varrho } } = 0\} \cap \tilde{\Omega}$ in the notation of Proposition~\ref{prop:main_det})
establishes that
\begin{equation}
\P\biggl(\limsup_{n \to \infty} \sup_{t \in [0, T]} \| X_t- \tilde{\Y}_t^{k(l(n))} \|_{H_\varrho}=0\biggr)=1.
\end{equation}
The fact that
$\forall \, \omega \in \tilde{\Omega}, \, t\in [0,T], \, n \in \N \colon \tilde{\Y}^n_t(\omega)  =   \mathbb{X}_t^n(\omega)$
and
the fact that
$\P(\tilde{\Omega})=1$
hence show that
\begin{equation}
\P\biggl(\limsup_{n \to \infty} \sup_{t \in [0, T]} \| X_t- \mathbb{X}_t^{k(l(n))} \|_{H_\varrho}=0\biggr)=1.
\end{equation}
This and Lebesgue's theorem of dominated convergence imply that
\begin{equation}
\limsup_{n \to \infty}  \E \biggl[\min\biggl\{ 1, \sup_{t \in [0,T]}\| X_t- \mathbb{X}_t^{k(l(n))} \|_{ H_{ \varrho } } \biggr\} \biggr] = 0.
\end{equation}
As $k \colon \N \to \N $ was an arbitrary strictly increasing function,
Lemma~\ref{lem:subsequence} proves that
\begin{equation}
\limsup_{n \to \infty}
\E \biggl[\min\biggl\{ 1, \sup_{t \in [0,T]}\| X_t- \mathbb{X}_t^{n} \|_{ H_{ \varrho } } \biggr\} \biggr] = 0.
\end{equation}
This concludes \eqref{item:exp}.
Next
note that, e.g., Lemma~2.4 in Hutzenthaler et al.~\cite{Salimova2016} (with $V=H_{\varrho}$, $\mathcal{V}=H_{\rho}$, $W=H_{-\alpha}$, $\mathcal{W}= H_{\nicefrac{-1}{2}}$, $\epsilon=\theta$, $\theta=(\max\{1, \|(\eta-A)^{-1} (\kappa-A)\|_{L(H)}\})^{-1} \tilde{\theta}$, $\varepsilon= \vartheta$, $\vartheta = \tilde{\vartheta}$, $F=F$ in the notation of Lemma~2.4 in Hutzenthaler et al.~\cite{Salimova2016}) ensures that for all $v, w \in H_{\varrho}$ it holds that 
\begin{align}
\label{eq:F_condition_proof_1}
\begin{split}
\| (\eta -A)^{\nicefrac{-1}{2}} [F(v) - F(w)] \|_{H}^2  & \leq  \|(\eta-A)^{\nicefrac{-1}{2}} (\kappa-A)^{\nicefrac{1}{2}} \|_{L(H)}^2 \|F(v)- F(w)\|_{H_{\nicefrac{-1}{2}}}^2 \\
&\leq \max\{1, \|(\eta-A)^{-1} (\kappa-A)\|_{L(H)}\}  \|F(v)- F(w)\|_{H_{\nicefrac{-1}{2}}}^2 \\
&\leq \tilde{\theta} \max\{1, \|v\|_{H_{\varrho}}^{\tilde{\vartheta}}\}\|v-w\|_{H_{\rho}}^2 + \tilde{\theta} \, \|v-w\|_{H_{\rho}}^{2+\tilde{\vartheta}}
\end{split}
\end{align}
and 
\begin{equation}
\label{eq:F_condition_proof_2}
\|F(v)\|_{H_{-\alpha}}^2 \leq \tilde{\theta} \max\{ 1, \|v\|_{H_{\varrho}}^{2 + \tilde{\vartheta}} \}.
\end{equation}
Furthermore, observe that the assumption that $\chi \in ( 0, \nicefrac{ ( 1 - \alpha - \rho ) }{( 1 + 2 \vartheta) }]\cap (0, \nicefrac{( \varrho - \rho ) }{( 1 + \vartheta) }]$ assures that $\chi \in ( 0, \nicefrac{ ( 1 - \alpha - \rho ) }{( 1 + \tilde{ \vartheta}) }]\cap (0, \nicefrac{( \varrho - \rho ) }{( 1 + \nicefrac{\tilde{\vartheta}}{2} ) }]$. Combining this, \eqref{eq:F_condition_proof_1}, and \eqref{eq:F_condition_proof_2} with  Proposition~\ref{prop:main_det:2} (with
 $ H = H $,
 $ \mathbb{H}= \mathbb{H}$,
$\eta=\eta$,
$ \theta= \tilde{\theta}$, $\vartheta = \tilde{\vartheta}$, $\kappa=\kappa$, $A=A$,
$ \varphi= \varphi$, $\alpha = \alpha$,
$\rho= \rho$, $ \varrho = \varrho$,
$\beta=1$,
$T =T$,
$\chi = \chi $, $p=p$, 
$ F = F $, $ \phi = \phi$, $ \Phi = \Phi $,
$ ( \mathbb{H}_n )_{ n \in \N } =  ( \mathbb{H}_n )_{ n \in \N }$,
$ ( P_n )_{ n \in \N } = ( P_n )_{ n \in \N }$,
$ (h_n)_{n \in \N} = (h_n)_{n \in \N}$,
$ (\Y^n)_{n \in \N} = ( [0,T] \ni t \mapsto \tilde{\Y}^{n}_t( \omega ) \in H_{\varrho} )_{n \in \N} $, 
$ (\mathbb{O}^n)_{n \in \N} = ( [0,T] \ni t \mapsto \mathbb{O}^{n}_t( \omega ) \in H_{\varrho} )_{n \in \N} $, 
$(\mathcal{O}^n)_{n \in \N} = ( [0,T] \ni t \mapsto \mathcal{O}^{n}_t( \omega ) \in P_n(H) )_{n \in \N} $         for  $\omega \in \Omega$ in the notation of Proposition~\ref{prop:main_det:2}) shows that for all $t \in [0, T]$, $ n \in \N $ with $ h_n \leq 1 $ it holds that
\begin{align}
\label{eq:cor_1bound}
\begin{split}
\| \tilde{\Y}^n_t \|_H^p  
&\leq 2^{p-1}\| \mathbb{O}^n_t \|_H^p   +  2^{p-1} \, t^{\nicefrac{p}{2}-1} \bigg[ 1+ \tfrac{ \tilde{\theta}  e^{ \kappa(2+ \tilde{\vartheta})} [ 1+ (\kappa + \sqrt{\eta} +\sqrt{\eta}|\kappa-\eta|e^{ \eta} ) \|(\kappa-A)^{ \rho - \varrho } \|_{L(H)} + |\tilde{\theta}|^{\nicefrac{1}{2}} +  \sqrt{\eta} ]^{2+\tilde{\vartheta}} }{(1- \varphi)( 1-\alpha - \rho)^{2+\tilde{\vartheta}}} \bigg]^{\nicefrac{p}{2}}    \\
& \quad \cdot\int_0^t e^{ \int_s^t p \, \phi( \mathbb{O}_{\lf u \rf_{h_n} }^n )  +2p\eta  \, du} \Big[  2 \Phi\big(\mathbb{O}_{ \lf s \rf_{h_n} }^n \big)  + \tfrac{\eta}{2} \|\mathbb{O}_s^n\|_H^2  \\
& \quad  + \left|\max \!\big\{ 1, \smallint\nolimits_{0}^T   \|  \sqrt{\eta} \mathcal{O}_u^n \|_{H_{\varrho}} \, du \big\}\right|^{2+ 2\tilde{\vartheta}} \max \!\big\{ 1,  \smallint\nolimits_{0}^T   \|  \sqrt{ \eta} \mathcal{O}_u^n \|_{H_{\varrho}}^2 \, du \big\}  \Big]^{\nicefrac{p}{2}} ds.
\end{split}
\end{align}
Next observe that H\"older's inequality implies  for all $n \in \N$  that 
\begin{align}
\begin{split}
& \left|\max \!\big\{ 1, \smallint\nolimits_{0}^T   \|  \sqrt{\eta} \mathcal{O}_u^n \|_{H_{\varrho}} \, du \big\}\right|^{2+ 2\tilde{\vartheta}} \max \!\big\{ 1,  \smallint\nolimits_{0}^T   \|  \sqrt{ \eta} \mathcal{O}_u^n \|_{H_{\varrho}}^2 \, du \big\}  \\
& =  \left|\max \!\big\{ 1, \smallint\nolimits_{0}^T   \|  \sqrt{\eta} \mathcal{O}_u^n \|_{H_{\varrho}} \, du \big\}\right|^{2+ 4\vartheta} \max \!\big\{ 1,  \smallint\nolimits_{0}^T   \|  \sqrt{ \eta} \mathcal{O}_u^n \|_{H_{\varrho}}^2 \, du \big\}  \\
& \leq |\max\{1, \eta\}|^{2+2\vartheta} \left|\max \!\big\{ 1, \smallint\nolimits_{0}^T   \|  \mathcal{O}_u^n \|_{H_{\varrho}} \, du \big\}\right|^{2+ 4\vartheta} \max \!\big\{ 1,  \smallint\nolimits_{0}^T   \|   \mathcal{O}_u^n \|_{H_{\varrho}}^2 \, du \big\} \\
&\leq |\max\{1, \eta\}|^{2+2\vartheta} \left|\max \!\big\{ 1, T \smallint\nolimits_{0}^T   \|   \mathcal{O}_u^n \|_{H_{\varrho}}^2 \, du \big\}\right|^{1+ 2\vartheta} \max \!\big\{ 1,  \smallint\nolimits_{0}^T   \|   \mathcal{O}_u^n \|_{H_{\varrho}}^2 \, du \big\} \\
&\leq |\max\{1, \eta\}|^{2+2\vartheta} \, |\max\{1, T\}|^{1+2\vartheta} \left|\max \!\big\{ 1, \smallint\nolimits_{0}^T   \|   \mathcal{O}_u^n \|_{H_{\varrho}}^2 \, du \big\}\right|^{2+ 2\vartheta} \\ 
&\leq |\max\{1, \eta\}|^{2+2\vartheta} \, |\max\{1, T\}|^{1+2\vartheta} \left|\max \!\big\{ 1, T^{p+p\vartheta-1} \smallint\nolimits_{0}^T   \|   \mathcal{O}_u^n \|_{H_{\varrho}}^{2p+2p\vartheta} \, du \big\}\right|^{\nicefrac{2}{p}} \\ 
&\leq |\max\{1, \eta, T\}|^{5+6\vartheta}  \left|\max \!\big\{ 1,  \smallint\nolimits_{0}^T   \|   \mathcal{O}_u^n \|_{H_{\varrho}}^{2p+2p\vartheta} \, du \big\}\right|^{\nicefrac{2}{p}}.
\end{split}
\end{align}
The fact that $\forall \, a, b, c \in \R \colon |a+b+c|^{\nicefrac{p}{2}} \leq 3^{\nicefrac{p}{2}-1} (|a|^{\nicefrac{p}{2}} + |b|^{\nicefrac{p}{2}} + |c|^{\nicefrac{p}{2}})$ together with \eqref{eq:cor_1bound} hence establishes  for all $t \in [0,T]$, $ n \in \N $ with $ h_n \leq 1 $  that 
\begin{align}
\begin{split}
&\| \tilde{\Y}^n_t \|_H^p 
\leq 2^{p-1}\| \mathbb{O}^n_t \|_H^p   +  2^{p-1} \, t^{\nicefrac{p}{2}-1} \bigg[ 1+ \tfrac{ \tilde{\theta}  e^{ \kappa(2+ 2\vartheta)} [ 1+ (\kappa + \sqrt{\eta} +\sqrt{\eta}|\kappa-\eta|e^{ \eta} ) \|(\kappa-A)^{ \rho - \varrho } \|_{L(H)} + |\tilde{\theta}|^{\nicefrac{1}{2}} +  \sqrt{\eta} ]^{2+2\vartheta} }{(1- \varphi)( 1-\alpha - \rho)^{2+2\vartheta}} \bigg]^{\nicefrac{p}{2}}    \\
 & \quad \cdot\int_0^t e^{ \int_s^t p \, \phi( \mathbb{O}_{\lf u \rf_{h_n} }^n ) +2  p\eta  \, du} \bigg[  2 \Phi\big(\mathbb{O}_{ \lf s \rf_{h_n} }^n \big)  + \tfrac{\eta}{2 } \|\mathbb{O}_s^n\|_H^2\\
& \quad  + |\max\{1, \eta, T\}|^{5+6\vartheta}  \left|\max \!\big\{ 1,  \smallint\nolimits_{0}^T   \|   \mathcal{O}_u^n \|_{H_{\varrho}}^{2p+2p\vartheta} \, du \big\}\right|^{\nicefrac{2}{p}}  \bigg]^{\nicefrac{p}{2}} ds\\
&\leq 2^{p-1}\|\mathbb{O}^n_t \|_H^p   +  2^{p-1} \, |3t|^{\nicefrac{p}{2}-1} \bigg[ 1+ \tfrac{ \tilde{\theta}  e^{ \kappa(2+ 2\vartheta)} [ 1+ (\kappa + \sqrt{\eta} +\sqrt{\eta}|\kappa-\eta|e^{ \eta} ) \|(\kappa-A)^{ \rho - \varrho } \|_{L(H)} + |\tilde{\theta}|^{\nicefrac{1}{2}} +  \sqrt{\eta} ]^{2+2\vartheta} }{(1- \varphi)( 1-\alpha - \rho)^{2+2\vartheta}} \bigg]^{\nicefrac{p}{2}}    \\
& \quad \cdot\int_0^t e^{ \int_s^t p\, \phi( \mathbb{O}_{\lf u \rf_{h_n} }^n ) +2 p\eta \, du} \Big[  2^{\nicefrac{p}{2}} \big| \Phi\big(\mathbb{O}_{ \lf s \rf_{h_n} }^n \big) \big|^{\nicefrac{p}{2}} + \big|\tfrac{\eta}{2 } \big|^{\nicefrac{p}{2}} \|\mathbb{O}_s^n\|_H^p \\
& \quad + |\max\{1, \eta, T\}|^{3p+3p\vartheta} \max \!\big\{ 1, \smallint\nolimits_{0}^T   \|   \mathcal{O}_u^n \|_{H_{\varrho}}^{2p+ 2p\vartheta} \, du \big\}  \Big] ds.
\end{split}
\end{align}
This, the fact that $\forall \, \omega \in \tilde{\Omega}, \, t\in [0,T], \, n \in \N \colon  \tilde{\Y}^n_t(\omega)  =   \mathbb{X}_t^n(\omega)$, and the fact that $\P(\tilde{\Omega})=1$  yield that for all $t \in [0, T]$,  $n \in \N $
with $ h_n \leq 1 $ it holds that 
\begin{align}
\begin{split}
&\E\big[ \| \mathbb{X}^n_t \|_H^p\big] = \E\big[ \| \tilde{\mathcal{X}}^n_t \|_H^p\big]  \\
& \leq   2^{p-1} \, \E\big[\| \mathbb{O}^n_t \|_H^p \big]  +  2^{p-1} \, |3t|^{\nicefrac{p}{2}-1} \bigg[ 1+ \tfrac{ \tilde{\theta}  e^{ \kappa(2+ 2\vartheta)} [  1+ (\kappa + \sqrt{\eta} +\sqrt{\eta}|\kappa-\eta|e^{ \eta} ) \|(\kappa-A)^{ \rho - \varrho } \|_{L(H)} + |\tilde{\theta}|^{\nicefrac{1}{2}} +  \sqrt{\eta} ]^{2+2\vartheta} }{(1- \varphi)( 1-\alpha - \rho)^{2+2\vartheta}} \bigg]^{\nicefrac{p}{2}}    \\
& \quad \cdot \E \bigg[\int_0^T e^{ \int_s^T p \, \phi( \mathbb{O}_{\lf u \rf_{h_n} }^n )  +2p\eta \, du} \Big[  2^{\nicefrac{p}{2}} \big| \Phi\big(\mathbb{O}_{ \lf s \rf_{h_n} }^n \big) \big|^{\nicefrac{p}{2}} + \big|\tfrac{\eta}{2 } \big|^{\nicefrac{p}{2}} \|\mathbb{O}_s^n\|_H^p \\
& \quad + |\max\{1, \eta, T\}|^{3p+3p\vartheta} \max \!\big\{ 1, \smallint\nolimits_{0}^T   \|   \mathcal{O}_u^n \|_{H_{\varrho}}^{2p+ 2p\vartheta} \, du \big\}  \Big] ds \bigg].
\end{split}
\end{align}
The assumption that $\forall \, t \in [0,T], \, n \in \N \colon \P (\mathbb{X}_t^n= \Y_t^n)=1$, the fact that $ \limsup_{ n \to \infty} h_n = 0$,  the assumption that $ \limsup_{ n \to \infty} \sup_{ t \in [0,T]} \E[ \| \mathbb{O}_t^n \|_H^p] < \infty$, and \eqref{eq:integrability} hence ensure that
\begin{equation}
\label{eq:a_priori_X} 
\limsup_{ n \to \infty } \sup_{ t \in [0,T] } \E\big[ \| \Y^n_t \|_H^p \big]= \limsup_{ n \to \infty } \sup_{ t \in [0,T] } \E\big[ \| \mathbb{X}^n_t \|_H^p \big] < \infty.
\end{equation}
Next note that \eqref{item:exp} and the fact $H_{\varrho} \subseteq H$ continuously
imply that
\begin{equation}
\limsup_{n \to \infty} \sup_{t \in [0,T]} \E \big[\!\min\!\big\{ 1,\| X_t- \mathbb{X}_t^{n} \|_{ H } \big\} \big] = 0.
\end{equation}
The assumption that
$\forall \, t \in [0,T], \, n \in \N \colon \P (\mathbb{X}_t^n= \Y_t^n)=1$
hence assures that
\begin{equation}
\limsup_{n \to \infty} \sup_{t \in [0,T]} \E \big[\!\min\!\big\{ 1,\| X_t- \Y_t^{n} \|_{ H } \big\} \big] = 0.
\end{equation}
Combining this with, e.g., Lemma~4.2 in Hutzenthaler et al.~\cite{Salimova2016}
proves for all $\varepsilon \in (0, \infty)$ that
\begin{equation}
\limsup_{n \to \infty} \sup_{t \in [0, T]} \P \bigl(\|X_t - \Y_t^n\|_H \geq \varepsilon \bigr)=0.
\end{equation}
E.g., Proposition~4.5 in Hutzenthaler et al.~\cite{Salimova2016} together with \eqref{eq:a_priori_X}
hence shows for all $q \in (0, p)$ that
$ \sup_{ t \in [0,T] } \E\big[ \| X_t \|^p_H \big] < \infty $
and
\begin{equation}
\limsup_{ n \to \infty } \sup_{ t \in [0,T] } \E\big[ \| X_t - \Y_t^n \|_H^q \big] = 0.
\end{equation}
Combining this with \eqref{eq:a_priori_X} establishes \eqref{item:pnorm} and \eqref{item:strong}.
The proof of Theorem~\ref{thm:strong} is thus completed.
\end{proof}

\section{Fernique's theorem}
\label{sec:fernique}

In this section we present a number of elementary and well-known results
and, in particular, Fernique's theorem, which is crucial for the derivations in Section~\ref{sec:abstract} below. 

\subsection{Uniqueness theorem for measures}

Proposition~\ref{thm:measure_uniqueness} can, e.g., be found 
as Lemma~1.42 in Klenke~\cite{Klenke2008}.

\begin{prop}[Uniqueness theorem for measures]
	\label{thm:measure_uniqueness}
	Consider the notation in Subsection~\ref{sec:notation},
	let 
	$ \Omega $ be a set,
	let
	$
	\mathcal{E} \subseteq \mathcal{P}( \Omega )
	$
	be an $ \cap $-stable subset of $ \mathcal{P}( \Omega ) $ (see, e.g., \cite[Definition~2.1]{JentzenPusnik2016}), and let
	$ 
	\mu_1, \mu_2 \colon \sigma_{ \Omega }( \mathcal{E} ) \to [0,\infty]
	$
	be measures which satisfy that there exists a sequence
	$ \Omega_n \in \{ A \in \mathcal{E} \colon \mu_1( A ) < \infty \} $, $ n \in \N $,
	such that
	$
	\cup_{ n \in \N } \Omega_n = \Omega
	$
	and 
	$ \mu_1|_{ \mathcal{E} } = \mu_2|_{ \mathcal{E} } $.
	Then it holds that $ \mu_1 = \mu_2 $.
\end{prop}
\begin{proof}[Proof of Proposition~\ref{thm:measure_uniqueness}]
	Throughout this proof let 
	$ \mathcal{S} \subseteq \mathcal{E} $
	be the set given by
	$ \mathcal{S} = \{ A \in \mathcal{E} \colon \mu_1(A) < \infty \} $
	and let
	$E_n \in \sigma_{\Omega}(\mathcal{E}) $, $ n \in \N_0 $,
	and
	$ \mathcal{D}_E \in \mathcal{P}( \sigma_{ \Omega }( \mathcal{E} ) ) $, 
	$ E \in \mathcal{S} $,
	be the sets which satisfy
	for all $n \in \N $, $ E \in \mathcal{S} $ that
	$E_0= \emptyset$, $E_n = \bigcup_{i=1}^n \Omega_n$, 
	and
	$ \mathcal{D}_E = \{ A \in \sigma_{ \Omega }( \mathcal{E} ) \colon
	\mu_1( A \cap E ) = \mu_2 ( A \cap E ) \} $.
	First, note that for all $ E \in  \mathcal{S} $  
	it holds that
	$ \Omega \in \mathcal{D}_E $.
	Next observe that for all 
	$ E \in  \mathcal{S} $,
	$ A, B \in  \mathcal{D}_E  $ with $ B \subseteq A $ it holds that
	\begin{equation}
	\begin{split}
	\mu_1 ( ( A \backslash B ) \cap E ) 
	&
	= \mu_1 ( A \cap E ) - \mu_1 ( B \cap E )
	\\
	&
	= \mu_2 ( A \cap E ) - \mu_2 ( B \cap E )
	= \mu_2 ( ( A \backslash B ) \cap E ) .
	\end{split}
	\end{equation}
	This shows for all  
	$ E \in  \mathcal{S} $,
	$ A, B \in \mathcal{D}_E $ 
	with 
	$ B \subseteq A $ 
	that 
	$ A \backslash B \in \mathcal{D}_E $.
	Moreover, note that for all 
	sets $ E \in \mathcal{S} $
	and all sequences
	$ A_n \in \mathcal{D}_E $, $ n \in \N $,
	with $ \forall \, i \in \N, \, j \in \N \backslash \{ i \} \colon A_i \cap A_j = \emptyset $
	it holds that
	\begin{equation}
	\begin{split}
	\mu_1 \big( \big( \cup_{n\in \N} A_n \big) \cap E \big)
	=
	\sum_{n=1}^\infty \mu_1 ( A_n \cap E )
	=
	\sum_{n=1}^\infty \mu_2 ( A_n \cap E ) 
	=
	\mu_2 \big ( \big( \cup_{n\in \N} A_n \big) \cap E \big ).
	\end{split}
	\end{equation}
	This proves for all
	sets $ E \in  \mathcal{S} $
	and all sequences
	$ A_n \in \mathcal{D}_E $, $ n \in \N $,
	with $ \forall \, i \in \N, \, j \in \N \backslash \{ i \} \colon A_i \cap A_j = \emptyset $ 
	that $ \cup_{n\in \N} A_n \in \mathcal{D}_E $.
	Therefore, we have established that for all
	$ E \in \mathcal{S} $ it holds that
	$ \mathcal{D}_E $ is a Dynkin system on $\Omega$ (see, e.g., Definition~2.2 in Jentzen \& Pu\v{s}nik~\cite{JentzenPusnik2016}).
	The fact that
	 $ \forall \,  E \in \mathcal{S} \colon \mathcal{E} \subseteq \mathcal{D}_E $
	hence shows for all
	$ E \in \mathcal{S} $ 
	that
	$ \delta_{ \Omega } ( \mathcal{E} ) \subseteq \mathcal{D}_E $ (see, e.g., Definition~2.3 in~\cite{JentzenPusnik2016}).
	Combining this, the assumption that $ \mathcal{E} $ is $ \cap $-stable,
	and, e.g., Theorem~2.5 in~\cite{JentzenPusnik2016} ensures  
	for all $ E \in \mathcal{S} $ that
	$ \sigma_{ \Omega } ( \mathcal{E} ) = \delta_{ \Omega } ( \mathcal{E} )
	\subseteq \mathcal{D}_E \subseteq \sigma_{ \Omega } ( \mathcal{E} ) .$
	This assures for all $ E \in \mathcal{S} $ that
	\begin{equation}
	\label{eq:valid_for_finiteM_sets}
	\sigma_{ \Omega } ( \mathcal{E} ) = \mathcal{D}_E 
	.
	\end{equation}
	Next note that for all $n \in \N$, $m \in \N \backslash \{n\}$ it holds that $E_n = \bigcup_{i=1}^n ((\Omega\backslash E_{i-1}) \cap \Omega_i)$ and  $((\Omega\backslash E_{n-1}) \cap \Omega_n)\cap ((\Omega\backslash E_{m-1}) \cap \Omega_m) = \emptyset $.
	Equation \eqref{eq:valid_for_finiteM_sets} hence proves for all $n \in \N$, $ A \in \sigma_{ \Omega } ( \mathcal{E} ) $ that
	\begin{equation}
	\mu_1( A \cap E_n )
	=
	\sum_{i=1}^n \mu_1\big((A \cap(\Omega\backslash E_{i-1})) \cap \Omega_i\big)
	=
	\sum_{i=1}^n \mu_2\big((A \cap(\Omega\backslash E_{i-1})) \cap \Omega_i\big)
	=
	\mu_2 ( A \cap E_n ).
	\end{equation}
	This implies for all
	$ A \in \sigma_{ \Omega } ( \mathcal{E} ) $ that
	\begin{equation}
	\mu_1( A )
	=
	\lim\nolimits_{n\to \infty} 
	\mu_1 ( A \cap E_n )
	=
	\lim\nolimits_{n\to \infty}
	\mu_2 (A \cap E_n )
	=
	\mu_2 ( A ).
	\end{equation}
	The proof of Proposition~\ref{thm:measure_uniqueness} is thus completed.
\end{proof}

\subsection{Borel sigma-algebras on normed vector spaces}

In this subsection we first recall the Hahn-Banach theorem (see, e.g., Werner~\cite[Theorem~III.1.5]{w05a}).

\begin{prop}[Hahn-Banach theorem; Extension of continuous linear functionals]
	\label{thm:Hahn_Banach}
	Let 
	$ \K \in \{ \R, \C \} $,
	let
	$ ( V , \left\| \cdot \right\|_V ) $
	be a normed $ \K $-vector space,
	let $ U \subseteq V $ be a $ \K $-subspace of $ V $,
	and let $ \phi \in U' $.
	Then there exists a function $ \varphi \in V' $
	such that
	\begin{equation}
	\varphi|_U = \phi  
	\qquad 
	\text{and}
	\qquad
	\left\| \varphi \right\|_{ V' }
	=
	\left\| \phi \right\|_{ U' }
	.
	\end{equation}
\end{prop}

The proof of Proposition~\ref{thm:Hahn_Banach}
employs the axiom of choice.
The next result, Corollary~\ref{cor:Hahn_Banach}, is a direct consequence
of the Hahn-Banach theorem.

\begin{cor}[Projections into $ 1 $-dimensional subspaces]
	\label{cor:Hahn_Banach}
	Let 
	$ \K \in \{ \R, \C \} $,
	let
	$ ( V , \left\| \cdot \right\|_V ) $
	be a nontrivial normed $ \K $-vector space,
	and let $ v \in V $.
	Then there exists a function $ \varphi \in V' $
	such that 
	\begin{equation}
	\label{eq:cor_Hahn_Banach}
	\varphi( v )
	=
	\left\| v \right\|_V
	\qquad 
	\text{and}
	\qquad
	\left\| \varphi \right\|_{ V' }
	=
	1
	.
	\end{equation}
\end{cor}

\begin{proof}[Proof	of Corollary~\ref{cor:Hahn_Banach}]
	We show Corollary~\ref{cor:Hahn_Banach} in two steps.
	In the first step we assume that $ v \neq 0 $.
	Let 
	$ U  \subseteq V$
	be the $ \K $-subspace of $ V $
	given by
	$ 
	U 
	= \left\{ \lambda v \in V \colon \lambda \in \K \right\} 
	= \operatorname{span}_V(\{ v \})
	$
	and let
	$ \phi \colon U \to \K $
	be the function which satisfies 
	for all $ \lambda \in \K $ that
	\begin{equation}
	\phi( \lambda v ) = 
	\lambda \left\| v \right\|_V
	.
	\end{equation}
	Proposition~\ref{thm:Hahn_Banach} implies that there exists a function $ \varphi \in V' $ such that
	\begin{equation}
	\varphi|_U = \phi
	\qquad
	\text{and}
	\qquad
	\left\| \varphi \right\|_{ V' }
	=
	\left\| \phi \right\|_{ U' }
	= 
	1
	.
	\end{equation}
	This proves \eqref{eq:cor_Hahn_Banach}
	in the case $ v \neq 0 $.
	In the second step we assume that $ v = 0 $.
	Note that the assumption that $V$ is nontrivial ensures that
	there exists a vector $ u \in V $ such that $ u \neq 0 $.
	The first step hence shows that there exists a function $ \varphi \in V' $
	such that
	\begin{equation}
	\varphi( u ) = \left\| u \right\|_V
	\qquad
	\text{and}
	\qquad 
	\left\| \varphi \right\|_{ V' } = 1
	.
	\end{equation}
	In addition, observe that $ \varphi( v ) = \varphi( 0 ) = 0 = \left\| v \right\|_V $.
	The proof of Corollary~\ref{cor:Hahn_Banach}
	is thus completed.
\end{proof}

The next result, Corollary~\ref{cor:norm},
is an immediate consequence of Corollary~\ref{cor:Hahn_Banach}
above.

\begin{cor}[Norm via the dual space]
	\label{cor:norm}
	Let 
	$ \K \in \{ \R, \C \} $,
	let
	$ 
	( V, \left\| \cdot \right\|_V ) 
	$
	be a nontrivial normed $ \K $-vector space,
	and let $ v \in V $.
	Then 
	\begin{equation}
	\label{eq:cor_norm}
	\left\|  
	v 
	\right\|_V
	=
	\sup_{ 
		\varphi \in V' \backslash \{ 0 \}
	}
	\frac{
		\Re(\varphi( v )) 
	}{
	\left\| \varphi \right\|_{ V' }
}
=
\sup_{ 
	\varphi \in V' \backslash \{ 0 \}
}
\frac{
	\left| \varphi( v ) \right|
}{
\left\| \varphi \right\|_{ V' }
}
.
\end{equation}
\end{cor}

If the normed vector space in
Corollary~\ref{cor:norm} is separable,
then the following result, Corollary~\ref{cor:norm2}, can be obtained.
Corollary~\ref{cor:norm2}
is also an immediate consequence of Corollary~\ref{cor:Hahn_Banach}
above.

\begin{cor}[Norm of a separable normed vector space via the dual space]
	\label{cor:norm2}
	Let 
	$ \K \in \{ \R, \C \} $
	and let
	$ ( V, \left\| \cdot \right\|_V ) $
	be a separable normed $ \K $-vector space.
	Then there exists a sequence
	$  \varphi_n \in  V' $, $n \in \N$,
	which satisfies for all $ v \in V $ that
	\begin{equation}
	\left\|  
	v 
	\right\|_V
	=
	\sup_{ 
		n \in \N
	}
	\Re(\varphi_n( v ))
	=
	\sup_{ 
		n \in \N
	}
	\left| \varphi_n( v ) \right|
	.
	\end{equation}
\end{cor}

\begin{proof}[Proof	of Corollary~\ref{cor:norm2}]
	Without loss of generality we assume that $V$ is nontrivial.
	The assumption that $ ( V, \left\| \cdot \right\|_V ) $
	is separable implies that there exists
	a sequence 
	$ v_n \in V $, $ n \in \N $,
	such that the set
	$
	\left\{ v_n \colon n \in \N \right\}
	$
	is dense in $ V $.
	Corollary~\ref{cor:Hahn_Banach}
	hence shows that there exists
	a sequence
	$ \varphi_n \in V' $, $ n \in \N $,
	which satisfies
	for all $ n \in \N $ that
	\begin{equation}
	\varphi_n( v_n )
	=
	\left\| v_n \right\|_V
	\qquad 
	\text{and}
	\qquad
	\left\| \varphi_n \right\|_{ V' }
	= 1
	.
	\end{equation}
	This ensures  for all $ k \in \N $  that
	\begin{equation}
	\left\| v_k \right\|_V
	=
	\sup_{ 
		n \in \N
	}
	|\varphi_n( v_k )|
	.
	\end{equation}
	Next let 
	$ v \in V $ and $ \varepsilon \in ( 0,\infty) $.
	Note that
	\begin{equation}
	\sup_{ n \in \N } 
	|\varphi_n( v )|
	\leq
	\sup_{ n \in \N }
	\left[
	\left\| \varphi_n \right\|_{ V' }
	\left\| v \right\|_V
	\right]
	=
	\left\| v \right\|_V
	.
	\end{equation}
	It thus remains to prove that
	\begin{equation}
	\left\| v \right\|_V
	\leq
	\varepsilon
	+
	\sup_{ n \in \N } 
	\Re(\varphi_n( v ))
	.
	\end{equation}
	For this observe that
	the fact that
	$
	\{ v_n \in V \colon n \in \N \}
	$
	is dense in $ V $ ensures
	that there exists a natural number $ k \in \N $
	such that $ \left\| v - v_k \right\|_V \leq \frac{ \varepsilon }{ 2 } $.
	This implies that
	\begin{equation}
	\begin{split}
	\left\| v \right\|_V
	& \leq
	\left\| v_k \right\|_V
	+
	\left\| v - v_k \right\|_V
	=
	\Re(\varphi_k( v_k ))
	+
	\left\| v - v_k \right\|_V
	\\ &
	=
	\Re(\varphi_k( v ))
	+
	\left\| v - v_k \right\|_V
	+
	\Re(\varphi_k( v_k - v ))
	\\ &
	\leq
	\Re(\varphi_k( v ))
	+
	\left\| v - v_k \right\|_V
	+
	\left\| \varphi_k \right\|_{ V' }
	\left\| v - v_k \right\|_V
	\\ &
	=
	\Re(\varphi_k( v ))
	+
	2 \left\| v - v_k \right\|_V
	\leq
	\sup_{ n \in \N }
	\Re(\varphi_n( v ))
	+
	2 \left\| v - v_k \right\|_V
	\leq
	\sup_{ n \in \N }
	\Re(\varphi_n( v ))
	+
	\varepsilon
	.
	\end{split}
	\end{equation}
	The proof of Corollary~\ref{cor:norm2}
	is thus completed.
\end{proof}

The last result of this subsection, Proposition~\ref{prop:charac_BE_lin} below, follows from Corollary~\ref{cor:norm2} above.  We refer to the statement of Proposition~\ref{prop:charac_BE_lin} as \emph{linear
characterization of the Borel sigma-algebra}.

\begin{prop}[Linear characterization of the Borel sigma-algebra]
	\label{prop:charac_BE_lin}
	Let $ \K \in \{ \R , \C \} $ and
	let $ ( V, \left\| \cdot \right\|_V ) $ be a separable normed $ \K $-vector space.
	Then there exists a sequence $ \varphi_n \in V' $, $ n \in \N $,
	such that
	\begin{equation}
	\begin{split}
	\mathcal{B}( V ) 
	& =
	\sigma_V\!\left( 
	\varphi \colon 
	\varphi \in V'
	\right)
	=
	\sigma_V\!\left( 
	\varphi_n \colon 
	n \in \N
	\right)
	.
	\end{split}
	\end{equation}
\end{prop}

\begin{proof}[Proof	of Proposition~\ref{prop:charac_BE_lin}]
Throughout this proof
	let $ f_v \colon V \to [0,\infty) $,
	$ v \in V $,
	be the functions which satisfy
	for all $ x, v \in V $  that
	\begin{equation}
	f_v( x )
	=
	\left\| x - v \right\|_V
	.
	\end{equation}
	Note that
	\begin{equation}
	\label{eq:Borel-sigma-algebra}
	\mathcal{B}( V ) 
	=
	\sigma_V\!\left(
	f_v \colon v \in V
	\right)
	.
	\end{equation}
	Next observe that 
	Corollary~\ref{cor:norm2}
	shows that there
	exists a sequence 
	$ \varphi_n \in V' $,
	$ n \in \N $,
	which satisfies for all $ v \in V $  that
	\begin{equation}
	\left\| v \right\|_V
	=
	\sup_{ n \in \N }
	\Re(\varphi_n( v ))
	.
	\end{equation}
	This implies that
	\begin{equation}
	\begin{split}
	\mathcal{B}( V ) & \supseteq
	\sigma_V\!\left(
	\varphi_n 
	\colon
	n \in \N 
	\right) \supseteq 	\sigma_V\big(
	( V \ni u \mapsto \Re(\varphi_n(u)) \in \R ) 
	\colon
	n \in \N 
	\big)
	\\ & =  
	\sigma_V\big(
	( V \ni u \mapsto 
	\Re(\varphi_n( u + v ))
	\in \R )
	\colon
	n \in \N , v \in V 
	\big)
	\supseteq
	\sigma_V\!\left(
	f_v
	\colon
	v \in V
	\right)
	.
	\end{split}
	\end{equation}
	Combining this with \eqref{eq:Borel-sigma-algebra} completes the proof of Proposition~\ref{prop:charac_BE_lin}.
\end{proof}

\subsection{Fourier transform of a measure}

In Lemma~\ref{lem:charac_fct2} further below we present a well-known result which states that the Fourier transform of a finite measure on a separable normed $ \R $-vector space determines the measure uniquely. The proof of Lemma~\ref{lem:charac_fct2} employs Proposition~\ref{prop:charac_BE_lin} and Proposition~\ref{thm:measure_uniqueness} above.

\begin{definition}[Image measure/Pushforward measure]
	\label{def:image_measure}
	Let
	$ \left( \Omega, \mathcal{A}, \mu \right) $
	be a measure space, 
	let
	$
	( \tilde{\Omega}, \tilde{\mathcal{A}} 
	)
	$
	be a measurable space,
	and 
	let
	$ f \colon \Omega \to \tilde{ \Omega } $
	be an $ \mathcal{A} $/$ \tilde{ \mathcal{A} } $-measurable function.
	Then we denote by
	$ f( \mu )_{ \tilde{\mathcal{A}} } \colon \tilde{ \mathcal{A} } \to [0,\infty] $
	the function which satisfies
	for all $ A \in \tilde{ \mathcal{A} } $ that
	\begin{equation}
	\big(
	f( \mu )_{ \tilde{\mathcal{A}} }
	\big)( A )
	=
	\mu\!\left(
	f^{ - 1 }( A )
	\right)
	\end{equation}
	and we call $ f( \mu )_{ \tilde{\mathcal{A}} } $
	the image measure of $ \mu $ under $ f $ associated to $ \tilde{\mathcal{A}} $.
\end{definition} 

\begin{prop}[Characteristic function] 
	\label{prop:charac_fct} 
	Let $ d \in \N $ and let $ \mu_k \colon \mathcal{B}( \R^d ) \to [0,\infty] $, $ k \in \{ 1, 2 \} $, be finite measures which satisfy for all $ \xi \in \R^d $  that 
	\begin{equation} 
	\int_{ \R^d } e^{ i \left< \xi, x \right>_{ \R^d } } \, \mu_1( dx ) = \int_{ \R^d } e^{ i \left< \xi, x \right>_{ \R^d } } \, \mu_2( dx ) . 
	\end{equation} 
	Then it holds that $ \mu_1 = \mu_2 $. 
\end{prop} 

Proposition~\ref{prop:charac_fct} is, e.g., proved as Theorem~15.8 in Klenke~\cite{Klenke2008}. 

\begin{definition}[Characteristic functional]
Let $ ( V, \left\| \cdot \right\|_V ) $ be a normed $ \R $-vector space,
let $ \mathscr{M} $ be the set of all finite measures $ \mu \colon \mathcal{B}( V ) \to [0,\infty] $ on $ ( V, \mathcal{B}( V ) ) $,
and let
$ \mathbb{M} $ be the set of all functions from $ V' $ to $ \C $.
Then we denote by
$ \mathbb{F}_V \colon \mathscr{M} \to \mathbb{M} $
the function which satisfies for all 
$ \mu \in \mathscr{M} $, $ \varphi \in V' $  that 
\begin{equation} 
( \mathbb{F}_V \mu )( \varphi ) = \big( \mathbb{F}_V( \mu ) \big)( \varphi ) = \int_{ V } e^{ i \, \varphi( x ) } \, \mu( d x ) 
\end{equation} 
and for every 
$ \mu \in \mathscr{M} $ we call $ \mathbb{F}_V( \mu ) $ the characteristic functional of $ \mu $. 
\end{definition}

\begin{lemma}[Characteristic functional determines measure uniquely] \label{lem:charac_fct2} Let $ ( V, \left\| \cdot \right\|_V ) $ be a separable normed $ \R $-vector space. Then $ \mathbb{F}_V $ is injective. 
\end{lemma} 
\begin{proof}[Proof of Lemma~\ref{lem:charac_fct2}]
Consider the notation in Subsection~\ref{sec:notation} and
let $ \mu_1, \mu_2 \colon \mathcal{B}( V ) \to [0,\infty] $
be finite measures on $ ( V, \mathcal{B}( V ) ) $ which satisfy
$ \mathbb{F}_V( \mu_1 ) = \mathbb{F}_V( \mu_2 ) $. Note that for all $ n \in \N $, $ \phi = ( \phi_1, \dots, \phi_n ) \in L( V, \R^n ) $, $ \xi \in \R^n $ it holds that
\begin{equation} 
	\begin{split} 
	\int_{ \R^n } e^{ i \left< \xi, x \right>_{ \R^n } } \, \bigl( \phi( \mu_1 )_{\mathcal{B}(\R^n)} \bigr)( dx ) & = \int_{ V } e^{ i \left< \xi, \phi( v ) \right>_{ \R^n } } \, ( \mu_1 )( dv ) = \bigl( \mathbb{F}_V \mu_1 \bigr)
	\bigl( V \ni v \mapsto \langle \xi , \phi( v ) \rangle_{ \R^n } \in \R \bigr) \\ 
	& = \bigl( \mathbb{F}_V \mu_2 \bigr)\bigl( V \ni v \mapsto \langle \xi , \phi( v ) \rangle_{ \R^n } \in \R \bigr)
	= \int_{ V } e^{ i \left< \xi, \phi( v ) \right>_{ \R^n } } \, ( \mu_2 )( dv ) \\
	& = \int_{ \R^n } e^{ i \left< \xi, x \right>_{ \R^n } } \, \bigl( \phi( \mu_2 )_{\mathcal{B}(\R^n)} \bigr)( dx ) . 
	\end{split} 
	\end{equation} 
	Proposition~\ref{prop:charac_fct} hence implies for all $ n \in \N $, $ \phi \in L( V, \R^n ) $  that 
	\begin{equation} \label{eq:coincide_phi} 
	\phi( \mu_1 )_{\mathcal{B}(\R^n)} = \phi( \mu_2 )_{\mathcal{B}(\R^n)} .
	\end{equation} 
	In the next step let $ \mathcal{E} \subseteq \mathcal{P}( V ) $ be the set given by
	\begin{equation}
	\mathcal{E} = \bigcup_{ n \in \N } \left\{ \phi^{ - 1 }( B ) \in \mathcal{P}( V ) \colon \phi \in L( V, \R^n ) , B \in \mathcal{B}( \R^n ) \right\} . 
	\end{equation} 
	Note that $ \mathcal{E} \subseteq \mathcal{B}( V ) $.
	In addition, observe that \eqref{eq:coincide_phi} shows that 
	\begin{equation} \mu_1|_{ \mathcal{E} } = \mu_2|_{ \mathcal{E} } . 
	\end{equation}
	This, the fact that $ \mathcal{E} $ is $ \cap $-stable, the fact $ V \in \mathcal{E} $, and Proposition~\ref{thm:measure_uniqueness} prove that 
	\begin{equation} \label{eq:mu12_coincide} \mu_1|_{ \sigma_{ V }( \mathcal{E} ) } = \mu_2|_{ \sigma_{ V }( \mathcal{E} ) } . 
	\end{equation} 
	Moreover, observe that Proposition~\ref{prop:charac_BE_lin} establishes that 
	\begin{equation} \sigma_V( \mathcal{E} ) = \mathcal{B}( V ) . 
	\end{equation} 
	Combining this with \eqref{eq:mu12_coincide} completes the proof of Lemma~\ref{lem:charac_fct2}. 
\end{proof}

\subsection{Fernique's theorem}

The proof of Fernique's theorem (see Proposition~\ref{thm:fernique} below) requires the two following well-known auxiliary lemmas, Lemma~\ref{lem:indep} and Lemma~\ref{lem:iid} below.

\begin{lemma}[Independent projections of random variables]
\label{lem:indep}
Let $(V_1, \left\| \cdot \right\|_{V_1})$ and $(V_2, \left\| \cdot \right\|_{V_2})$ be separable normed $ \R $-vector spaces, 
let $(\Omega, \mathcal{F}, \P)$  be a probability space, and let $X_1 \colon \Omega \to V_1$ and  $X_2 \colon \Omega \to V_2$ be 
functions which satisfy 
for all $\varphi_1 \in (V_1)'$, $\varphi_2 \in (V_2)'$ 
that 
$  \varphi_1 \circ X_1 
\colon \Omega \to \R $ and 
$ \varphi_2 \circ X_2 
\colon \Omega \to \R $ are independent random variables. Then it holds that $X_1$ and $X_2$ are independent random variables.
\end{lemma}
\begin{proof}[Proof of Lemma~\ref{lem:indep}]
Note that
the assumption that
$ \forall \, \varphi_1 \in (V_1)', \, \varphi_2 \in (V_2)' \colon
(
\varphi_1 \circ X_1 $
and
$ \varphi_2 \circ X_2 $
are
$\mathcal{ F } $/$ \mathcal{ B }( \R )$-measurable$)$
and Proposition~\ref{prop:charac_BE_lin}
show that
$X_1$ is $ \mathcal{ F } $/$\mathcal{ B }( V_1 ) $-measurable
and that
$X_2$ is $ \mathcal{ F } $/$\mathcal{ B }( V_2 ) $-measurable.
Throughout this proof let 
$
  ( \tilde{\Omega}, \tilde{\mathcal{F}}, \tilde{\P} )
$
be the probability space given by
\begin{equation}
  \tilde{\Omega} = V_1 \times V_2, 
  \quad
  \tilde{\mathcal{F}} = \mathcal{B}(V_1) \otimes \mathcal{B}(V_2) , 
  \quad
  \text{and}
  \quad
  \tilde{\P} = X_1(\P)_{\mathcal{B}(V_1)} \otimes X_2(\P)_{\mathcal{B}(V_2)} 
  . 
\end{equation}
Next note that for all $\varphi \in L( \tilde{\Omega} ,\R)$ it holds that
	\begin{align}
	\label{eq:indep}
	\begin{split}
	\big( \mathbb{F}_{\tilde{\Omega}} \, \tilde{\P}\big)(\varphi)
	& = 
	\int_{\tilde{\Omega}} \exp\!\big( i \, \varphi( x_1, x_2 )\big) \, \tilde{\P}(d(x_1,x_2))  
	\\
	&
	=  \int_{\tilde{\Omega}} 
	\exp\!\big( i \, \varphi( x_1, 0 )\big) 
	\exp\!\big( i \, \varphi( 0, x_2 )\big)\,  \big(X_1(\P)_{\mathcal{B}(V_1)} \otimes X_2(\P)_{\mathcal{B}(V_2)} \big)(d(x_1,x_2)) 
	\\
	&=  \int_{V_1} 
	\exp\!\big( i \, \varphi( x_1, 0 )\big) \, X_1(\P)_{\mathcal{B}(V_1)}(dx_1) 
	\int_{V_2} \exp\!\big( i \, \varphi( 0, x_2 )\big) \, X_2(\P)_{\mathcal{B}(V_2)}(dx_2) \\
	& = 
	\E\!\left[e^{ i \, \varphi( X_1, 0 )}\right] \E\!\left[e^{ i \, \varphi( 0, X_2 )}\right].
	\end{split}
	\end{align}
	Moreover, observe that for all $\varphi \in L( \tilde{\Omega} ,\R)$ it holds that $(V_1 \ni v \mapsto \varphi( v, 0 ) \in \R) \in (V_1)' $ 
	and $(V_2 \ni v \mapsto \varphi( 0,v ) \in \R) \in (V_2)' $. 
	This and \eqref{eq:indep} imply for all $\varphi \in L( \tilde{\Omega} ,\R)$ that
	\begin{align}
	\begin{split}
	  \big( \mathbb{F}_{\tilde{\Omega}} \, \tilde{\P} \big)(\varphi) 
	  & = 
    	  \E\!\left[
    	    e^{ i \, \varphi( X_1, 0 )}
    	    \, 
    	    e^{ i \, \varphi( 0, X_2 )}
    	  \right]
	  =
	  \E\!\left[
	    e^{ i \{ \varphi( X_1, 0 ) + \varphi( 0, X_2 ) \} }
	  \right] 
	\\ &
	  =
	  \E\!\left[e^{ i \, \varphi( X_1, X_2 )}\right] 
	  = \bigl(
	    \mathbb{F}_{\tilde{\Omega}}\bigl[ (X_1,X_2)( \P )_{\tilde{\mathcal{F}}} \bigr]
	  \bigr) (\varphi).
	\end{split}
	\end{align}
	Combining this with Lemma~\ref{lem:charac_fct2} yields that
	\begin{align}
  	  X_1(\P)_{\mathcal{B}(V_1)} \otimes X_2(\P)_{\mathcal{B}(V_2)} 
  	  =
	  \tilde{\P}
	  =
  	  (X_1,X_2)(\P)_{\mathcal{B}(V_1) \otimes \mathcal{B}(V_2)} .
	\end{align}
	The proof of Lemma~\ref{lem:indep} is thus completed.
\end{proof}

Lemma~\ref{lem:iid} below demonstrates under suitable hypotheses that an appropriate orthogonal transformation 
of two appropriate independent random variables also results in independent random variables. 
Observe that the columns of the $ 2 \times 2 $-matrix 
\begin{equation}
\label{eq:orthogonal_transform}
    \left(
      \begin{array}{cc}
        \nicefrac{ 1 }{ \sqrt{2} }
      &
        \nicefrac{ 1 }{ \sqrt{2} }
      \\
        \nicefrac{ 1 }{ \sqrt{2} }
      &
        - \nicefrac{ 1 }{ \sqrt{2} }
      \end{array}
    \right)
\end{equation}
constitute an orthonormal basis of $ \R^2 $.
Roughly speaking, the orthogonal transformation 
associated to \eqref{eq:orthogonal_transform} is employed in the next lemma.

\begin{lemma}[Orthogonal transformations of independent random variables]
\label{lem:iid}
Let $(V, \left\| \cdot \right\|_V)$ be a separable normed $ \R $-vector space, 
let $(\Omega, \mathcal{F}, \P)$ be a probability space, 
let $X_1, X_2 \colon \Omega \to V$ be independent random variables 
which satisfy for every $\varphi \in V'$ 
that $ \varphi \circ X_1 
\colon \Omega \to \R $ and 
$ \varphi \circ X_2 
\colon \Omega \to \R $ 
are identically distributed centered Gaussian random variables, 
and let $Y_1, Y_2 \colon \Omega \to V$ satisfy  
\begin{equation}
  Y_1= 2^{-\nicefrac{1}{2}}(X_1 + X_2)
\qquad 
  \text{and}
\qquad 
  Y_2 = 2^{-\nicefrac{1}{2}}(X_1 - X_2)
  .
\end{equation}
Then 
\begin{enumerate}[(i)]
\item 
it holds that $Y_1$ and $Y_2$ are independent random variables and
\item \label{item:equal_dist}
it holds that
$Y_1(\P)_{\mathcal{B}(V)}=Y_2(\P)_{\mathcal{B}(V)}=X_1(\P)_{\mathcal{B}(V)}=X_2(\P)_{\mathcal{B}(V)}$.
\end{enumerate}
\end{lemma}

\begin{proof}[Proof of Lemma~\ref{lem:iid}]
Observe that the assumption that
$ \forall \, \varphi \in V' \colon
( \varphi \circ X_1
\text{ and }
\varphi \circ X_2 $
are identically distributed random variables$ ) $
proves
for all $\varphi \in V'$ that
\begin{equation}
	  \E\!\left[ 
	    e^{ 
	      i \,
	      \varphi(X_1) 
	    } 
	  \right] 
	= 
	  \E\!\left[ 
	    e^{ 
	      i \,
	      \varphi(X_2) 
	    } 
	  \right]
	  .
\end{equation}
Lemma~\ref{lem:charac_fct2}
hence shows that
\begin{equation} \label{eq:equal_dist_X}
X_1(\P)_{\mathcal{B}(V)}=X_2(\P)_{\mathcal{B}(V)}.
\end{equation}
In addition, note that the hypothesis that $ X_1 $ and $ X_2 $ are independent random variables ensures 
for all $\varphi_1, \varphi_2 \in V'$, $\xi = (\xi_1, \xi_2) \in \R^2 $ that 
	\begin{equation}
	\begin{split}
	  \E\!\left[ 
	    e^{ i \langle \xi, (\varphi_1(Y_1), \varphi_2(Y_2)) \rangle_{\R^2} } 
	  \right] 
	  & =
	  \E\!\left[ 
	    e^{ i \, \xi_1 \varphi_1(Y_1) + i \, \xi_2 \varphi_2(Y_2) } 
	  \right] 
	\\ &
	  =  \E\big[ e^{ i  \,2^{-\nicefrac{1}{2}} \xi_1 \varphi_1 (X_1+X_2) + i \,  2^{-\nicefrac{1}{2}} \xi_2 \varphi_2 (X_1-X_2)}\big] 
	\\
	& 
	  = \E\big[ 
	    e^{ i \, 2^{-\nicefrac{1}{2}} (\xi_1\varphi_1 +\xi_2\varphi_2) ( X_1) + i \, 2^{-\nicefrac{1}{2}} (\xi_1\varphi_1 - \xi_2\varphi_2) (X_2)}
	  \big] 
	\\
	& 
	  = 
	  \E\big[ e^{ i \, 2^{-\nicefrac{1}{2}} (\xi_1\varphi_1 +\xi_2\varphi_2) ( X_1)} \big]  
	  \,
	  \E\big[ e^{ i \, 2^{-\nicefrac{1}{2}} (\xi_1\varphi_1 - \xi_2\varphi_2) (X_2)}
	  \big] .
\end{split}
\end{equation}
Hence, we obtain that for all $\varphi_1, \varphi_2 \in V'$, $\xi = (\xi_1, \xi_2) \in \R^2 $ it holds that 
\begin{equation}
\begin{split}
	&
	  \E\!\left[ 
	    e^{ 
	      i \langle \xi, (\varphi_1(Y_1), \varphi_2(Y_2)) \rangle_{\R^2} 
	    } 
	  \right] 
	\\
	& = 
  	  \exp\!\left( - \tfrac{1}{2} \E\!\left[| 2^{ - \nicefrac{1}{2} } (\xi_1\varphi_1 +\xi_2\varphi_2) ( X_1)|^2\right] \right) 
  	  \exp\!\left( - \tfrac{1}{2} \E\!\left[| 2^{ - \nicefrac{1}{2} } (\xi_1\varphi_1 -\xi_2\varphi_2) ( X_2)|^2\right]\right) 
	\\
	&
	= 
  	  \exp\!\left( - \tfrac{1}{4} \E\!\left[| (\xi_1\varphi_1 +\xi_2\varphi_2) ( X_1)|^2\right] \right) 
  	  \exp\!\left( 
  	    - \tfrac{1}{4} 
  	    \E\!\left[| (\xi_1\varphi_1 -\xi_2\varphi_2) ( X_2)|^2\right]
  	  \right) 
  	\\
  	&
  	=
  	  \exp\!\left( 
  	    - 
  	    \tfrac{ 1 }{ 4 } 
  	    \left\{
	      \E\!\left[
		| ( \xi_1 \varphi_1 + \xi_2 \varphi_2 )( X_1 ) |^2 
	      \right]
	      +
	      \E\!\left[
	        | ( \xi_1 \varphi_1 - \xi_2 \varphi_2 )( X_1 ) |^2
	      \right]
	    \right\}
  	  \right) 
  	\\
  	&
  	=
  	  \exp\!\left( 
  	    - 
  	    \tfrac{ 1 }{ 4 } 
  	    \left\{
  	      2
  	      \,
	      \E\!\left[
		| ( \xi_1 \varphi_1 )( X_1 ) |^2 
	      \right]
	      +
	      2 \,
	      \E\!\left[
	        | ( \xi_2 \varphi_2 )( X_1 ) |^2
	      \right]
	    \right\}
  	  \right) 
  	  .
\end{split}
\end{equation}
This shows for all $\varphi_1, \varphi_2 \in V'$, $\xi = (\xi_1, \xi_2) \in \R^2 $ that 
\begin{equation}
\label{eq:distribution_Fernique_same}
\begin{split}
	  \E\!\left[ 
	    e^{ 
	      i \langle \xi, (\varphi_1(Y_1), \varphi_2(Y_2)) \rangle_{\R^2} 
	    } 
	  \right]
	& = \exp \! \left( - \tfrac{1}{2}\E\!\left[|\xi_1 \varphi_1(X_1)|^2\right]\right) \exp \! \left( - \tfrac{1}{2}\E\!\left[|\xi_2 \varphi_2(X_2)|^2\right]\right) \\
	&= 
	  \E\big[ 
	    e^{ 
	      i \, \xi_1 \varphi_1( X_1 ) 
	    }
	  \big] 
	  \, 
	  \E\big[ 
	    e^{ 
	      i  \, \xi_2 \varphi_2( X_2 ) 
	    }
	  \big] 
	=
	  \E\big[ 
	    e^{ 
	      i \, \xi_1 \varphi_1( X_1 ) 
	      +
	      i  \, \xi_2 \varphi_2( X_2 ) 
	    }
	  \big] 
	\\ &
	= 
	  \E \!\left[ e^{ i  \langle \xi,(\varphi_1(X_1), \varphi_2(X_2))\rangle_{\R^2}}\right].
	\end{split}
	\end{equation}
In particular, this implies
for all $\varphi \in V'$ that
\begin{equation}
\E\!\left[ 
 e^{ 
   i \,
   \varphi(Y_1) 
 } 
\right] 
= 
\E\!\left[ 
 e^{ 
   i \,
   \varphi(X_1) 
 } 
\right]
\qquad \text{and} \qquad
\E\!\left[ 
 e^{ 
   i \,
   \varphi(Y_2) 
 } 
\right] 
= 
\E\!\left[ 
 e^{ 
   i \,
   \varphi(X_2) 
 } 
\right].
\end{equation}
Lemma~\ref{lem:charac_fct2}
hence establishes that
\begin{equation}
  Y_1(\P)_{\mathcal{B}(V)}=X_1(\P)_{\mathcal{B}(V)}
  \qquad 
  \text{and}
  \qquad
  Y_2(\P)_{\mathcal{B}(V)}=X_2(\P)_{\mathcal{B}(V)}
  .
\end{equation}
This and \eqref{eq:equal_dist_X} prove~\eqref{item:equal_dist}.
Next note that
Lemma~\ref{lem:charac_fct2} 
and \eqref{eq:distribution_Fernique_same} show 
for all $\varphi_1, \varphi_2 \in V'$ 
that
\begin{equation}
( \varphi_1 \circ Y_1, \varphi_2 \circ Y_2 )( \P )_{ \mathcal{B}( \R^2 ) }
=
( \varphi_1 \circ X_1, \varphi_2 \circ X_2 )( \P )_{ \mathcal{B}( \R^2 ) }.
\end{equation}
This, the assumption that $ X_1 $ and $ X_2 $ are independent, and~\eqref{item:equal_dist} ensure
for all $ \varphi_1, \varphi_2 \in V' $ that
\begin{equation}
\begin{split}
( \varphi_1 \circ Y_1 )( \P )_{ \mathcal{B}( \R ) } \otimes ( \varphi_2 \circ Y_2 )( \P )_{ \mathcal{B}( \R ) }
& =
( \varphi_1 \circ X_1 )( \P )_{ \mathcal{B}( \R ) } \otimes ( \varphi_2 \circ X_2 )( \P )_{ \mathcal{B}( \R ) }
\\ & =
( \varphi_1 \circ X_1, \varphi_2 \circ X_2 )( \P )_{ \mathcal{B}( \R^2 ) }
=
( \varphi_1 \circ Y_1, \varphi_2 \circ Y_2 )( \P )_{ \mathcal{B}( \R^2 ) }.
\end{split}
\end{equation}
This proves
for every $ \varphi_1, \varphi_2 \in V' $
that
$
  \varphi_1 \circ Y_1
$
and
$
  \varphi_2 \circ Y_2
$
are independent random variables.
Lemma~\ref{lem:indep} hence establishes that $Y_1$ and $Y_2$ are independent random variables. The proof of Lemma~\ref{lem:iid} is thus completed.
\end{proof}

In the next result, Proposition~\ref{thm:fernique} below, we present Fernique's theorem (see, e.g., Theorem~8.2.1 in Stroock~\cite{Stroock2010}).

\begin{prop}[Fernique's theorem]
\label{thm:fernique}
Let $ (V, \left\| \cdot \right\|_V) $ be a separable normed $ \R $-vector space, 
let $ (\Omega, \mathcal{F}, \P) $ be a probability space, 
let $X \colon \Omega \to V$ be a function which satisfies  for all $\varphi \in V'$  that $\varphi \circ X \colon \Omega \to \R $ is a
centered Gaussian random variable, 
and let $ R \in (0, \infty) $ satisfy 
\begin{equation}
  R \geq 
  \inf\!\big( 
    \{ r \in [0, \infty) \colon \P(\|X\|_V \leq r) \geq \nicefrac{9}{10}
  \} \big)
  .
\end{equation}
Then
	\begin{align}\label{eq:th:fern0}
	\begin{split}
	\E \! \left[ \exp\! \left( \frac{\|X\|_V^2}{18 R^2}\right)\right] \leq \sqrt{e} + \sum_{n=0}^{\infty} \bigg[\frac{e}{3}\bigg]^{(2^n)}< 13 < \infty.
	\end{split}
	\end{align}	
\end{prop}

\begin{proof}[Proof of Proposition~\ref{thm:fernique}]
	Throughout this proof let
	$
	  ( \tilde{\Omega}, \tilde{\mathcal{F}}, \tilde{\P} )
	$
	be the probability space given by
	$\tilde{\Omega} = \Omega \times \Omega$, 
	$\tilde{\mathcal{F}} = \mathcal{F} \otimes \mathcal{F}$, 
	and
	$\tilde{\P} = \P \otimes \P$, 
	let $ Y_1, Y_2 \colon \tilde{\Omega} \to V $ 
	be the functions which satisfy for all 
	$\omega_1, \omega_2 \in \Omega$ that 
	\begin{equation}
	  Y_1(\omega_1, \omega_2) = X(\omega_1) 
	\qquad 
	  \text{and}
	\qquad
	  Y_2(\omega_1, \omega_2) = X(\omega_2) , 
	\end{equation}
	let $ Z_1, Z_2 \colon \tilde{\Omega} \to V $ 
	be the functions which satisfy 
	\begin{equation}
	  Z_1 = 2^{-\nicefrac{1}{2}}(Y_1 + Y_2) 
	  \qquad
	  \text{and}
	  \qquad
	  Z_2 = 2^{-\nicefrac{1}{2}}(Y_1 - Y_2) , 
	\end{equation}
	and let $t_n \in (0, \infty)$, $n \in \N_0$, be the sequence 
	of real numbers which satisfies for all $n \in \N$ that 
	\begin{equation}
	  t_0 = R  
	  \qquad 
	  \text{and} 
	  \qquad 
	  t_n = R + \sqrt{2} \, t_{ n - 1 } . 
	\end{equation}
	Observe that 
	$ Y_1 $ and $ Y_2 $ are independent random variables.
	In addition, note that
	for every $\varphi \in V'$ it holds that 
	the random variables 
	$ \varphi \circ Y_1 \colon \tilde{\Omega }\to \R $ and 
	$ \varphi\circ Y_2 \colon \tilde{\Omega} \to \R $  
	have the same distribution on $ ( \R, \mathcal{B}(\R) ) $ 
	as the random variable $ \varphi \circ X \colon \Omega \to \R $. 
	Lemma~\ref{lem:iid} 
	hence ensures that  $Z_1$ and $Z_2$ are 
	independent random variables and 
	\begin{equation}
	  Z_1(\tilde{\P})_{\mathcal{B}(V)}=
	  Z_2(\tilde{\P})_{\mathcal{B}(V)}=
	  Y_1(\tilde{\P})_{\mathcal{B}(V)}=X(\P)_{\mathcal{B}(V)}
	  .
	\end{equation}
	This proves  for all $s, t \in (0, \infty)$ with $s \leq t$  that
	\begin{align}
	\begin{split}
	&\P(\|X\|_V \leq s) \, \P(\|X\|_V > t)= \tilde{\P}(\|Z_2\|_V \leq s) \, \tilde{\P}(\|Z_1\|_V > t) \\
	&= \tilde{\P}\big(\{\|Z_2\|_V \leq s\} \cap \{\|Z_1\|_V > t\}\big) \\
	&= \tilde{\P}\big(\{\|Y_1- Y_2\|_V \leq \sqrt{2} \,s \} \cap \{ \|Y_1 +Y_2\|_V > \sqrt{2} \,t \}\big) \\
	& \leq \tilde{\P}\big(\{| \|Y_1\|_V- \|Y_2\|_V | \leq \sqrt{2} \, s \} \cap \{ \|Y_1\|_V +\|Y_2\|_V > \sqrt{2} \, t \}\big) \\
	& \leq \tilde{\P}\big(\!\min\{\|Y_1\|_V, \|Y_2\|_V\} > 2^{-\nicefrac{1}{2}}(t-s)\big) = \big| \P\big(\|X\|_V > 2^{-\nicefrac{1}{2}}(t-s) \big) \big|^2.
	\end{split}
	\end{align}
	This, in turn, implies for all $n \in \N$  that
	\begin{align}
	\P(\|X\|_V \leq R) \,\P(\|X\|_V > t_n) \leq | \P(\|X\|_V > t_{n-1} ) |^2.
	\end{align}
	The fact that  $\P(\|X\|_V \leq R) \geq \nicefrac{9}{10} >0$ hence shows for all $n \in \N$  that
	\begin{align}
	\frac{\P(\|X\|_V > t_n)}{\P(\|X\|_V \leq R) } \leq \left( \frac{\P(\|X\|_V > t_{n-1})}{\P(\|X\|_V \leq R) } \right)^2.
	\end{align}
	This and induction on $n \in \N_0$ establish for all $n \in \N_0$ that
	\begin{align}
	\label{eq:thm:fern1}
	\frac{\P(\|X\|_V > t_n)}{\P(\|X\|_V \leq R) } \leq \left( \frac{\P(\|X\|_V > R)}{\P(\|X\|_V \leq R) } \right)^{(2^n)}.
	\end{align}
	Moreover, induction on $n \in \N_0$ ensures for all $n \in \N_0$ that
	\begin{align}
	t_n = R \cdot \frac{2^{\frac{n+1}{2}}-1}{\sqrt{2}-1} \leq  (\sqrt{2}+1) \, 2^{\frac{n+1}{2}} R \leq 3 \cdot 2^{\frac{n+1}{2}} R .
	\end{align}
	Combining this with \eqref{eq:thm:fern1} 
	and the fact that 
	$\P(\|X\|_V \leq R) \geq \nicefrac{9}{10} \geq 9 \, \P(\|X\|_V > R)$  
	yields that for all $n \in \N_0$ it holds that
	\begin{align}
	\P(\|X\|_V > 3 \cdot 2^{\frac{n}{2}} R ) \leq 3^{-(2^n)}.
	\end{align} 
	The fact that $\nicefrac{e}{3}<1$ hence shows that
	\begin{align}
	\begin{split}
	\E \! \left[ \exp\! \left( \frac{\|X\|_V^2}{18 R^2}\right)\right]
	&\leq \sqrt{e} \, \P(\|X\|_V \leq 3 R) + \sum_{n=0}^{\infty} e^{(2^n)} \, \P \big( 3 \cdot 2^{\frac{n}{2}} R < \|X\|_V \leq 3 \cdot 2^{\frac{n+1}{2}} R\big) \\
	&\leq \sqrt{e}+ \sum_{n=0}^{\infty} e^{(2^n)} \, \P \big( \|X\|_V > 3 \cdot 2^{\frac{n}{2}} R \big)  \leq \sqrt{e} + \sum_{n=0}^{\infty} \bigg[\frac{e}{3}\bigg]^{(2^n)}\\
	& \leq \sqrt{e} + \sum_{n=0}^{\infty} \bigg[\frac{e}{3}\bigg]^{n} = \sqrt{e} + \frac{3}{3-e} < 13 < \infty.
	\end{split}
	\end{align}
	The proof of Proposition~\ref{thm:fernique} is thus completed.
\end{proof}

\section{Abstract examples}
\label{sec:abstract}

In this section we verify the assumptions of Theorem~\ref{thm:strong} above
in the case of more specific SPDEs (see the setting in Subsection~\ref{setting:example} below)
and establish strong convergence in this setting in Proposition~\ref{abs:prop:last} below.
First, we show a result on transformations of semigroups for solutions of SPDEs in Proposition~\ref{prop:transform_SG} below.
Next we combine this with Fernique's theorem (see Proposition~\ref{thm:fernique} above)
and the elementary results in Lemmas~\ref{abs:phi(w)_finite}--\ref{lem:conv:series},
Proposition~\ref{prop:abstract}, and Lemma~\ref{lemma:conv:rate}
to derive certain properties of stochastic convolution processes (see Proposition~\ref{prop:exists} below).
Finally, the latter allow us to apply Theorem~\ref{thm:strong} in order to prove Proposition~\ref{abs:prop:last}.

\subsection{Transformations of semigroups for solutions of SPDEs}

Roughly speaking, Proposition~\ref{prop:transform_SG} below
proves that a mild solution of an SPDE does not depend on a shift of the linear part of the drift coefficient function
if the nonlinear part of the drift coefficient function is shifted accordingly.
This result is achieved under optimal hypotheses in the sense that
the hypotheses of Proposition~\ref{prop:transform_SG}
are required for the mathematical formulation of the statement to be meaningful
(see, in particular, \eqref{eq:assume1}--\eqref{eq:assume2}).
To the best of our knowledge, Proposition~\ref{prop:transform_SG} is the first result in the literature
to establish this assertion under optimal hypotheses,
even in the special case of partial differential equations (where the diffusion coefficient function is zero).

\begin{prop}
\label{prop:transform_SG}
Consider the notation in Subsection~\ref{sec:notation},
let 
$ ( H, \left< \cdot, \cdot \right>_H, \left\| \cdot \right\|_H ) $
and
$ ( U, \left< \cdot, \cdot \right>_U, \left\| \cdot \right\|_U ) $
be separable $ \R $-Hilbert spaces,
let $ \mathbb{H} \subseteq H$ be a nonempty orthonormal basis of $ H $,
let $T \in (0, \infty)$, $ \alpha, \beta, \gamma, \eta, \kappa \in \R $,
let $\lambda \colon \mathbb{H} \to \R $
be a function which satisfies
$\sup_{b \in \mathbb{H}} \lambda_b < \kappa $,
let $ A \colon D(A) \subseteq H \to H $
be the linear operator which satisfies
$ D(A) = \{ v \in H \colon \sum_{b \in \mathbb{H}} | \lambda_b \langle b , v \rangle_H |^2 < \infty \} $
and
$ \forall \, v \in D(A) \colon A v = \sum_{b \in \mathbb{H}} \lambda_b \langle b , v \rangle_H b $,
let 
$ 
( 
H_r 
,
\left< \cdot, \cdot \right>_{ H_r }
,
\left\| \cdot \right\|_{ H_r } 
) 
$, $ r \in \R $,
be a family of interpolation spaces associated 
to $ \kappa - A $,
let $ ( \Omega, \mathcal{F}, \P)$ be a probability space with a normal filtration $( \mathbb{F}_t )_{ t \in [0,T] }  $,
let $ ( W_t )_{ t \in [0,T] } $ be an
$ \operatorname{Id}_U $-cylindrical 
$ ( \Omega, \mathcal{F}, \P, ( \mathbb{F}_t )_{ t \in [0,T] } ) $-Wiener process,
let
$ O \in \mathcal{B}( H_{ \gamma } ) $,
let
$ F \colon O \to H_{ \alpha } $
be a $ \mathcal{B}( O ) $/$ \mathcal{B}( H_{ \alpha } ) $-measurable function,
let
$ \tilde{F} \colon O \to H_{ \min\{ \alpha , \gamma \} } $
be the function which satisfies for all $ v \in O $  that
$
\tilde{F}( v ) = \eta v + F( v )
$,
let
$ B \colon O \to \mathrm{HS}( U, H_{ \beta } ) $
be a $ \mathcal{B}( O ) $/$ \mathcal{B}( \mathrm{HS}( U, H_{ \beta } ) ) $-measurable function,
let
$ \xi \colon \Omega \to O $
be an $ \mathbb{F}_0 $/$ \mathcal{B}( O ) $-measurable function,
and let
$ X \colon [0,T] \times \Omega \to O $ 
be  an $ ( \mathbb{F}_t )_{ t \in [0,T] } $/$ \mathcal{B}( O ) $-predictable stochastic process 
which satisfies for all $ t \in [0,T] $  that
\begin{equation}
\label{eq:assume1}
\P\!\left(
\int_0^t 
\| 
e^{  ( t - s )A } F( X_s ) 
\|_{ H_{ \gamma } } 
+ 
\| 
e^{  ( t - s )A } 
B( X_s ) 
\|^2_{ \mathrm{HS}( U, H_{ \gamma } ) } \, ds < \infty
\right) = 1
\end{equation}
and
\begin{equation}
\label{eq:assume2}
\begin{split}
\left[ 
X_t 
\right]_{ \P, \mathcal{B}( H_{ \gamma } ) }
& =
\left[ 
e^{ tA } \xi 
+
\int_0^t
\mathbbm{1}_{
	\{
	\int_0^t \| e^{  ( t - u )A } F( X_u ) \|_{ H_{ \gamma } } du < \infty 
	\}
}
\,
e^{  ( t - s )A } F( X_s ) \, ds
\right]_{ \P, \mathcal{B}( H_{ \gamma } ) }
\\ & \quad +
\int_0^t
e^{  ( t - s ) A} B( X_s ) \, dW_s
.
\end{split}
\end{equation}
Then it holds for all  $ t \in [0,T] $  that
\begin{equation}
\P\!\left(
\int_0^t 
\| 
e^{  ( t - s )(A- \eta) } \tilde{F}( X_s ) 
\|_{ H_{ \gamma } } 
+ 
\| 
e^{  ( t - s )(A-\eta) } 
B( X_s ) 
\|^2_{ \mathrm{HS}( U , H_{ \gamma } ) } \, ds < \infty
\right) = 1
\end{equation}
and
\begin{equation}
\begin{split}
\left[ 
X_t 
\right]_{ \P, \mathcal{B}( H_{ \gamma } ) }
& =
\left[ 
e^{ t(A-\eta) } \xi 
+
\int_0^t
\mathbbm{1}_{
	\{
	\int_0^t \| e^{  ( t - u )(A-\eta) } \tilde{F}( X_u ) \|_{ H_{ \gamma } } du < \infty 
	\}
}
\,
e^{  ( t - s )(A- \eta)} \tilde{F}( X_s ) \, ds
\right]_{ \P, \mathcal{B}( H_{ \gamma } ) }
\\ & \quad +
\int_0^t
e^{  ( t - s ) (A-\eta)} B( X_s ) \, dW_s
.
\end{split}
\end{equation}
\end{prop}

\begin{proof}[Proof of Proposition~\ref{prop:transform_SG}]
Throughout this proof let $ \psi, \psi_1 \colon [0,T] \times H_{ \gamma } \to H_{ \gamma } $
and
$ \psi_2 \colon [0,T] \times H_{ \gamma } \to L( H_{ \gamma } ) $
be the functions which satisfy
for all $ (t, x) \in [0,T] \times H_{ \gamma } $,
$ v \in H_{ \gamma } $
that
\begin{equation}
\psi( t, x ) 
= 
e^{ \eta t } x
,
\quad
\psi_1( t, x )
=
\tfrac{ \partial }{ \partial t } \psi( t, x )
=
\eta 
\,
\psi( t, x )
,
\quad
\text{and}
\quad
\psi_2( t, x ) \, v
=
\tfrac{ \partial }{ \partial x } \psi( t, x ) \, v
=
\psi( t, v )
.
\end{equation}
Next observe that
\eqref{eq:assume1}--\eqref{eq:assume2}
imply that for all $ t \in [0,T] $ it holds that
\begin{equation}
\begin{split}
&
\P\!\left(	\int_0^t 
\| 
e^{ ( t - s )( A - \eta )  } 
e^{ - \eta s }
F( X_s ) 
\|_{ H_{ \gamma } } 
+ 
\| 
e^{ ( t - s ) ( A - \eta )  } 
e^{ - \eta s }
B( X_s ) 
\|^2_{ \mathrm{HS}( U, H_{ \gamma } ) } \, ds < \infty \right)
\\ & =
\P \!\left(
e^{ - \eta t }
\int_0^t 
\| 
e^{ ( t - s ) A  } 
F( X_s ) 
\|_{ H_{ \gamma } } \, ds 
+ 
e^{ - 2 \eta t }
\int_0^t 
\| 
e^{ ( t - s )  A  } 
B( X_s ) 
\|^2_{ \mathrm{HS}( U, H_{ \gamma } ) } \, ds
< \infty \right)
\\ & =
\P \!\left(
\int_0^t 
\| 
e^{ ( t - s ) A  } 
F( X_s ) 
\|_{ H_{ \gamma } }
+
\| 
e^{ ( t - s )  A  } 
B( X_s ) 
\|^2_{ \mathrm{HS}( U, H_{ \gamma } ) } \, ds
< \infty \right) =1
\end{split}
\end{equation}
and
\begin{equation}
\begin{split}
\left[ 
e^{ - \eta t } X_t 
\right]_{ \P, \mathcal{B}( H_{ \gamma } ) }
& =
\left[ 
e^{ t(A-\eta) } \xi 
+
\int_0^t
\mathbbm{1}_{
	\{
	\int_0^t \| e^{  ( t - u )A }
	F( X_u ) \|_{ H_{ \gamma } } du < \infty 
	\}
}
\,
e^{  ( t - s )(A- \eta)}
e^{ - \eta s }  F( X_s ) \, ds
\right]_{ \P, \mathcal{B}( H_{ \gamma } ) }
\\ & \quad +
\int_0^t
e^{  ( t - s ) (A-\eta)}
e^{ - \eta s }  B( X_s ) \, dW_s .
\end{split}
\end{equation}
Note that this establishes that the stochastic process
$
\big( [0,T] \times \Omega \ni (t,\omega) \mapsto  e^{ - \eta t } X_t(\omega) \in H_{\gamma} \big)
$
is an $ ( \Omega, \mathcal{F}, \P, ( \mathbb{F}_t )_{ t \in [0,T] } ) $-mild It\^{o} process with evolution family
$ \big( \{ (t_1, t_2) \in [0,T]^2 \colon t_1 < t_2 \} \ni ( s , t ) \mapsto e^{  ( t - s )( A - \eta ) } \in L(H_{\min\{\alpha,\beta,\gamma\}}, H_{\gamma}) \big)
$,
mild drift
$ \big( 
[0,T] \times \Omega \ni (t, \omega) \mapsto	e^{ - \eta t } F( X_t(\omega) ) \in  H_{\min\{\alpha,\beta,\gamma\}} \big)
$,
and mild diffusion 
$ \big( 
[0,T] \times \Omega \ni (t, \omega) \mapsto
e^{ - \eta t } B( X_t(\omega) ) \in \mathrm{HS}(U, H_{\min\{\alpha,\beta,\gamma\}}) \big)
$ (see Definition~1 in Da~Prato, Jentzen, \& R\"ockner~\cite{DaPratoJentzenRoeckner_Mild}).
The mild It\^{o} formula in Theorem~1 in Da~Prato, Jentzen, \& R\"ockner~\cite{DaPratoJentzenRoeckner_Mild} 
hence proves that
for all $ t \in [0,T] $ it holds that
\begin{gather}
\P\! \left( \int_0^t
\big\|
	\psi_1 \bigl( 
	s, 
	e^{  ( t - s )( A - \eta ) } 
	e^{ - \eta s } 
	X_s
	\bigr)
\big\|_{ H_{ \gamma } }
\,
ds
< \infty \right)=1
,\\
 \P \! \left(\int_0^t
\big\|
	\psi_2\!\left( 
	s, 
	e^{ ( t - s ) ( A - \eta ) } 
	e^{ - \eta s } 
	X_s 
	\right)
	e^{
		 ( t - s ) ( A - \eta )
	}
	e^{ - \eta s } 
	F( X_s )
\big\|_{ H_{ \gamma } }
\,
ds
< \infty \right)=1
,\\
\P \! \left( \int_0^t
\big\|
	\psi_2\!\left( 
	s, 
	e^{ ( t - s ) ( A - \eta )  } 
	e^{ - \eta s } 
	X_s
	\right)
	e^{
		 ( t - s ) ( A - \eta )
	}
	e^{ - \eta s } 
	B( X_s )
\big\|_{ \mathrm{HS}( U, H_{ \gamma } ) }^2
\,
ds
< \infty \right)=1 ,
\end{gather}
and
\begin{equation}
\begin{split}
& \left[ X_t \right]_{ \P, \mathcal{B}( H_{ \gamma } ) }
=
\left[ \psi\!\left( 
t, e^{ - \eta t } X_t 
\right) \right]_{ \P, \mathcal{B}( H_{ \gamma } ) }
\\ & = 
\bigg[ \psi\!\left( 0, 
e^{  t ( A - \eta )}
e^{ - \eta \cdot 0 } X_0 
\right) 
+
 \int_0^t \mathbbm{1}_{\{ \int_0^t \|\psi_1( 
	u, 
	e^{  ( t - u )( A - \eta ) } 
	e^{ - \eta u } 
	X_u
	)\|_{H_{\gamma}} du < \infty\}}
	\,
	\psi_1\!\left( 
s, 
e^{  ( t - s )( A - \eta ) } 
e^{ - \eta s } 
X_s
\right)
ds  \\
& \quad + \int_0^t \mathbbm{1}_{\{ \int_0^t \|	\psi_2( 
	u, 
	e^{ ( t - u ) ( A - \eta ) } 
	e^{ - \eta u } 
	X_u ) \,
	e^{
		( t - u ) ( A - \eta )
	}
	e^{ - \eta u } 
	F( X_u )\|_{H_{\gamma}} du < \infty\}} 	\\
& \qquad \quad \cdot \psi_2\!\left( 
s, 
e^{ ( t - s ) ( A - \eta ) } 
e^{ - \eta s } 
X_s 
\right)
e^{
	( t - s ) ( A - \eta )
}
e^{ - \eta s } 
F( X_s ) \, ds \bigg]_{ \P, \mathcal{B}( H_{ \gamma } ) } \\
& \quad +
\int_0^t
\psi_2\!\left( 
s, 
e^{ ( t - s ) ( A - \eta )  } 
e^{ - \eta s } 
X_s
\right)
e^{
	( t - s ) ( A - \eta )
}
e^{ - \eta s } 
B( X_s ) \,
dW_s
.
\end{split}
\end{equation}
This ensures that for all $t \in [0,T]$ it holds that
\begin{equation}
\label{eq:trans:first}
\P\! \left( \int_0^t
\big\|
e^{ ( t - s ) ( A - \eta )  }
\eta
X_s
\big\|_{ H_{ \gamma } }
+
\big\|
	e^{
	 ( t - s ) ( A - \eta )
}
F( X_s )
\big\|_{ H_{ \gamma } }
+
\big\|
e^{
	 ( t - s ) ( A - \eta )
}
B( X_s )
\big\|_{ \mathrm{HS}( U, H_{ \gamma } ) }^2
\,
ds
< \infty \right)=1
\end{equation}
and
\begin{equation}
\label{eq:trans:last}
\begin{split}
&
\left[ X_t \right]_{ \P, \mathcal{B}( H_{ \gamma } ) }
=
\bigg[e^{  t( A - \eta ) } 
\xi  + \int_0^t \mathbbm{1}_{\{ \int_0^t
	\|
	e^{ ( t - u ) ( A - \eta )  }
	\eta
	X_u\|_{ H_{ \gamma } }
	du
	< \infty \}}
\,
e^{ ( t - s ) ( A - \eta )  }
\eta
X_s
\,
ds
\\ & \quad
+ \int_0^t \mathbbm{1}_{\{ \int_0^t
	\|
	e^{ ( t - u ) ( A - \eta )  }
	F(X_u)\|_{ H_{ \gamma } }
	du
	< \infty \}} \,
e^{ ( t - s ) ( A - \eta )  }
F(X_s)
\,
ds 
  \bigg]_{ \P, \mathcal{B}( H_{ \gamma } ) }
+
\int_0^t
e^{  ( t - s )( A - \eta ) }
B( X_s ) \,
dW_s
.
\end{split}
\end{equation}
Moreover, \eqref{eq:trans:first} shows for all $t \in [0,T]$ that
\begin{equation}
\begin{split}
&
\P \biggl(
\int_0^t
\big\|
e^{ ( t - s ) ( A - \eta )  }
\tilde{F}(X_s)\big\|_{ H_{ \gamma } } 
\,
ds
< \infty \biggr)
= 
\P \biggl(
\int_0^t
\bigl\|
e^{ ( t - s ) ( A - \eta )  } \big[
\eta X_s + F(X_s) \bigr]\big\|_{ H_{ \gamma } } \, ds < \infty \biggr) \\
& \geq
\P \biggl(
\int_0^t
\big\|
e^{ ( t - s ) ( A - \eta )  }
\eta
X_s\big\|_{ H_{ \gamma } }
+ \big\|
e^{ ( t - s ) ( A - \eta )  }  F(X_s) \big\|_{ H_{ \gamma } } \, ds < \infty \biggr)
= 1.
\end{split}
\end{equation}
Combining this with \eqref{eq:trans:first}--\eqref{eq:trans:last} completes the proof of Proposition~\ref{prop:transform_SG}.
\end{proof}

\subsection{Setting}\label{setting:example}
Consider the notation in Subsection~\ref{sec:notation},
let
$ T, c_0, \gamma, \theta, \vartheta  \in (0,\infty)$,
$ \alpha \in [ 0, \nicefrac{1}{2}]$,
$\varphi \in [0,1)$,
$\rho \in [0, \nicefrac{1}{4})$,
$\varrho \in (\rho, \nicefrac{1}{4})$,
$ \chi \in  (0, \nicefrac{( \varrho - \rho ) }{( 1 + \vartheta) }] $,
$( H, \left< \cdot , \cdot \right>_H, \left\| \cdot \right\|_H ) = (L^2(\lambda_{(0,1)}; \R), \langle \cdot , \cdot \rangle_{L^2(\lambda_{(0,1)}; \R)}, \left\| \cdot \right\|_{L^2(\lambda_{(0,1)}; \R)} )$,
let
$(e_n)_{n \in \N} \colon \N \to H $
and
$(\lambda_n)_{n \in \N} \colon \N \to (0, \infty)$
be the functions which satisfy
for all $ n \in \N $ that
$ e_n = [ (\sqrt{2} \sin(n \pi x) )_{x \in (0,1)}]_{\lambda_{(0,1)} , \mathcal{B}(\R)}$
and
$ \lambda_n = c_0 \pi^2 n^2 $,
let
$ A \colon D(A) \subseteq H \to H $
be the linear operator which satisfies
$ D(A) = \{ v \in H \colon \sum_{k = 1}^\infty | \lambda_k \langle e_k , v \rangle_H |^2 < \infty \} $
and
$ \forall \, v \in D(A) \colon A v = \sum_{k = 1}^\infty - \lambda_k \langle e_k , v \rangle_H e_k$,
let
$ ( H_r, \left< \cdot , \cdot \right>_{ H_r }, \left\| \cdot \right\|_{ H_r } ) $, $ r \in \R $,
be a family of interpolation spaces associated to $ -A $,
let $\xi \in H_{\nicefrac{1}{2}} $,
let
$F \in \mathcal{C}(H_{\varrho}, H_{-\alpha} )$,
$(P_n)_{n \in \N} \colon \N \to L(H_{-1}) $,
$(h_n)_{n \in \N} \colon \N \to (0, T] $,
and
$ \phi, \Phi \colon H_1 \to  [0,\infty) $
be functions which
satisfy
for all $ u \in H_1 $, $ n \in \N $, $ v, w \in P_n( H ) $ that
$\phi(u)= \gamma + \gamma \left[\sup\nolimits_{x \in (0,1)} |\und{u}(x)|^2 \right]$,
$\Phi(u)= \gamma + \gamma \left[\sup\nolimits_{x \in (0,1)} |\und{u}(x)|^{\gamma} \right]$,
$ P_n(u) = \sum_{k =1 }^n \langle e_k, u \rangle_H e_k $,
$ \limsup_{ m \to \infty} h_m =0$,
$ \left< v, P_n F( v + w ) \right>_H \leq \phi( w ) \| v \|^2_H + \varphi \| v \|^2_{ H_{ \nicefrac{1}{2} }} + \Phi( w ) $,
and
$ \left\| F(v) - F(w) \right\|_{ H_{ - \alpha } } \leq \theta \, ( 1 + \| v \|_{ H_{ \rho } }^{ \vartheta } + \|w\|_{H_{\rho}}^{\vartheta}) \, \|v-w\|_{H_{\rho}} $,
let $ ( \Omega, \F, \P ) $ be a probability space,
let $(W_t)_{t \in [0, T]}$ be an $\Id_H$-cylindrical $( \Omega, \F, \P )$-Wiener process,
and let
$ \Y^n, \mathcal{O}^n \colon [0, T] \times \Omega \to P_n(H)$, $ n \in \N$,
be stochastic processes which satisfy
for all $ n \in \N $, $ t \in [ 0, T ] $ that
$\left[\mathcal{O}_t^n \right]_{\P, \mathcal{B}(H)} = \int_0^t P_n \, e^{(t-s)A}  \, dW_s$
and
\begin{equation}\label{eq:set:abstract}
\P \Big( \Y_t^n = P_n \, e^{ t A } \, \xi + \smallint_0^t P_n \,  e^{  ( t - s ) A } \, \one_{ \{ \| \Y_{ \lf s \rf_{h_n} }^n \|_{ H_{\varrho} } + \| \mathcal{O}_{ \lf s \rf_{h_n} }^n +P_n \, e^{ \lf s \rf_{ h_n } A } \xi \|_{ H_{\varrho} } \leq | h_n|^{ - \chi } \}} \, F \big(  \Y_{ \lf s \rf_{ h_n } }^n \big) \, ds + \mathcal{O}_t^n \Big)=1. 
\end{equation}

\subsection{Properties of the stochastic convolution process}

The proof of the next result, Lemma~\ref{abs:phi(w)_finite} below, is a slight adaptation of the proof of Lemma~5.6 in Hutzenthaler et al.~\cite{Salimova2016}.

\begin{lemma}\label{abs:phi(w)_finite}
Assume the setting in Subsection~\ref{setting:example}, let $\beta \in (0, \nicefrac{1}{2}]$, $p \in (\nicefrac{1}{\beta}, \infty)$, $t \in [0,T]$, $n \in \N$, $\eta \in [0,\infty)$,
let
$ \mathbb{O} \colon \Omega \to P_n(H) $
be an $ \F $/$ \mathcal{B}(P_n(H)) $-measurable function which satisfies 
$\left[\mathbb{O} \right]_{\P, \mathcal{B}(H)} = \int_0^t P_n \, e^{(t-s)(A-\eta)}  \, dW_s$, and let $Y \colon \Omega \to \R$ be a standard normally distributed random variable. Then 
\begin{align}\label{abs:norm:eq}
\begin{split}
& \bigl(\E \! \left[\sup\nolimits_{x \in (0,1)} | \und{\mathbb{O}}(x) |^2 \right] \bigr)^{\nicefrac{1}{2}} \leq  \pi^{2}  \big( \E \big[|Y|^p \big] \big)^{\nicefrac{1}{p}} \left[ \sum_{k =1}^n \frac{ k^{4 \beta}}{\lambda_k+\eta}\right]^{\nicefrac{1}{2}} \\
& \quad \cdot \sup \! \Big(\Big\{ \! \sup\nolimits_{x \in (0,1)} |v(x)| \colon \big[v \in \mathcal{C}((0,1), \R) \text{ and } \|v\|_{\W^{\beta, p}((0,1), \mathbb{R})} \leq 1\big] \Big\}\Big)  < \infty.
\end{split}
\end{align}
\end{lemma}
\begin{proof}[Proof of Lemma~\ref{abs:phi(w)_finite}] 
First, observe that Jensen's inequality shows that
\begin{align*}\label{abs:phi(w)_eq1}
&\E \! \left[\sup\nolimits_{x \in (0,1)} | \und{\mathbb{O}}(x) |^2 \right] \\
&\leq \Big[\sup \! \Big(\Big\{ \! \sup\nolimits_{x \in (0,1)} |v(x)| \colon \big[v \in \mathcal{C}((0,1), \R) \text{ and } \|v\|_{\W^{\beta, p}((0,1), \mathbb{R})} \leq 1\big] \Big\}\Big)\Big]^2 \, \E  \big[\| \und{\mathbb{O}} \|_{\W^{\beta, p}((0,1), \mathbb{R})}^2 \big] \numberthis \\
& \leq \Big[\sup \!\Big(\Big\{ \! \sup\nolimits_{x \in (0,1)} |v(x)| \colon \big[v \in \mathcal{C}((0,1), \R) \text{ and } \|v\|_{\W^{\beta, p}((0,1), \mathbb{R})} \leq 1\big] \Big\}\Big)\Big]^2  \big( \E \big[\| \und{\mathbb{O}} \|_{\W^{\beta, p}((0,1), \mathbb{R})}^p \big]\big)^{\!\nicefrac{2}{p}}.
\end{align*}
In addition, note that
\begin{align}\label{abs:phi(w)_all}
\begin{split}
&\E \big[\| \und{\mathbb{O}} \|_{\W^{\beta, p}((0,1), \mathbb{R})}^p \big] = \E \! \left[ \int_0^1 |\und{\mathbb{O}} (x)|^p \, dx + \int_0^1 \int_0^1 \frac{|\und{\mathbb{O}} (x)- \und{\mathbb{O}}(y)|^p}{|x-y|^{1+ \beta p}} \, dx \, dy\right]\\
& = \E \big[|Y|^p \big] \int_0^1 \left( \E \! \left[ |\und{\mathbb{O}} (x)|^2 \right] \right)^{\nicefrac{p}{2}} dx  + \E \big[|Y|^p \big] \int_0^1 \int_0^1 \frac{\left( \E \! \left[ |\und{\mathbb{O}} (x) - \und{\mathbb{O}}(y) |^2 \right] \right)^{\nicefrac{p}{2}}}{|x-y|^{1+ \beta p}}\,  dx \, dy.
\end{split}
\end{align}
Furthermore, It\^o's isometry yields for all $ x \in (0,1)$  that
\begin{align}
\begin{split}
\E \! \left[ |\und{\mathbb{O}}(x)|^2 \right] & = \E  \Bigg[ \left| \sum\limits_{k =1}^n \und{e_k} (x) \int_0^t e^{-(\lambda_k+\eta)(t-s)}  \, d \! \left<  e_k, W_s \right>_H  \right|^2 \Bigg]= \sum\limits_{k =1}^n  |\und{e_k} (x)|^2 \int_0^t e^{-2(\lambda_k+\eta)(t-s)}  \, ds \\
& \leq \sum_{k =1}^n \frac{|\und{e_k}(x)|^2 }{ 2(\lambda_k+\eta)} \leq   \sum_{k =1}^n  \frac{1}{ \lambda_k+\eta}.
\end{split}
\end{align}
This implies that 
\begin{align}\label{abs:phi(w)_1}
\int_0^1\left( \E \! \left[ |\und{\mathbb{O}}(x)|^2 \right] \right)^{\nicefrac{p}{2}} dx \leq   \left[\sum_{k =1}^n  \frac{1}{\lambda_k+\eta}  \right]^{\nicefrac{p}{2}}.
\end{align}
Next note that again It\^o's isometry  ensures  for all  $ x, y \in (0,1)$  that
\begin{align}\label{eq:o_diff}
\begin{split}
\E \! \left[ |\und{\mathbb{O}} (x) - \und{\mathbb{O}}(y) |^2 \right] &=  \E  \Bigg[ \left| \sum\limits_{k =1}^n \left[\und{e_k}(x) - \und{e_k}(y)\right] \int_0^t e^{-(\lambda_k+\eta)(t-s)} \ d \! \left<  e_k, W_s \right>_H \right|^2 \Bigg] \\
&\leq  \sum_{k =1}^n  \frac{|\und{e_k}(x) - \und{e_k}(y)|^2 }{ 2(\lambda_k+\eta)}.
\end{split}
\end{align}
Moreover, the fact that $\beta \leq \nicefrac{1}{2}$ and  the fact that $\forall \, x, y \in \R \colon |\sin(x)-\sin(y)|\leq |x-y|$ prove that for all $x, y \in (0,1)$, $k \in \N$ it holds that
\begin{align}
\begin{split}
|\und{e_k}(x) - \und{e_k}(y)|^2 & =2 \, |\sin(k\pi x) - \sin(k \pi y)|^2 \\
&= 2 \, |\sin(k\pi x) - \sin(k \pi y)|^{2-4\beta} |\sin(k\pi x) - \sin(k \pi y)|^{4\beta}  \leq  2^{3-4\beta} |k\pi|^{4\beta}  |x-y|^{4\beta}.
\end{split}
\end{align}
This together with \eqref{eq:o_diff} establishes for all $x, y \in (0, 1)$ that
\begin{align}
\E \! \left[ |\und{\mathbb{O}} (x) - \und{\mathbb{O}}(y) |^2 \right] \leq 2^{2-4\beta} \,\pi^{4\beta} \,  |x-y|^{4 \beta}  \sum_{k =1}^n \frac{ k^{4 \beta} \, }{ \lambda_k+\eta}.
\end{align}
The fact that $\beta p \geq 1$ hence ensures   that
\begin{align}\label{abs:phi(w)_2}
\begin{split}
&\int_0^1\int_0^1 \frac{\left( \E \! \left[ |\und{\mathbb{O}} (x) - \und{\mathbb{O}}(y) |^2 \right] \right)^{\nicefrac{p}{2}}}{|x-y|^{1+ \beta p}} \,  dx \, dy \\
&\leq 2^{p(1-2\beta)} \pi^{2p\beta} \! \left[ \sum_{k =1}^n \frac{k^{4 \beta}  }{\lambda_k+\eta}\right]^{\nicefrac{p}{2}} \int_0^1 \int_0^1 |x-y|^{\beta p -1} \, dx \, dy   \leq 2^{p(1-2\beta)} \pi^{2p\beta} \! \left[ \sum_{k =1}^n \frac{k^{4 \beta}  }{\lambda_k+\eta}\right]^{\nicefrac{p}{2}}.
\end{split}
\end{align}
Combining this, \eqref{abs:phi(w)_all}, and \eqref{abs:phi(w)_1} proves that 
\begin{align}\label{abs:phi(w)_last}
\begin{split}
& \big( \E \big[\| \und{\mathbb{O}} \|_{\W^{\beta, p}((0,1), \mathbb{R})}^p \big]\big)^{\!\nicefrac{1}{p}}
\leq \big( \E \big[|Y|^p \big] \big)^{\nicefrac{1}{p}} \left\{  \left[\sum_{k =1}^n  \frac{1}{\lambda_k+\eta}  \right]^{\nicefrac{p}{2}} + 2^{p(1-2\beta)} \pi^{2p\beta}\! \left[ \sum_{k =1}^n \frac{k^{4 \beta}  }{\lambda_k+\eta}\right]^{\nicefrac{p}{2}} \right\}^{\!\nicefrac{1}{p}} \\
& \leq \big( \E \big[|Y|^p \big] \big)^{\nicefrac{1}{p}} \left\{   2^{p(1-2\beta)+1} \pi^{2p\beta}\! \left[ \sum_{k =1}^n \frac{k^{4 \beta}  }{\lambda_k+\eta}\right]^{\nicefrac{p}{2}} \right\}^{\!\nicefrac{1}{p}} \\
& \leq 2^{2-2\beta} \,\pi^{2\beta}  \big( \E \big[|Y|^p \big] \big)^{\nicefrac{1}{p}} \left[ \sum_{k =1}^n \frac{ k^{4 \beta}}{\lambda_k+\eta}\right]^{\nicefrac{1}{2}} \leq \pi^{2}  \big( \E \big[|Y|^p \big] \big)^{\nicefrac{1}{p}} \left[ \sum_{k =1}^n \frac{ k^{4 \beta}}{\lambda_k+\eta}\right]^{\nicefrac{1}{2}}.
\end{split}
\end{align}
In addition, note that the fact that $\beta p > 1$ and  the Sobolev embedding theorem  yield that
\begin{align}
\sup \! \Big(\Big\{ \! \sup\nolimits_{x \in (0,1)} |v(x)| \colon \big[v \in \mathcal{C}((0,1), \R) \text{ and } \|v\|_{\W^{\beta, p}((0,1), \mathbb{R})} \leq 1\big] \Big\}\Big) < \infty.
\end{align}
This, \eqref{abs:phi(w)_eq1}, and \eqref{abs:phi(w)_last} show \eqref{abs:norm:eq}. The proof of Lemma~\ref{abs:phi(w)_finite} is thus completed.
\end{proof}


\begin{lemma}
	\label{lem:conv:series}
Let $\alpha \in \R$, $\beta \in (1+ \alpha, \infty)$. Then it holds that
\begin{align}
\limsup_{\eta \to \infty} \left( \sum_{k = 1}^\infty \frac{k^{\alpha}}{k^{\beta} + \eta} \right) = 0.
\end{align}
\end{lemma}
\begin{proof}[Proof of Lemma~\ref{lem:conv:series}]
Observe that for all $\eta \in [0, \infty)$, $k \in \N$ it holds that
\begin{align}
\frac{k^{\alpha}}{k^{\beta} + \eta} \leq \frac{1}{k^{\beta -\alpha}} \qquad \text{and} \qquad \sum_{n = 1}^\infty \frac{1}{n^{\beta - \alpha} } < \infty.
\end{align}
Lebesgue's theorem of dominated convergence hence ensures that
\begin{align}
\limsup_{\eta \to \infty} \left( \sum_{k = 1}^\infty \frac{k^{\alpha}}{k^{\beta} + \eta} \right) = \sum_{k = 1}^\infty \limsup_{\eta \to \infty} \left( \frac{k^{\alpha}}{k^{\beta} + \eta} \right) = 0.
\end{align} 
The proof of Lemma~\ref{lem:conv:series} is thus completed.
\end{proof}
\pagebreak

\begin{prop}\label{prop:abstract}
Assume the setting in Subsection~\ref{setting:example}, let  $\beta \in (0, \nicefrac{1}{4})$, $p \in (\nicefrac{1}{\beta}, \infty)$,  $\eta \in [0,\infty)$, let $\tilde{ \mathcal{O} }^n, \mathbb{O}^n \colon [0,T] \times \Omega \to P_n(H)$, $n \in \N$, be stochastic processes with continuous sample paths which satisfy for all $n \in \N$, $t \in [0,T]$ that $[\tilde{\mathcal{O}}_t^n ]_{\P, \mathcal{B}(H)} = \int_0^t P_n \, e^{(t-s)A}  \, dW_s$ and $\left[\mathbb{O}_t^n \right]_{\P, \mathcal{B}(H)} = \int_0^t P_n \, e^{(t-s)(A-\eta)}  \, dW_s$, and assume
\begin{align}\label{abstract:gamma}
\begin{split}
720  p^3 T  \gamma \pi^{4}  \left[ \sum_{k = 1}^\infty \frac{ k^{4 \beta}}{\lambda_k+\eta}\right]  \Big[\sup \! \Big(\Big\{ \! \sup\nolimits_{x \in (0,1)} |v(x)| \colon \big[v \in \mathcal{C}((0,1), \R) \text{ and } \|v\|_{\W^{\beta, p}((0,1), \mathbb{R})} \leq 1\big] \Big\}\Big)\Big]^2 \leq 1.
\end{split}
\end{align}
Then
\begin{enumerate}[(i)]
\item
\label{item:finite1}
it holds that $ \sup_{ n \in \N } \sup_{ s \in [0,T]} \E[ \| \mathbb{O}_s^n + P_n \, e^{s(A-\eta)}  \xi \|_H^p] < \infty$ and
\item
it holds that
\begin{align}
& \nonumber \sup_{ n \in \N } \E\biggl[ \int_0^T \exp \left( \smallint_s^T p\,\phi\big( \mathbb{O}_{\lf u \rf_{h_n} }^n +  P_n \, e^{\lf u \rf_{h_n} (A-\eta)} \xi\big) \, du \right) \max\Big\{ 1, \big|\Phi(\mathbb{O}_{ \lf s \rf_{h_n} }^n + P_n \, e^{\lf s \rf_{h_n} (A-\eta)} \xi)\big|^{\nicefrac{p}{2}}, \\
& \quad  \big\|\mathbb{O}_s^n+ P_n \, e^{s (A-\eta)} \xi\big\|_H^p, \smallint\nolimits_{0}^T \big\| \tilde{\mathcal{O}}_u^n+ P_n \, e^{uA} \xi \big\|_{H_{\varrho}}^{2p+ 2p\vartheta} \, du \Big\} \, ds \biggr]< \infty.
\end{align}
\end{enumerate}
\end{prop}
\begin{proof}[Proof of Proposition~\ref{prop:abstract}]
First, note that the Burkholder-Davis-Gundy inequality proves that
for all standard normally distributed random variables $ Y \colon \Omega \to \R $ it holds that
\begin{equation}
\big( \E \big[|Y|^p \big] \big)^{\nicefrac{1}{p}}
\leq
\sqrt{\tfrac{p(p-1)}{2}}
\leq p.
\end{equation}
Markov's inequality, Lemma~\ref{abs:phi(w)_finite}, and \eqref{abstract:gamma} hence imply that for all $n \in \N$, $t \in [0,T]$ it holds that
\begin{align}
\begin{split}
& \P \! \left( \sup\nolimits_{x \in (0,1)} | \und{\mathbb{O}_t^n}(x) |^2 \geq \frac{1}{72 p  T  \gamma} \right) \leq 72  p T  \gamma \,  \E \! \left[\sup\nolimits_{x \in (0,1)} | \und{\mathbb{O}_t^n}(x) |^2  \right]  \leq 72  p^3 T  \gamma \pi^{4}  \left[ \sum_{k =1}^n \frac{ k^{4 \beta}}{\lambda_k+\eta}\right]\\
& \quad \cdot \Big[\sup \! \Big(\Big\{ \! \sup\nolimits_{x \in (0,1)} |v(x)| \colon \big[v \in \mathcal{C}((0,1), \R) \text{ and } \|v\|_{\W^{\beta, p}((0,1), \mathbb{R})} \leq 1\big] \Big\}\Big)\Big]^2 \leq \frac{1}{10}.
\end{split}
\end{align}
This and Proposition~\ref{thm:fernique} (with $V = P_n(H)$, $ \left\|\cdot \right\|_V = (P_n(H) \ni v \mapsto \sup_{x \in (0,1)} |\und{v}(x)| \in [ 0, \infty ) )$, $X= \mathbb{O}_t^n$, $R= (72pT\gamma)^{\nicefrac{-1}{2}}$ for $t \in [0,T]$, $n \in \N$ in the notation of Proposition~\ref{thm:fernique}) show that for all $n \in \N$, $t \in [0, T]$ it holds that
\begin{align}\label{eq:abstract:13}
\E \! \left[ \exp \! \left(4 p T \gamma  \Big\{ \! \sup\nolimits_{x \in (0,1)} | \und{\mathbb{O}_t^n}(x) |^2\Big\} \right)\right] \leq 13.
\end{align}
In addition, H\"older's inequality yields for all  $n \in \N$ that
\begin{align*}\label{eq:abstract:holder}
&\E\bigg[ \int_0^T \exp \left( \smallint_s^T p\,\phi\big( \mathbb{O}_{\lf u \rf_{h_n} }^n +  P_n \, e^{\lf u \rf_{h_n} (A-\eta)} \xi\big) \, du \right) \max\Big\{ 1, \big|\Phi(\mathbb{O}_{ \lf s \rf_{h_n} }^n + P_n \, e^{\lf s \rf_{h_n} (A-\eta)} \xi)\big|^{\nicefrac{p}{2}}, \\
& \big\|\mathbb{O}_s^n+ P_n \, e^{s (A-\eta)} \xi\big\|_H^p, \smallint\nolimits_{0}^T \big\| \tilde{\mathcal{O}}_u^n+ P_n \, e^{uA} \xi \big\|_{H_{\varrho}}^{2p+ 2p\vartheta} \, du \Big\} \, ds \bigg]^2 \\
& = \bigg( \int_0^T  \E\bigg[\! \exp \left( \smallint_s^T p\,\phi\big( \mathbb{O}_{\lf u \rf_{h_n} }^n +  P_n \, e^{\lf u \rf_{h_n} (A-\eta)} \xi\big) \, du \right) \max\Big\{ 1, \big|\Phi(\mathbb{O}_{ \lf s \rf_{h_n} }^n + P_n \, e^{\lf s \rf_{h_n} (A-\eta)} \xi)\big|^{\nicefrac{p}{2}}, \\
& \quad  \big\|\mathbb{O}_s^n+ P_n \, e^{s (A-\eta)} \xi\big\|_H^p, \smallint\nolimits_{0}^T \big\| \tilde{\mathcal{O}}_u^n+ P_n \, e^{uA} \xi \big\|_{H_{\varrho}}^{2p+ 2p\vartheta} \, du \Big\} \bigg] \, ds \bigg)^2 \\
& \leq T  \int_0^T  \E\bigg[\! \exp \left( \smallint_s^T p\,\phi\big( \mathbb{O}_{\lf u \rf_{h_n} }^n +  P_n \, e^{\lf u \rf_{h_n} (A-\eta)} \xi\big) \, du \right) \max\Big\{ 1, \big|\Phi(\mathbb{O}_{ \lf s \rf_{h_n} }^n + P_n \, e^{\lf s \rf_{h_n} (A-\eta)} \xi)\big|^{\nicefrac{p}{2}}, \numberthis \\
& \quad  \big\|\mathbb{O}_s^n+ P_n \, e^{s (A-\eta)} \xi\big\|_H^p, \smallint\nolimits_{0}^T \big\| \tilde{\mathcal{O}}_u^n+ P_n \, e^{uA} \xi \big\|_{H_{\varrho}}^{2p+ 2p\vartheta} \, du \Big\} \bigg]^2 \, ds \\
& \leq T  \int_0^T  \E\bigg[\! \exp \left( \smallint_s^T 2p \,\phi\big( \mathbb{O}_{\lf u \rf_{h_n} }^n +  P_n \, e^{\lf u \rf_{h_n} (A-\eta)} \xi\big) \, du \right) \bigg] \E \bigg[ \! \max\Big\{ 1, \big|\Phi(\mathbb{O}_{ \lf s \rf_{h_n} }^n + P_n \, e^{\lf s \rf_{h_n} (A-\eta)} \xi)\big|^{p}, \\
& \quad  \big\|\mathbb{O}_s^n+ P_n \, e^{s (A-\eta)} \xi\big\|_H^{2p}, T \smallint\nolimits_{0}^T \big\| \tilde{\mathcal{O}}_u^n+ P_n \, e^{uA} \xi \big\|_{H_{\varrho}}^{4p+ 4p\vartheta} \, du \Big\} \bigg] \, ds \\
& \leq T \, \E\bigg[ \! \exp \left( \smallint_0^T 2p \,\phi\big( \mathbb{O}_{\lf u \rf_{h_n} }^n +  P_n \, e^{\lf u \rf_{h_n} (A-\eta)} \xi\big) \, du \right) \bigg] \int_0^T  \E \bigg[ 1 +  \big|\Phi(\mathbb{O}_{ \lf s \rf_{h_n} }^n + P_n \, e^{\lf s \rf_{h_n} (A-\eta)} \xi)\big|^{p} \\
&  \quad   + \big\|\mathbb{O}_s^n+ P_n \, e^{s (A-\eta)} \xi\big\|_H^{2p} + T \smallint\nolimits_{0}^T \big\| \tilde{\mathcal{O}}_u^n+ P_n \, e^{uA} \xi \big\|_{H_{\varrho}}^{4p+ 4p\vartheta} \, du \bigg] \, ds.
\end{align*}
Next note that the fact that $ \forall \, x, y \in \R \colon |x+y|^2 \leq 2 x^2 + 2y^2$ ensures that for all $n \in \N$ it holds that
\begin{align}\label{eq:abstract:phi1}
\begin{split}
& \E \!\left[ \exp \! \left( \smallint_0^T  2p \,\phi\big(\mathbb{O}_{\lf u \rf_{h_n}}^n +  P_n \, e^{\lf u \rf_{h_n} (A-\eta)} \xi\big) \, du \right) \right] \\
&= \E \!\left[ \exp \! \left( \smallint_0^T  2p\gamma +  2p\gamma  \Big\{ \!\sup\nolimits_{x \in (0,1)} \big|\und{\mathbb{O}_{\lf u \rf_{h_n}}^n} \!(x) + \und{ P_n \, e^{\lf u \rf_{h_n} (A-\eta)} \xi}(x)\big|^2 \Big\} \, du \right) \right] \\
& \leq \exp \!\left(   2p\gamma T + 4p\gamma \smallint_0^T  \Big\{ \!\sup\nolimits_{x \in (0,1)} |\und{ P_n \, e^{\lf u \rf_{h_n} (A-\eta)} \xi}(x)|^2 \Big\} \, du \right) \\
& \quad \cdot \E \!\left[ \exp \! \left( \smallint_0^T    4p\gamma  \Big\{ \!\sup\nolimits_{x \in (0,1)} \big|\und{\mathbb{O}_{\lf u \rf_{h_n}}^n} \!(x)\big|^2 \Big\} \, du \right) \right].
\end{split}
\end{align}
Furthermore, e.g., Lemma~2.22 in Cox, Hutzenthaler, \& Jentzen~\cite{CoxHutzenthalerJentzen2014}  and  \eqref{eq:abstract:13} prove  for all $n \in \N$ that
\begin{equation}\label{abstract:prop:jen}
\begin{split}
& \E \! \left[ \exp \! \left(  \smallint_0^T 4 p \gamma \Big\{\!\sup\nolimits_{x \in (0,1)} \big|\und{\mathbb{O}_{\lf u \rf_{h_n}}^n} \! (x)\big|^2 \Big\} \,  du \right)\right] \\\
& \leq \frac{1}{T}  \int_0^T \E\! \left[ \exp \! \left(  4  p T \gamma \Big\{\!\sup\nolimits_{x \in (0,1)} \big|\und{\mathbb{O}_{\lf u \rf_{h_n}}^n}\!(x)\big|^2 \Big\} \right)\right] du \leq 13 .
\end{split}
\end{equation} 
This and \eqref{eq:abstract:phi1} show  for all $n \in \N$  that
\begin{align}
\label{eq:abs:phi0}
\begin{split}
& \E \!\left[ \exp \! \left( \smallint_0^T  2p \,\phi\big(\mathbb{O}_{\lf u \rf_{h_n}}^n +  P_n \, e^{\lf u \rf_{h_n} (A-\eta)} \xi\big) \, du \right) \right] \\
& \leq 13 \exp \!\left(   2p\gamma T + 4p\gamma \smallint_0^T  \Big\{ \!\sup\nolimits_{x \in (0,1)} |\und{ P_n \, e^{\lf u \rf_{h_n} (A-\eta)} \xi}(x)|^2 \Big\} \, du \right).
\end{split}
\end{align}
In addition, the Sobolev embedding theorem ensures that 
\begin{align}
\sup \! \Big(\Big\{\! \sup\nolimits_{x \in (0,1)} |\und{v}(x)| \colon \big[ v \in H_{\nicefrac{1}{2}} \text{ and } \|v\|_{H_{\nicefrac{1}{2}}} \leq 1\big] \Big\}\Big) < \infty.
\end{align}
This establishes for all $n \in \N$, $s \in [0, T]$  that
\begin{align}
\label{eq:abs:xi}
\begin{split}
&\sup\nolimits_{x \in (0,1)} | \und{P_n \, e^{s(A-\eta)} \xi} (x) | \\
&\leq  \Big[\sup \! \Big(\Big\{ \!\sup\nolimits_{x \in (0,1)} |\und{v}(x)| \colon \big[v \in H_{\nicefrac{1}{2}} \text{ and } \|v\|_{H_{\nicefrac{1}{2}}} \leq 1\big] \Big\}\Big) \Big]  \|P_n \, e^{s(A-\eta)}  \xi \|_{H_{\nicefrac{1}{2}}}  \\
& \leq \Big[\sup \! \Big(\Big\{ \sup\nolimits_{x \in (0,1)} |\und{v}(x)| \colon \big[v \in H_{\nicefrac{1}{2}} \text{ and } \|v\|_{H_{\nicefrac{1}{2}}} \leq 1\big] \Big\}\Big) \Big] \| \xi \|_{H_{\nicefrac{1}{2}}} < \infty.
\end{split}
\end{align}
Combing this with \eqref{eq:abs:phi0} implies that 
\begin{align}\label{eq:abstract:phi}
\sup_{ n \in \N } \E \!\left[ \exp \! \left( \smallint_0^T  2p \,\phi\big(\mathbb{O}_{\lf u \rf_{h_n}}^n +  P_n \, e^{\lf u \rf_{h_n} (A-\eta)} \xi\big) \, du \right) \right] < \infty.
\end{align}
Next note that the fact that $ \forall \, x, y \in \R, \, a \in [0, \infty) \colon |x+y|^a \leq 2^{\max\{a-1,0\}} |x|^a + 2^{\max\{a-1,0\}} |y|^a$
and the triangle inequality show that for all $n \in \N$, $s \in [0, T]$ it holds that
\begin{align}
\label{eq:abstract:Phi0}
\begin{split}
& \E \! \left[ \big| \Phi\big( \mathbb{O}_{ \lf s \rf_{h_n } }^n + P_n \, e^{\lf s \rf_{h_n} (A-\eta)} \xi\big) \big|^p \right]\\
&=  \E  \bigg[ \left|\gamma + \gamma \Big\{ \!\sup\nolimits_{x \in (0,1)} \big|\und{\mathbb{O}_{\lf s \rf_{h_n}}^n} \! (x) +  \und{ P_n \,e^{\lf s \rf_{h_n} (A-\eta)} \xi}(x)\big|^{\gamma} \Big\} \right|^p \bigg] \\
& \leq  \E \! \left[ 2^{p-1} \gamma^p + 2^{p-1} \gamma^p \Big\{ \!\sup\nolimits_{x \in (0,1)} \big|\und{\mathbb{O}_{\lf s \rf_{h_n}}^n} \!(x) +  \und{ P_n \,e^{\lf s \rf_{h_n} (A-\eta)} \xi}(x)\big|^{p\gamma} \Big\} \right] \\
& \leq  \E \! \left[ 2^{p-1} \gamma^p + 2^{p-1} \gamma^p \, 2^{\max\{p\gamma-1,0\}} \Big\{ \!\sup\nolimits_{x \in (0,1)} \big|\und{\mathbb{O}_{\lf s \rf_{h_n}}^n} \!(x)\big|^{p\gamma}+ \!\sup\nolimits_{x \in (0,1)} \big|\und{ P_n \,e^{\lf s \rf_{h_n} (A-\eta)} \xi}(x)\big|^{p\gamma} \Big\} \right] \\
& \leq  2^{p-1} \gamma^p +  2^{p(\gamma+1)-1} \gamma^p \, \E \! \left[  \Big\{ \!\sup\nolimits_{x \in (0,1)} \big|\und{\mathbb{O}_{\lf s \rf_{h_n}}^n} \!(x)\big|^{p\gamma} \Big\}+ \Big\{ \!\sup\nolimits_{x \in (0,1)} \big|\und{ P_n \,e^{\lf s \rf_{h_n} (A-\eta)} \xi}(x)\big|^{p\gamma} \Big\} \right] .
\end{split}
\end{align}
Furthermore, observe that, e.g., Lemma~5.7 in Hutzenthaler et al.~\cite{Salimova2016} (with $a=4pT\gamma$, $x=\sup\nolimits_{x \in (0,1)} |\und{\mathbb{O}_{s}^n (\omega)}(x)|^2$, $r=\nicefrac{r}{2}$ for $\omega \in \Omega$, $s \in [0, T]$, $n \in \N$, $ r \in [ 0, \infty ) $ in the notation of Lemma~5.7 in Hutzenthaler et al.~\cite{Salimova2016})
and \eqref{eq:abstract:13}
ensure that for all $ r \in [ 0, \infty ) $, $n \in \N$, $s \in [0, T]$ it holds that
\begin{equation}
\label{eq:abstract:estimate-exponential}
\E \! \left[\Big\{ \!\sup\nolimits_{x \in (0,1)} \big|\und{\mathbb{O}_{s}^n}(x)\big|^{r} \Big\}\right]
\leq
\tfrac{(\lf \nicefrac{r}{2} \rf_1 +1)!}{ |4pT\gamma|^{\nicefrac{r}{2}}} \,
\E \! \left[\exp \! \left( 4pT\gamma\Big\{ \!\sup\nolimits_{x \in (0,1)} \big|\und{\mathbb{O}_{s}^n}(x)\big|^{2} \Big\} \right)\right]
\leq
\tfrac{ 13 \, (\lf \nicefrac{r}{2} \rf_1 +1)!}{ |4pT\gamma|^{\nicefrac{r}{2}}}.
\end{equation}
Combining this and \eqref{eq:abstract:Phi0} proves for all $n \in \N$, $s \in [0, T]$ that
\begin{align}
\begin{split}
&\E \! \left[ \big| \Phi\big( \mathbb{O}_{ \lf s \rf_{h_n } }^n + P_n \, e^{\lf s \rf_{h_n} (A-\eta)} \xi\big) \big|^p \right] \\
&\leq 2^{p-1} \gamma^p+  2^{p(\gamma+1)-1} \gamma^p  \Big\{ \!\sup\nolimits_{x \in (0,1)} \big|\und{ P_n \,e^{\lf s \rf_{h_n} (A-\eta)} \xi}(x)\big|^{p\gamma} \Big\} +\tfrac{13 \cdot 2^{p(\gamma+1)-1} \gamma^p (\lf \nicefrac{p\gamma}{2} \rf_1 +1)!}{ |4pT\gamma|^{\nicefrac{p\gamma}{2}}} .
\end{split}
\end{align}
This together with \eqref{eq:abs:xi} yields that 
\begin{align}\label{eq:abstract:Phi}
\begin{split}
\sup_{ n \in \N } \int_0^T \E \! \left[ \big| \Phi\big( \mathbb{O}_{ \lf s \rf_{h_n } }^n + P_n \, e^{\lf s \rf_{h_n} (A-\eta)} \xi\big) \big|^p \right] ds < \infty.
\end{split}
\end{align}
Moreover, \eqref{eq:abstract:estimate-exponential} establishes for all $ r \in [ 1, \infty ) $, $n \in \N$, $s \in [0, T]$ that
\begin{align}
\label{eq:abstract:O0}
\begin{split}
\E \!\left[\big\|\mathbb{O}_s^n+ P_n \, e^{s (A-\eta)} \xi\big\|_H^{r} \right] & \leq  \E\! \left[ 2^{r-1}  \|\mathbb{O}_s^n\|_H^{r} + 2^{r-1} \| P_n \, e^{s (A-\eta)} \xi\|_H^{r} \right]\\
& \leq 2^{r-1} \, \E \! \left[\Big\{ \!\sup\nolimits_{x \in (0,1)} \big|\und{\mathbb{O}_{s}^n}(x)\big|^{r} \Big\}\right] + 2^{r-1} \| P_n \, e^{s (A-\eta)} \xi\|_H^{r}\\
& \leq  \tfrac{13 \cdot 2^{r-1}  (\lf \nicefrac{r}{2} \rf_1 +1)!}{ |4pT\gamma|^{\nicefrac{r}{2}}}  + 2^{r-1} \|\xi\|_H^{r}.
\end{split}
\end{align}
Observe that this implies that 
\begin{equation}
\sup_{ n \in \N }  \sup_{s \in [0,T]} \E \big[ \big\|  \mathbb{O}_s^n  + P_n \, e^{s(A-\eta)}  \xi  \big\|_H^p \big]
\leq \tfrac{13 \cdot 2^{p-1}  (\lf \nicefrac{p}{2} \rf_1 +1)!}{ |4pT\gamma|^{\nicefrac{p}{2}}}  + 2^{p-1} \|\xi\|_H^{p}
< \infty.
\end{equation}
This proves \eqref{item:finite1}.
In addition, \eqref{eq:abstract:O0} shows that 
\begin{align}\label{eq:abstract:O}
\begin{split}
\sup_{ n \in \N } \int_0^T \E \! \left[ \big\|\mathbb{O}_s^n+ P_n \, e^{s (A-\eta)} \xi\big\|_H^{2p} \right] ds < \infty.
\end{split}
\end{align}
In the next step note that for all $n \in \N$ it holds that
\begin{align}
\label{eq:abstract:int}
\begin{split}
\E \! \left[\smallint_{0}^T \big\| \tilde{\mathcal{O}}_u^n+ P_n \, e^{uA} \xi \big\|_{H_{\varrho}}^{4p+ 4p\vartheta} \, du  \right] &\leq 2^{4p+ 4p\vartheta-1} \, \E \!\left[\smallint_{0}^T \big\| \tilde{\mathcal{O}}_u^n  \big\|_{H_{\varrho}}^{4p+ 4p\vartheta}+ \|P_n \, e^{uA} \xi \|_{H_{\varrho}}^{4p+ 4p\vartheta} \, du  \right] \\
&\leq 2^{4p+ 4p\vartheta-1} \, \E \!\left[\smallint_{0}^T \big\| \tilde{\mathcal{O}}_u^n  \big\|_{H_{\varrho}}^{4p+ 4p\vartheta}+ \| \xi \|_{H_{\varrho}}^{4p+ 4p\vartheta} \, du  \right].
\end{split}
\end{align}
Furthermore, observe that the Burkholder-Davis-Gundy-type inequality in Da~Prato \& Zabczyk~\cite[Lemma~7.7]{dz92} proves  for all  $n \in \N$, $u \in [0, T]$  that
\begin{align}
\begin{split}
\E \!\left[ \big\| \tilde{\mathcal{O}}_u^n  \big\|_{H_{\varrho}}^{4p+ 4p\vartheta} \right] & = \E \!\left[\left\|  \int_0^u P_n \, e^{(u-s)A}  \, dW_s \right\|_{H_{\varrho}}^{4p+ 4p\vartheta} \right]  \\
&\leq \left[\tfrac{(4p+ 4p\vartheta)(4p+ 4p\vartheta-1)}{2}\right]^{2p+ 2p\vartheta} \left[ \int_0^u \|P_n\, e^{(u-s)A} \|_{\mathrm{HS}(H, H_{\varrho})}^2 \,ds\right]^{2p+ 2p\vartheta} \\
& \leq \left[4p+ 4p\vartheta  \right]^{4p+ 4p\vartheta} \left[ \int_0^u \| (-A)^{\varrho} \,e^{(u-s)A} \|_{\mathrm{HS}(H)}^2 \, ds\right]^{2p+ 2p\vartheta}\\
& = \left[4p+ 4p\vartheta \right]^{4p+ 4p\vartheta} \left[ \sum_{k = 1}^\infty \int_0^u (\lambda_k)^{2\varrho} \,e^{-2\lambda_k s} \, ds\right]^{2p+ 2p\vartheta}\\
& = \left[4p+ 4p\vartheta\right]^{4p+ 4p\vartheta} \left[ \sum_{k = 1}^\infty  \frac{(\lambda_k)^{2\varrho} (1-e^{-2\lambda_k u})}{2\lambda_k} \right]^{2p+ 2p\vartheta}\\
& \leq \left[4p+ 4p\vartheta \right]^{4p+ 4p\vartheta} \left[ \sum_{k = 1}^\infty  (\lambda_k)^{2\varrho-1}  \right]^{2p+ 2p\vartheta} < \infty.
\end{split}
\end{align}
Combining this with \eqref{eq:abstract:int}  yields that
\begin{align}
\begin{split}
\sup_{ n \in \N } \E \! \left[ \smallint_{0}^T \big\| \tilde{\mathcal{O}}_u^n+ P_n \, e^{uA} \xi \big\|_{H_{\varrho}}^{4p+ 4p\vartheta} \, du \right] < \infty.
\end{split}
\end{align}
This, \eqref{eq:abstract:holder}, \eqref{eq:abstract:phi}, \eqref{eq:abstract:Phi}, and \eqref{eq:abstract:O} ensure that
\begin{align}\label{eq:abstract:last}
\begin{split}
&\sup_{ n \in \N } \E\biggl[ \int_0^T \exp \left( \smallint_s^T p\,\phi\big( \mathbb{O}_{\lf u \rf_{h_n} }^n +  P_n \, e^{\lf u \rf_{h_n} (A-\eta)} \xi\big) \, du \right) \max\Big\{ 1, \big|\Phi(\mathbb{O}_{ \lf s \rf_{h_n} }^n + P_n \, e^{\lf s \rf_{h_n} (A-\eta)} \xi)\big|^{\nicefrac{p}{2}}, \\
& \quad  \big\|\mathbb{O}_s^n+ P_n \, e^{s (A-\eta)} \xi\big\|_H^p, \smallint\nolimits_{0}^T \big\| \tilde{\mathcal{O}}_u^n+ P_n \, e^{uA} \xi \big\|_{H_{\varrho}}^{2p+ 2p\vartheta} \, du \Big\} \, ds \biggr]< \infty.
\end{split}
\end{align}
The proof of Proposition~\ref{prop:abstract} is thus completed.
\end{proof}


The proof of the next elementary result, Lemma~\ref{lemma:conv:rate}, is a slight adaptation of the proof of Lemma~5.9 in Hutzenthaler et al.~\cite{Salimova2016}.

\begin{lemma}\label{lemma:conv:rate}
Assume the setting in Subsection~\ref{setting:example}, let  $p \in [2, \infty)$, $n \in \N$, $\varepsilon \in [0, \nicefrac{1}{4} -\varrho)$, and  let $O \colon  [0, T] \times \Omega \to H_{\varrho}$  be a  stochastic process which satisfies for all $t \in [0, T]$ that $ [O_t ]_{\P, \mathcal{B}(H)} =  \int_0^t e^{(t-s)A}  \, dW_s$. Then 
\begin{align}
\sup_{t \in [0, T]} \bigl( \E \bigl[ \|O_t  - \mathcal{O}_t^n \|_{ H_{\varrho} }^p \bigr] \bigr)^{ \nicefrac{1}{p} }
\leq
\left[\tfrac{p(p-1)}{4 ( c_0 \pi^2 )^{2\varepsilon} } \sum_{k = 1}^\infty  (\lambda_k)^{2\varrho+2\varepsilon-1}\right]^{\nicefrac{1}{2}} n^{-2\varepsilon} < \infty.
\end{align}
\end{lemma}
\begin{proof}[Proof of Lemma~\ref{lemma:conv:rate}]
First, note that the Burkholder-Davis-Gundy-type inequality in Da~Prato \& Zabczyk~\cite[Lemma~7.7]{dz92} shows for all  $t \in [0, T]$   that
\begin{align}\label{eq:burk}
\begin{split}
\bigl( \E \bigl[ \|O_t  - \mathcal{O}_t^n \|_{ H_{\varrho} }^p \bigr] \bigr)^{ \nicefrac{1}{p} }
& =
\biggl( \E \biggl[
    \biggl\| \int_0^t (\mathrm{Id}_{H_{\varrho}}- P_n) \, e^{(t-s)A} \,  d W_s \biggr\|_{ H_{\varrho} }^p
\biggr] \biggr)^{ \nicefrac{1}{p} } \\
& \leq
\left[ \tfrac{p(p-1)}{2} \int_0^t \big \|(\mathrm{Id}_{H_{\varrho}}- P_n) \,  e^{(t-s)A} \big\|^2_{\mathrm{HS}(H, H_{\varrho})} \, ds \right]^{\nicefrac{1}{2}}.
\end{split}
\end{align}
Next observe that for all  $ t \in [0, T] $ it holds that
\begin{align}
\begin{split}
\int_0^t \big \|(\mathrm{Id}_{H_{\varrho}}- P_n) \, e^{(t-s)A} \big\|^2_{\mathrm{HS}(H, H_{\varrho})} \, ds
& \leq  \int_0^t  \|\mathrm{Id}_{H_{\varrho+\varepsilon}}- P_n|_{H_{\varrho+\varepsilon}} \|^2_{L(H_{\varrho+\varepsilon}, H_{\varrho})} \, \| e^{(t-s)A} \|^2_{\mathrm{HS}(H, H_{\varrho+\varepsilon})} \, ds\\ 
& = \| (  - A )^{ - \varepsilon } ( \Id_H - P_n |_H ) \|^2_{ L( H ) } \int_0^t \| e^{sA} \|^2_{\mathrm{HS}(H, H_{\varrho+\varepsilon})} \, ds\\
& =  \| ( - A )^{ - 1} ( \Id_H - P_n |_H ) \|^{2\varepsilon}_{ L( H ) } \sum_{k = 1}^\infty \int_0^t (\lambda_k)^{2\varrho+2\varepsilon} e^{-2\lambda_k s} \, ds \\ 
& \leq |\lambda_{n+1}|^{-2\varepsilon} \sum_{k = 1}^\infty \frac{ (\lambda_k)^{2\varrho+2\varepsilon}}{2\lambda_k } \leq \tfrac{1}{2} ( c_0 \pi^2 n^2 )^{-2\varepsilon} \sum_{k = 1}^\infty (\lambda_k)^{2\varrho+2\varepsilon-1}.
\end{split}
\end{align}
This  and \eqref{eq:burk} ensure that for all  $t \in [0, T]$ it holds that
\begin{align}
\bigl( \E \bigl[ \|O_t  - \mathcal{O}_t^n \|_{ H_{\varrho} }^p \bigr] \bigr)^{ \nicefrac{1}{p} }
\leq
\left[\tfrac{p(p-1)}{4(c_0 \pi^2)^{2\varepsilon}} \sum_{k = 1}^\infty  (\lambda_k)^{2\varrho+2\varepsilon-1} \right]^{\nicefrac{1}{2}} n^{-2\varepsilon} < \infty.
\end{align}
The proof of Lemma~\ref{lemma:conv:rate} is thus completed.
\end{proof}

Some of the arguments in the proof of the next result, Proposition~\ref{prop:exists} below, are similar to the arguments in the proof of Corollary~5.10 in Hutzenthaler et al.~\cite{Salimova2016}.

\begin{prop}
\label{prop:exists}
Assume the setting in Subsection~\ref{setting:example} and let $p \in (0, \infty)$. Then there exist a  real number $\eta \in [0,\infty)$ and stochastic processes
$ O \colon [0, T] \times \Omega \to H_{\varrho}$
and $ \tilde{\mathcal{O}}^n, \mathbb{O}^n \colon [0, T] \times \Omega \to P_n(H)$, $n \in \N$,
with continuous sample paths
such that
\begin{enumerate}[(i)]
	\item \label{item:O}
	it holds for all $t \in [0, T]$ that
	$[O_t]_{\P, \mathcal{B}(H)} =  \int_0^{t}  e^{(t-s)A} \,  dW_s$,
	\item \label{item:mathcal:O}
	it holds for all $n \in \N$, $t \in [0, T]$ that
	$[\tilde{\mathcal{O}}^n_t]_{\P, \mathcal{B}(H)} =  \int_0^{t} P_n \, e^{(t-s)A} \, dW_s$,
	\item \label{item:mathbb:O}
	it holds for all $n \in \N$, $t \in [0, T]$ that
	$ \mathbb{O}^n_t
	= \tilde{\mathcal{O}}_t^n + P_n\, e^{tA} \xi
	- \int_0^t e^{(t-s)(A-\eta)} \, \eta \, ( \tilde{\mathcal{O}}_s^n + P_n \, e^{sA} \xi ) \, ds $,
	\item \label{item:conv}
	it holds that
	$\P \bigl( \limsup_{n \to \infty} \sup_{s \in [0, T]} \| (O_s + e^{sA}  \xi) - (\tilde{\mathcal{O}}_s^n + P_n \, e^{sA}  \xi ) \|_{H_{\varrho}} =0 \bigr)=1$,
	\item \label{item:scheme}
	it holds for all $n \in \N$, $t \in [0, T]$ that
	\begin{equation*}
    \P \Bigl( \Y_t^n = P_n \, e^{ t A } \xi + \smallint_0^t P_n \,  e^{  ( t - s ) A } \, \one_{ \{ \| \Y_{ \lf s \rf_{h_n} }^n \|_{ H_{\varrho} } + \| \tilde{\mathcal{O}}_{ \lf s \rf_{h_n} }^n +P_n  \, e^{ \lf s \rf_{ h_n } A } \xi \|_{ H_{\varrho} } \leq | h_n|^{ - \chi } \}} \, F \big(  \Y_{ \lf s \rf_{ h_n } }^n \big) \, ds + \tilde{\mathcal{O}}_t^n \Bigr)=1,
	\end{equation*}
	and
	\item \label{item:regularity}
	it holds that 
	\begin{align}
	\nonumber
	& \sup_{ n \in \N } \E\biggl[ \int_0^T e^{ \int_s^T p\, \phi( \mathbb{O}_{\lf u \rf_{h_n} }^n )  \, du} \max\bigl\{ 1, |\Phi(\mathbb{O}_{ \lf s \rf_{h_n} }^n )|^{\nicefrac{p}{2}}, \|\mathbb{O}_s^n\|_H^p, \smallint\nolimits_{0}^T \| \tilde{\mathcal{O}}_u^n+ P_n \, e^{uA} \xi \|_{H_{\varrho}}^{2p+ 2p\vartheta} \, du \bigr\} \, ds \biggr]
	\\ & +
	\sup_{ n \in \N } \sup_{ s \in [0,T]} \E[ \| \mathbb{O}_s^n \|_H^p]
	< \infty.  
	\end{align}
\end{enumerate}
\end{prop}
\begin{proof}[Proof of Proposition~\ref{prop:exists}]
Throughout this proof let $\varepsilon \in (0, \nicefrac{1}{4}- \varrho)$, $\beta \in (0, \nicefrac{1}{4})$, $q \in (\max\{p, \nicefrac{1}{\beta},\nicefrac{4}{\varepsilon} \}, $ $ \infty)$.
Observe that Lemma~\ref{abs:phi(w)_finite} and Lemma~\ref{lem:conv:series} ensure that there exists a real number $\eta \in [0, \infty)$ such that
\begin{align}
\label{eq:eps:bound}
720  q^3 T  \gamma \pi^{4}  \left[ \sum_{k = 1}^\infty \frac{ k^{4 \beta}}{\lambda_k+\eta}\right]  \Big[\sup \! \Big(\Big\{ \! \sup\nolimits_{x \in (0,1)} |v(x)| \colon \big[v \in \mathcal{C}((0,1), \R) \text{ and } \|v\|_{\W^{\beta, q}((0,1), \mathbb{R})} \leq 1\big] \Big\}\Big)\Big]^2 \leq 1.
\end{align}
Next note that the Burkholder-Davis-Gundy-type inequality in Da~Prato \& Zabczyk~\cite[Lemma~7.7]{dz92} yields   for all $n \in \N$,  $t_1, t_2 \in [0, T]$ with $t_1 \leq t_2$  that
\begin{align}
\begin{split}
& \biggl( \E \biggl[
\biggl\| \int_0^{t_1}P_n \, e^{(t_1-s)A} \,  dW_s - \int_0^{t_2} P_n \, e^{(t_2-s)A} \,  dW_s \biggr\|_{  H_{\varrho} }^q
\biggr] \biggr)^{ \nicefrac{2}{q} } \\
& + \biggl( \E \biggl[
\biggl\| \int_0^{t_1} e^{(t_1-s)A} \,  dW_s - \int_0^{t_2}  e^{(t_2-s)A} \,  dW_s \biggr\|_{ H_{\varrho} }^q
\biggr] \biggr)^{ \nicefrac{2}{q} } \\
& =
\biggl( \E \biggl[
\biggl\| \int_0^{t_2}  \bigl( \one_{(-\infty,t_1)}(s)\, P_n \, e^{\max\{ t_1-s, 0 \}A} - P_n \, e^{(t_2-s)A} \bigr) \,  dW_s \biggr\|_{ H_{\varrho} }^q
\biggr] \biggr)^{ \nicefrac{2}{q} } \\
& \quad  +
\biggl( \E \biggl[
\biggl\| \int_0^{t_2} \bigl( \one_{(-\infty,t_1)}(s)\, e^{ \max\{ t_1-s, 0 \}A} -   e^{(t_2-s)A} \bigr) \,  dW_s \biggr\|_{ H_{\varrho} }^q
\biggr] \biggr)^{ \nicefrac{2}{q} } \\
&\leq \tfrac{q(q-1)}{2} \int_0^{t_2}  \left\|  \one_{(-\infty,t_1)}(s)\, P_n \, e^{\max\{ t_1-s, 0 \}A} -  P_n \, e^{(t_2-s)A}   \right\|^2_{\mathrm{HS}(H, H_{\varrho})} ds \\
& \quad + \tfrac{q(q-1)}{2}   \int_0^{t_2} \left\|  \one_{(-\infty,t_1)}(s)\, e^{\max\{ t_1-s, 0 \} A} -   e^{(t_2-s)A}  \right\|^2_{\mathrm{HS}(H, H_{\varrho})} ds \\
& \leq q (q-1) \bigg[ \int_{t_1}^{t_2} \big \| e^{(t_2-s)A} \big\|^2_{\mathrm{HS}(H, H_{\varrho})} \, ds + \int_{0}^{t_1} \big \| e^{(t_1-s)A}- e^{(t_2-s)A} \big\|^2_{\mathrm{HS}(H, H_{\varrho})} \, ds \bigg] \\
& = q (q-1) \bigg[ \int_{t_1}^{t_2} \big \|  e^{(t_2-s)A}  \big\|^2_{\mathrm{HS}(H, H_{\varrho})} \, ds+ \int_0^{t_1} \big\| e^{(t_1-s)A} \big(\Id_H - e^{(t_2-t_1)A}\big)  \big\|_{\mathrm{HS}(H, H_{\varrho})}^2 \, ds \bigg].
\end{split}
\end{align}
This proves for all   $n \in \N$,  $t_1, t_2 \in [0, T]$ with $t_1 \leq t_2$  that
\begin{align}
\label{eq:O:Holder3}
\begin{split}
& \biggl( \E \biggl[
\biggl\| \int_0^{t_1}P_n \, e^{(t_1-s)A} \,  dW_s - \int_0^{t_2} P_n \, e^{(t_2-s)A} \,  dW_s \biggr\|_{  H_{\varrho} }^q
\biggr] \biggr)^{ \nicefrac{2}{q} } \\
& + \biggl( \E \biggl[
\biggl\| \int_0^{t_1} e^{(t_1-s)A} \,  dW_s - \int_0^{t_2}  e^{(t_2-s)A} \,  dW_s \biggr\|_{ H_{\varrho} }^q
\biggr] \biggr)^{ \nicefrac{2}{q} } \\
&   \leq q(q-1) \int_{t_1}^{t_2} \big \| (-A)^{\varrho} \, e^{(t_2-s)A}  \big\|^2_{\mathrm{HS}(H)} \, ds \\
&\quad + q (q-1) \int_0^{t_1} \big\| (-A)^{\varrho+\varepsilon} \, e^{(t_1-s)A}  \big\|_{\mathrm{HS}(H)}^2 \, \big\|(-A)^{-\varepsilon} \big(\Id_H - e^{(t_2-t_1)A}\big)  \big\|_{L(H)}^2 \, ds \\
& = q (q-1) \sum_{k = 1}^\infty \int_{t_1}^{t_2} (\lambda_k)^{2\varrho} \,e^{-2(t_2-s) \lambda_k} \, ds  \\
& \quad +  q (q-1) \big\|(-A)^{-\varepsilon} \big(\Id_H - e^{(t_2-t_1)A}\big)  \big\|_{L(H)}^2 \sum_{k = 1}^\infty \int_0^{t_1} (\lambda_k)^{2\varrho+2\varepsilon} \,e^{-2(t_1-s) \lambda_k} \, ds \\
& = q(q-1) \sum_{k = 1}^\infty\frac{( \lambda_k)^{2\varrho} (1-e^{-2\lambda_k(t_2-t_1)})}{2\lambda_k} \\
&\quad + q(q-1) \big\|(-A)^{-\varepsilon} \big(\Id_H - e^{(t_2-t_1)A}\big)  \big\|_{L(H)}^2 \sum_{k = 1}^\infty\frac{( \lambda_k)^{2\varrho+2\varepsilon} (1-e^{-2\lambda_k t_1})}{2\lambda_k}.
\end{split}
\end{align}
Moreover, note that the fact that
$ \forall \, r \in [ 0, 1 ], \, t \in [ 0, \infty ) \colon
\| (-A)^{-r} ( \operatorname{Id}_H - e^{ t A } ) \|_{L(H)} \leq t^r $
(cf., e.g., Lemma~11.36 in Renardy \& Rogers~\cite{RenardyRogers1993})
implies that 
\begin{align}
\label{eq:O:Holder2}
\begin{split}
& \sup_{ t_1, t_2 \in [ 0, T ], \, t_1 < t_2 }
\frac{\|(-A)^{-\varepsilon} (\Id_H - e^{(t_2-t_1)A})  \|_{L(H)}^2  }{(t_2-t_1)^{2\varepsilon}}
= \sup_{ t \in ( 0, T ] } \bigl( t^{ -\varepsilon }
\| (-A)^{-\varepsilon} (\Id_H - e^{tA})  \|_{L(H)} \bigr)^2 
\leq 1.
\end{split}
\end{align}
The fact that $\forall \, x\in \R \colon 1-e^{-x} \leq x$ and \eqref{eq:O:Holder3} hence establish for all   $n \in \N$,  $t_1, t_2 \in [0, T]$ with $t_1 < t_2$  that
\begin{align}\label{eq:O:Holder}
\begin{split}
& \biggl( \E \biggl[
\biggl\| \int_0^{t_1}P_n \, e^{(t_1-s)A} \,  dW_s - \int_0^{t_2} P_n \, e^{(t_2-s)A} \,  dW_s \biggr\|_{  H_{\varrho} }^q
\biggr] \biggr)^{ \nicefrac{2}{q} } \\
& + \biggl( \E \biggl[
\biggl\| \int_0^{t_1} e^{(t_1-s)A} \,  dW_s - \int_0^{t_2}  e^{(t_2-s)A} \,  dW_s \biggr\|_{ H_{\varrho} }^q
\biggr] \biggr)^{ \nicefrac{2}{q} } \\
& \leq q(q-1) \sum_{k = 1}^\infty\frac{(\lambda_k)^{2\varrho-1} (1-e^{-2\lambda_k(t_2-t_1)})^{2\varepsilon}}{2} \\
&\quad + q(q-1) \big\|(-A)^{-\varepsilon} \big(\Id_H - e^{(t_2-t_1)A}\big)  \big\|_{L(H)}^2 \sum_{k = 1}^\infty ( \lambda_k)^{2\varrho+2\varepsilon-1} \\
& \leq q(q-1)   \Bigg[ \sum_{k = 1}^\infty( \lambda_k)^{2\varrho+2\varepsilon-1} \Bigg] \Bigg( 1+  \frac{\|(-A)^{-\varepsilon} (\Id_H - e^{(t_2-t_1)A})  \|_{L(H)}^2  }{(t_2-t_1)^{2\varepsilon}}\Bigg) (t_2-t_1)^{2\varepsilon} \\
& \leq 2 \, q(q-1)   \Bigg[ \sum_{k = 1}^\infty ( \lambda_k)^{2\varrho+2\varepsilon-1} \Bigg] (t_2-t_1)^{2\varepsilon} < \infty.
\end{split}
\end{align}
Combining this with the  fact that $ q \varepsilon > 1$ and the Kolmogorov-Chentsov continuity theorem  shows that there exist  stochastic processes
$ O \colon [0, T] \times \Omega \to H_{\varrho}$,
$ \tilde{\mathcal{O}}^n \colon [0, T] \times \Omega \to P_n(H)$, $n \in \N$,
and $ \mathbb{O}^n \colon [0, T] \times \Omega \to P_n(H)$, $n \in \N$,
with continuous sample paths
which satisfy
for all $n \in \N$, $t \in [0, T]$ that
$[O_t]_{\P, \mathcal{B}(H)} =  \int_0^{t}  e^{(t-s)A} \,  dW_s$,
$[\tilde{\mathcal{O}}^n_t]_{\P, \mathcal{B}(H)} =  \int_0^{t} P_n \, e^{(t-s)A} \, dW_s$,
and
\begin{equation}
\label{eq:mathbbO_def}
\mathbb{O}^n_t
= \tilde{\mathcal{O}}_t^n + P_n\, e^{tA} \xi
- \int_0^t e^{(t-s)(A-\eta)} \, \eta \, ( \tilde{\mathcal{O}}_s^n + P_n \, e^{sA} \xi ) \, ds.
\end{equation}
This proves \eqref{item:O}--\eqref{item:mathbb:O}.
Next observe that the fact that
$ \forall \, n \in \N, \, t \in [0, T] \colon \P(\mathcal{O}_t^n = \tilde{\mathcal{O}}^n_t )=1$
and
Lemma~\ref{lemma:conv:rate} demonstrate that 
\begin{align}\
\sup_{n \in \N} \biggl\{n^{\varepsilon} \sup_{t \in [0, T]} \bigl( \E \bigl[ \|O_t  - \tilde{\mathcal{O}}_t^n \|_{ H_{\varrho} }^q \bigr] \bigr)^{ \nicefrac{1}{q} } \biggr\} < \infty.
\end{align}
The fact that $ O \colon [0, T] \times \Omega \to H_{\varrho}$ and  $\tilde{\mathcal{O}}^n \colon [0, T] \times \Omega \to P_n(H)$, $n \in \N$, are stochastic processes with continuous sample paths, \eqref{eq:O:Holder}, and Cox et al.~\cite[Corollary~2.11]{CoxWelti2016} (with $T=T$, $p=q$, $\beta= \varepsilon$, $\theta^N= \{ \frac{k T}{N} \in [0, T] \colon k \in \{ 0, 1, \ldots, N \} \}$, $ E = H_{\varrho} $, $Y^N= ([0,T] \times \Omega  \ni (t, \omega) \mapsto \tilde{\mathcal{O}}^N_t(\omega) \in H_{\varrho})$, $Y^0= O$, $\alpha=0$, $\varepsilon= \nicefrac{\varepsilon}{2}$ for $N \in \N$ in the notation of   Cox et al.~\cite[Corollary~2.11]{CoxWelti2016}) hence prove that
\begin{align}
\sup_{n \in \N}
\biggl\{ n^{(\nicefrac{\varepsilon}{2}- \nicefrac{1}{q})}
\biggl(
\E \biggl[
    \sup_{t \in [0, T]} \|O_t - \tilde{\mathcal{O}}_t^n \|_{H_{\varrho}}^q
\biggr] \biggr)^{ \nicefrac{1}{q} }
\biggr\} < \infty.
\end{align}
This, the fact that $\nicefrac{\varepsilon}{2}- \nicefrac{1}{q} > \nicefrac{1}{q}$, and Hutzenthaler \& Jentzen~\cite[Lemma~3.21]{Hutzenthaler2015} (cf., e.g.,  Graham \& Talay~\cite[Theorem~7.12]{Graham2013} and  Kloeden \& Neuenkirch~\cite[Lemma~2.1]{Kloeden2007}) ensure that  
\begin{align}\label{eq:O:conv}
\P \bigg(\limsup_{n \to \infty} \sup_{s \in [0, T]} \| O_s - \tilde{\mathcal{O}}_s^n \|_{H_{\varrho}} =0 \bigg)=1.
\end{align}
Next note that  for all $n \in \N$ it holds that 
\begin{align}
\begin{split}
\sup_{ s \in [ 0, T ] }
\|(\Id_H- P_n) \, e^{sA}  \xi \|_{H_{\varrho}} &\leq \|(-A)^{(\varrho-\nicefrac{1}{2})} (\Id_H-P_n |_H ) \|_{L(H)} \|\xi\|_{H_{\nicefrac{1}{2}}} \\
&=  \|(-A)^{-1} (\Id_H-P_n |_H ) \|_{L(H)}^{(\nicefrac{1}{2}-\varrho)} \|\xi\|_{H_{\nicefrac{1}{2}}}
\leq ( c_0 \pi^2 n^2 )^{(\varrho - \nicefrac{1}{2})} \|\xi\|_{H_{\nicefrac{1}{2}}} .
\end{split}
\end{align}
Combining this with \eqref{eq:O:conv} shows that 
\begin{align}
\P \bigg( \limsup_{n \to \infty} \sup_{s \in [0, T]} \| (O_s + e^{sA}  \xi) - (\tilde{\mathcal{O}}_s^n + P_n \, e^{sA}  \xi ) \|_{H_{\varrho}} =0 \bigg)=1.
\end{align}
This establishes~\eqref{item:conv}. Furthermore, the fact that  $ \forall \, n \in \N, \, t \in [0, T] \colon \P(\mathcal{O}_t^n = \tilde{\mathcal{O}}^n_t )=1$ and \eqref{eq:set:abstract}
prove~\eqref{item:scheme}.
Moreover, Proposition~\ref{prop:transform_SG}
(with $H=H$,
$U=H$,
$ \mathbb{H} = \{ e_k \in H \colon k \in \N \} $,
$T=T$,
$\alpha= 0$,
$\beta=0$,
$\gamma=0$,
$\eta=\eta$,
$\kappa=0$,
$A=A$,
$ ( W_t )_{ t \in [ 0, T ] } = ( W_t )_{ t \in [ 0, T ] } $,
$O=P_n(H)$,
$F= (P_n(H) \ni v \mapsto 0 \in H)$,
$ \tilde{F} = ( P_n(H) \ni v \mapsto \eta v \in H) $,
$B= ( P_n(H) \ni v \mapsto ( H \ni u \mapsto P_n(u) \in H ) \in \mathrm{HS}(H) ) $,
$\xi= ( \Omega \ni \omega \mapsto P_n \, \xi \in P_n(H) ) $,
$X= ( [0,T] \times \Omega \ni (t,\omega) \mapsto (\tilde{\mathcal{O}}^n_t( \omega ) + P_n \, e^{tA} \xi) \in P_n( H ) ) $
for $n \in \N$
in the notation of  Proposition~\ref{prop:transform_SG})
ensures that for all $n \in \N$, $t \in [0,T]$ it holds that 
\begin{equation}
\begin{split}
[ \tilde{\mathcal{O}}_t^n + P_n\, e^{tA} \xi ]_{\P, \mathcal{B}(H)}
& =
\left[
P_n \, e^{t(A-\eta)} \xi
+ \int_0^t e^{(t-s)(A-\eta)} \, \eta \, ( \tilde{\mathcal{O}}_s^n + P_n \, e^{sA} \xi ) \, ds
\right]_{\P, \mathcal{B}(H)}
\\ & \quad 
+ \int_0^{t} P_n \, e^{(t-s)(A-\eta)} \, dW_s.
\end{split}
\end{equation}
This and \eqref{eq:mathbbO_def} imply for all $n \in \N$, $t \in [0,T]$ that
\begin{equation}
\label{eq:trans:O}
[ \mathbb{O}^n_t - P_n \, e^{t(A-\eta)} \xi ]_{\P, \mathcal{B}(H)}
= \int_0^{t} P_n \, e^{(t-s)(A-\eta)} \, dW_s.
\end{equation}
In addition, note that for all $ n \in \N $ it holds that
$ [0,T] \times \Omega \ni (t,\omega) \mapsto ( \mathbb{O}^n_t( \omega ) - P_n \, e^{t(A-\eta)} \xi ) \in P_n( H ) $
is a stochastic process with continuous sample paths.
Proposition~\ref{prop:abstract}, \eqref{eq:trans:O}, and \eqref{eq:eps:bound} hence show that 
\begin{align}
\begin{split}
& \sup_{ n \in \N } \E\biggl[ \int_0^T e^{ \int_s^T q\, \phi( \mathbb{O}_{\lf u \rf_{h_n} }^n )  \, du} \max\bigl\{ 1, |\Phi(\mathbb{O}_{ \lf s \rf_{h_n} }^n )|^{\nicefrac{q}{2}}, \|\mathbb{O}_s^n\|_H^q, \smallint\nolimits_{0}^T \| \tilde{\mathcal{O}}_u^n+ P_n \, e^{uA} \xi \|_{H_{\varrho}}^{2q+ 2q\vartheta} \, du \bigr\} \, ds \biggr]
\\ & +
\sup_{ n \in \N } \sup_{ s \in [0,T]} \E[ \| \mathbb{O}_s^n \|_H^q]
< \infty.
\end{split}
\end{align}
This establishes~\eqref{item:regularity}.
The proof of Proposition~\ref{prop:exists} is thus completed. 
\end{proof}

\subsection{Strong convergence}

\begin{prop}
\label{abs:prop:last}
Assume the setting in Subsection~\ref{setting:example} and let $ X\colon [0, T] \times \Omega \to H_{\varrho}$  be a  stochastic process with continuous sample paths which satisfies for all $t \in [0, T]$  that  $ [X_t ]_{\P, \mathcal{B}(H)} =    [ e^{ t A }   \xi  + \smallint_0^t e^{  ( t - s ) A}  \, F (  X_s ) \, ds ]_{\P, \mathcal{B}(H)} + \int_0^t e^{(t-s)A} \,  dW_s$. Then it holds for all $p \in (0, \infty)$ that
\begin{align}
\limsup_{n \to \infty} \sup_{t \in [0,T]} \E \big[ \| X_t -\Y_t^n \|_H^p \big] = 0.
\end{align}
\end{prop}
\begin{proof}[Proof of Proposition~\ref{abs:prop:last}]
Throughout this proof let  $p \in (0,\infty)$, $ q \in ( \max\{ p, 2 \}, \infty ) $.
Note that Proposition~\ref{prop:exists}
shows that there exist a real number $\eta \in [0,\infty)$ and stochastic processes
$ O \colon [0, T] \times \Omega \to H_{\varrho}$,
$ \tilde{\mathcal{O}}^n \colon [0, T] \times \Omega \to P_n(H)$, $n \in \N$,
and $ \mathbb{O}^n \colon [0, T] \times \Omega \to P_n(H)$, $n \in \N$,
with continuous sample paths
which satisfy
for all $n \in \N$, $t \in [0, T]$ that
\begin{gather}
\label{eq:O:as}
[O_t]_{\P, \mathcal{B}(H)} =  \int_0^{t}  e^{(t-s)A} \,  dW_s,
\qquad\qquad\!\!
[\tilde{\mathcal{O}}^n_t]_{\P, \mathcal{B}(H)} =  \int_0^{t} P_n \, e^{(t-s)A} \, dW_s,
\\
\mathbb{O}^n_t
= \tilde{\mathcal{O}}_t^n + P_n\, e^{tA} \xi
- \int_0^t e^{(t-s)(A-\eta)} \, \eta \, ( \tilde{\mathcal{O}}_s^n + P_n \, e^{sA} \xi ) \, ds,
\\
\P \bigg( \limsup_{m \to \infty} \sup_{s \in [0, T]} \| (O_s + e^{sA}  \xi) - (\tilde{\mathcal{O}}_s^m + P_m \, e^{sA}  \xi ) \|_{H_{\varrho}} =0 \bigg)=1,
\\\label{eq:X^n}
\!\P \Big( \Y_t^n = P_n \, e^{ t A } \xi \!+\! \smallint_0^t P_n \,  e^{  ( t - s ) A } \, \one_{ \{ \| \Y_{ \lf s \rf_{h_n} }^n \|_{ H_{\varrho} } + \| \tilde{\mathcal{O}}_{ \lf s \rf_{h_n} }^n +P_n  \, e^{ \lf s \rf_{ h_n } A } \xi \|_{ H_{\varrho} } \leq | h_n|^{ - \chi } \}} \, F \big(  \Y_{ \lf s \rf_{ h_n } }^n \big) \, ds \!+\! \tilde{\mathcal{O}}_t^n \Big)=1,
\end{gather}
and
\begin{align}\label{eq:strong:limsup}
\nonumber
& \limsup_{ m \to \infty} \E\biggl[ \int_0^T e^{ \int_s^T q\, \phi( \mathbb{O}_{\lf u \rf_{h_m} }^m )  \, du} \max\bigl\{ 1, |\Phi(\mathbb{O}_{ \lf s \rf_{h_m} }^m )|^{\nicefrac{q}{2}}, \|\mathbb{O}_s^m\|_H^q, \smallint\nolimits_{0}^T \| \tilde{\mathcal{O}}_u^m+ P_m \, e^{uA} \xi \|_{H_{\varrho}}^{2q+ 2q\vartheta} \, du \bigr\} \, ds \biggr]
\\ & +
\limsup_{ m \to \infty} \sup_{ s \in [0,T]} \E[ \| \mathbb{O}_s^m \|_H^q]
< \infty.
\end{align}
In addition, observe that the fact that
$\eta \in [0, \infty)$
and the assumption that
$\forall \, n \in \N, \, v, w \in P_n( H ) \colon \left< v, P_n F( v + w ) \right>_H \leq \phi( w ) \| v \|^2_H + \varphi \| v \|^2_{ H_{ \nicefrac{1}{2} } } + \Phi( w )$
establish
that for all $ n \in \N $, $ v, w \in P_n( H ) $ it holds that
\begin{equation}
\left< v, P_n F( v + w ) \right>_H \leq \phi( w ) \| v \|^2_H + \varphi \| (\eta-A)^{\nicefrac{1}{2}} v \|^2_{ H} + \Phi( w ).
\end{equation}
Combining this,
the fact that
$ \nicefrac{ ( 1 - \alpha - \rho ) }{( 1 + 2 \vartheta ) }
\geq
\nicefrac{ ( 2 \varrho - \rho ) }{( 1 + 2 \vartheta ) }
\geq
\nicefrac{ ( 2 \varrho - 2 \rho ) }{( 2 + 2 \vartheta ) }
=
\nicefrac{( \varrho - \rho ) }{( 1 + \vartheta) } $,
the fact that $ \forall \, t \in [0, T] \colon \P(X_t = \int_0^t e^{(t-s)A} \, F(X_s) \, ds +O_t + e^{tA}  \xi)=1$,
and
\eqref{eq:O:as}--\eqref{eq:strong:limsup}
with \eqref{item:strong} in Theorem~\ref{thm:strong} (with
$ H = H $,
$\mathbb{H}= \{ e_k \in H \colon k \in \mathbb{N} \}$,
$\eta=\eta$,
$ \theta=  \theta$,
$\kappa= 0$,
$ A = A $,
$\varphi= \varphi$,
$\alpha= \alpha$,
$\rho= \rho$,
$\varrho =\varrho$,
$\vartheta=\vartheta$,
$ T = T $,
$ \chi = \chi $,
$p=q$,
$F=F$,
$\phi=\phi$,
$\Phi=\Phi$,
$ ( \mathbb{H}_n )_{ n \in \N } = ( \{ e_k \in H \colon k \in \{1, 2, \ldots, n\}\} )_{ n \in \N } $,
$ ( P_n )_{ n \in \N } = ( P_n )_{ n \in \N } $,
$ ( h_n )_{ n \in \N } = ( h_n )_{ n \in \N } $,
$ ( \mathbb{O}^n )_{ n \in \N } = ( [0,T] \times \Omega \ni (t,\omega) \mapsto \mathbb{O}^n_t( \omega ) \in H_{\varrho} )_{ n \in \N } $,
$ ( \Y^n )_{ n \in \N } = ([0, T] \times \Omega \ni (t,\omega) \mapsto \Y_t^n(\omega) \in H_{\varrho})_{ n \in \N } $,
$ ( \mathcal{O}^n )_{ n \in \N } = ( [0,T] \times \Omega \ni (t,\omega) \mapsto (\tilde{\mathcal{O}}^n_t( \omega ) + P_n \, e^{tA} \xi) \in P_n(H) )_{ n \in \N } $,
$X=X$,
$ O = ( [0,T] \times \Omega \ni (t,\omega) \mapsto (O_t( \omega ) + e^{tA}  \xi) \in H_{\varrho}) $,
$q=p$
in the notation of \eqref{item:strong} in Theorem~\ref{thm:strong})
completes the proof of Proposition~\ref{abs:prop:last}. 
\end{proof}

\section{Examples}
\label{sec:examples}

In this section we demonstrate how Proposition~\ref{abs:prop:last} can be applied to  stochastic Burgers and stochastic Allen-Cahn equations (see Corollary~\ref{cor:burgers:short} and Corollary~\ref{cor:cahn:short} below).

\subsection{Stochastic Burgers equations}
\label{sec:Burgers}

\subsubsection{Setting}\label{setting:burgers}
Consider the notation in Subsection~\ref{sec:notation},
let $c_1 \in \R$, $ T, c_0 \in (0,\infty)$, 
$\varrho \in (\nicefrac{1}{8}, \nicefrac{1}{4})$,
$ \chi \in  (0, \nicefrac{\varrho }{2 } - \nicefrac{1}{16}] $,
$( H, \left< \cdot , \cdot \right>_H, \left\| \cdot \right\|_H ) = (L^2(\lambda_{(0,1)}; \R), \langle \cdot , \cdot \rangle_{L^2(\lambda_{(0,1)}; \R)}, \left\| \cdot \right\|_{L^2(\lambda_{(0,1)}; \R)} )$,
let $(e_n)_{n \in \N} \colon \N \to H $
be the function which satisfies
for all $ n \in \N $ that
$ e_n = [ (\sqrt{2} \sin(n \pi x) )_{x \in (0,1)}]_{\lambda_{(0,1)} , \mathcal{B}(\R)}$,
let $ A \colon D(A) \subseteq H \to H $ be the linear operator
which satisfies
$ D(A) = \{ v \in H \colon $ $ \sum_{k = 1}^\infty k^4 | \langle e_k , v \rangle_H |^2 < \infty \} $
and $ \forall \, v \in D(A) \colon A v = - c_0 \pi^2 \sum_{k = 1}^\infty k^2 \langle e_k , v \rangle_H e_k$,
let $ ( H_r, \left< \cdot , \cdot \right>_{ H_r }, \left\| \cdot \right\|_{ H_r } ) $, $ r \in \R $, be a family of interpolation spaces associated to $ -A $,
let
$\xi \in H_{\nicefrac{1}{2}} $,
let
$F \colon H_{\nicefrac{1}{8}} \to H_{-\nicefrac{1}{2}} $,
$(P_n)_{n \in \N} \colon \N \to L(H) $,
and
$(h_n)_{n \in \N} \colon \N \to (0, T] $
be functions which satisfy
for all $v \in H_{\nicefrac{1}{8}}$, $n \in \N$ that
$F(v)= c_1(v^2)'$,
$ P_n(v) = \sum_{k = 1 }^n \langle e_k, v \rangle_H e_k $,
and
$ \limsup_{ m \to \infty} h_m =0$,
let $ ( \Omega, \F, \P ) $ be a probability space,
let $(W_t)_{t \in [0, T]}$ be an $\Id_H$-cylindrical $( \Omega, \F, \P )$-Wiener process,
and let $ \Y^n, \mathcal{O}^n \colon [0, T] \times \Omega \to P_n(H)$, $ n \in \N$, be stochastic processes
which satisfy
for all $ n \in \N $, $ t \in [0,T] $ that
$\left[\mathcal{O}_t^n \right]_{\P, \mathcal{B}(H)} = \int_0^t P_n \, e^{(t-s)A}  \, dW_s$
and
\begin{equation}\label{eq:set:burgers}
\P \Big( \Y_t^n = P_n \, e^{ t A } \xi + \smallint_0^t P_n \,  e^{  ( t - s ) A } \, \one_{ \{ \| \Y_{ \lf s \rf_{h_n} }^n \|_{ H_{\varrho} } + \| \mathcal{O}_{ \lf s \rf_{h_n} }^n +P_n \, e^{ \lf s \rf_{ h_n } A } \xi \|_{ H_{\varrho} } \leq | h_n|^{ - \chi } \}} \, F \big(  \Y_{ \lf s \rf_{ h_n } }^n \big) \, ds + \mathcal{O}_t^n \Big)=1. 
\end{equation}

\subsubsection{Properties of the nonlinearity}

In this subsection we establish a few elementary properties of the function $F$ in Subsection~\ref{setting:burgers} above, see Lemmas~\ref{coer:burgers}--\ref{nonlin:burgers} below.
For the proof of these properties we present two elementary and well-known facts in Lemmas~\ref{burgers_norms}--\ref{lem:sup:v} below.
See, e.g., Section~4 in Jentzen, Kloeden, \& Winkel~\cite{jkw09} for the next result, Lemma~\ref{burgers_norms}.
\begin{lemma}\label{burgers_norms}
Assume the setting in Subsection~\ref{setting:burgers}. Then it holds for all $v \in H_{\nicefrac{1}{2}}$ that $\|v\|_{H_{\nicefrac{1}{2}}} = \sqrt{c_0} \|v'\|_H$.
\end{lemma}	
\begin{proof}[Proof of Lemma~\ref{burgers_norms}]
Note that integration by parts proves for all $v \in H_1$ that 
\begin{align}\label{burgers_norms_1}
\begin{split}
\|v\|_{H_{\nicefrac{1}{2}}}^2  & =\| (-A)^{\nicefrac{1}{2}} v \|^2_{ H} = \langle (-A)^{\nicefrac{1}{2}} v, (-A)^{\nicefrac{1}{2}} v \rangle_H=  -\langle v, Av \rangle_H \\ & =  -c_0 \langle v, v'' \rangle_H
= c_0 \langle v', v' \rangle_H =  c_0 \|v'\|_H^2.
\end{split}
\end{align}
The fact that $H_1 \subseteq H_{\nicefrac{1}{2}}$ is dense in $H_{\nicefrac{1}{2}}$ and the fact that $(H_{\nicefrac{1}{2}} \ni v \mapsto v' \in H) \in L(H_{\nicefrac{1}{2}}, H)$ thus complete the proof of Lemma~\ref{burgers_norms}.
\end{proof}
Cf., e.g., Lemma~4.7 in Bl\"omker \& Jentzen~\cite{BloemkerJentzen2013} for the next lemma.
\begin{lemma}
	\label{lem:sup:v}
Assume the setting in Subsection~\ref{setting:burgers}. Then it holds that
\begin{align}
\sup_{v \in H \backslash \{0\}} \frac{\|v'\|_{H_{-\nicefrac{1}{2}}}}{\|v\|_H} = (c_0)^{-\nicefrac{1}{2}}.
\end{align}
\end{lemma}
\begin{proof}[Proof of Lemma~\ref{lem:sup:v}]
Observe that, e.g., (iii) in Lemma~3.10 in Jacobe de Naurois et al.~\cite{Naurois2015} and Lemma~\ref{burgers_norms} show that
\begin{equation}
\begin{split}
\sup_{v \in H \backslash \{0\}} \frac{\|v'\|_{H_{-\nicefrac{1}{2}}}}{\|v\|_H} & = \sup_{v \in H \backslash \{0\}} \sup_{w \in H_{\nicefrac{1}{2}} \backslash  \{0\}} \frac{|\langle v',w  \rangle_H |}{\|v\|_H \|w \|_{H_{\nicefrac{1}{2}}}} \\
&=  \sup_{w \in H_{\nicefrac{1}{2}} \backslash  \{0\}} \sup_{v \in H \backslash \{0\}} \frac{|\langle v,w'  \rangle_H |}{\|v\|_H \|w \|_{H_{\nicefrac{1}{2}}}}\\
& = \sup_{w \in H_{\nicefrac{1}{2}} \backslash  \{0\}} \frac{\|w'\|_H}{\|w\|_{H_{\nicefrac{1}{2}}}}= (c_0)^{-\nicefrac{1}{2}}.
\end{split}
\end{equation}
The proof of Lemma~\ref{lem:sup:v} is thus completed.
\end{proof}

The next simple lemma is a slight modification of Lemma~5.7 in Bl\"omker \& Jentzen~\cite{BloemkerJentzen2013}.
\begin{lemma}
\label{coer:burgers}
Assume the setting in Subsection~\ref{setting:burgers} and let $v, w \in H_{ \nicefrac{1}{2} }$. Then it holds that $F(v+w) \in H$ and
\begin{align}
\label{eq:coer:burgers}
\begin{split}
\bigl| \left< v, F( v + w ) \right>_H \bigr| &\leq \max\big\{\tfrac{2|c_1|^2 }{c_0} , 4\big\} \| v \|_H^2 \big[\!\sup\nolimits_{x \in (0,1)} |\und{w}(x)|^2\big]+ \tfrac{3}{4}\| v \|^2_{ H_{ \nicefrac{1}{2} } } \\
&\quad + \max\big\{\tfrac{2|c_1|^2 }{c_0} , 4\big\} \big[1+\sup\nolimits_{x \in (0,1)} |\und{w}(x)|^{ \max\{\nicefrac{2|c_1|^2 }{c_0} , 4\}}\big].
\end{split}
\end{align}
\end{lemma}
\begin{proof}[Proof of Lemma~\ref{coer:burgers}]
Observe that, e.g., Lemma~4.5 in Jentzen \& Pu\v{s}nik~\cite{JentzenPusnik2016} ensures for all $ u \in H_{ \nicefrac{1}{2} } $ that $F(u) \in H$. Next note that integration by parts and, e.g., again Lemma~4.5 in Jentzen \& Pu\v{s}nik~\cite{JentzenPusnik2016} yield that
\begin{align}
\begin{split}
3 \, \langle  v', v^2 \rangle_H &= 2 \, \langle v \cdot v', v \rangle_H + \langle v', v^2 \rangle_H = \langle (v^2)', v \rangle_H + \langle v', v^2 \rangle_H \\
& = - \langle v^2, v' \rangle_H + \langle v', v^2 \rangle_H  = 0.
\end{split}
\end{align}
Applying integration by parts again hence shows that
\begin{align}
\begin{split}
\langle v, F(v+w) \rangle_H &= c_1 \langle v, [(v+w)^2]' \rangle_H = - c_1 \langle v', (v+w)^2 \rangle_H\\
& = - c_1 \langle v', v^2 \rangle_H - 2 \, c_1 \langle v', v \cdot w \rangle_H - c_1 \langle v', w^2 \rangle_H\\
& = - 2 \, c_1 \langle v', v \cdot w \rangle_H - c_1 \langle v', w^2 \rangle_H.
\end{split}
\end{align}
This, the Cauchy-Schwartz inequality, the fact that $\forall \, a,b \in \R, \, \varepsilon \in (0, \infty) \colon 2ab \leq \varepsilon a^2 + \frac{b^2}{\varepsilon}$, and Lemma~\ref{burgers_norms}  prove that
\begin{align}
\begin{split}
\bigl| \left< v, F( v + w ) \right>_H \bigr| &\leq 2 \, |c_1|\|v'\|_H \| v \cdot w \|_H + |c_1| \|v'\|_H \|w^2\|_H \\
& \leq \tfrac{c_0}{2} \|v'\|_H^2 + \tfrac{2|c_1|^2}{c_0} \| v \cdot w \|_H^2 + \tfrac{c_0}{4}\|v'\|_H^2 + \tfrac{|c_1|^2 }{c_0}\| w^2 \|_H^2 \\
& \leq \tfrac{2|c_1|^2}{c_0} \| v \|_H^2 \big[\!\sup\nolimits_{x \in (0,1)} |\und{w}(x)|^2\big]+ \tfrac{3 c_0}{4}\|v'\|_H^2 + \tfrac{|c_1|^2 }{c_0} \big[\!\sup\nolimits_{x \in (0,1)} |\und{w}(x)|^4\big] \\
& = \tfrac{2|c_1|^2}{c_0} \| v \|_H^2 \big[\!\sup\nolimits_{x \in (0,1)} |\und{w}(x)|^2\big]+ \tfrac{3}{4}\| v \|^2_{ H_{ \nicefrac{1}{2} } } + \tfrac{|c_1|^2 }{c_0} \big[\!\sup\nolimits_{x \in (0,1)} |\und{w}(x)|^4\big]\\
& \leq \max\big\{\tfrac{2|c_1|^2 }{c_0} , 4\big\} \| v \|_H^2 \big[\!\sup\nolimits_{x \in (0,1)} |\und{w}(x)|^2\big]+ \tfrac{3}{4}\| v \|^2_{ H_{ \nicefrac{1}{2} } } \\
&\quad + \max\big\{\tfrac{2|c_1|^2 }{c_0} , 4\big\} \big[1+ \sup\nolimits_{x \in (0,1)} |\und{w}(x)|^{ \max\{\nicefrac{2|c_1|^2 }{c_0} , 4\}}\big] . 
\end{split}
\end{align}
The proof of Lemma~\ref{coer:burgers} is thus completed.
\end{proof}

\begin{lemma}\label{nonlin:burgers}
Assume the setting in Subsection~\ref{setting:burgers} and let $v, w \in H_{\nicefrac{1}{8}}$. Then it holds that $F \in \mathcal{C}(H_{\nicefrac{1}{8}}, H_{-\nicefrac{1}{2}})$ and
\begin{align}
\begin{split}
&\|F(v) - F(w) \|_{H_{-\nicefrac{1}{2}}} \\
& \leq |c_1|  |c_0|^{-\nicefrac{1}{2}} \left[\sup_{u \in  H_{\nicefrac{1}{8}} \setminus \{0\}} \frac{\|u\|^2_{L^4(\lambda_{(0,1)}; \R)}}{\|u\|^2_{H_{\nicefrac{1}{8}}}} \right] \big( 1+ \left\| v \right\|_{ H_{\nicefrac{1}{8}}} +  \left\| w \right\|_{ H_{\nicefrac{1}{8}}} \big) \left\| v - w \right\|_{ H_{\nicefrac{1}{8}}} < \infty.
\end{split}
\end{align}	
\end{lemma}
\begin{proof}[Proof of Lemma~\ref{nonlin:burgers}]
Observe that the fact that $v^2, w^2 \in H$ and Lemma~\ref{lem:sup:v} establish that
\begin{align}
\begin{split}
\|(v^2)' - (w^2)' \|_{H_{-\nicefrac{1}{2}}} &\leq |c_0|^{-\nicefrac{1}{2}} \| v^2 - w^2 \|_H  \\
&\leq |c_0|^{-\nicefrac{1}{2}} \|v+w \|_{L^4(\lambda_{(0,1)}; \R)} \|v -w \|_{L^4(\lambda_{(0,1)}; \R)} \\
& \leq |c_0|^{-\nicefrac{1}{2}} \left[\sup_{u \in  H_{ \nicefrac{1}{8} } \setminus \{0\}} \frac{\|u\|_{L^4(\lambda_{(0,1)}; \R)}}{\|u\|_{H_{\nicefrac{1}{8}}}} \right]^2 \! \! \left\| v + w \right\|_{ H_{ \nicefrac{1}{8} }} \left\| v - w \right\|_{ H_{\nicefrac{1}{8} }} \\
& \leq  |c_0|^{-\nicefrac{1}{2}} \left[\sup_{u \in  H_{ \nicefrac{1}{8} } \setminus \{0\}} \frac{\|u\|^2_{L^4(\lambda_{(0,1)}; \R)}}{\|u\|^2_{H_{\nicefrac{1}{8}}}} \right] \big( 1+ \left\| v \right\|_{ H_{\nicefrac{1}{8} }} +  \left\| w \right\|_{ H_{\nicefrac{1}{8}}} \big) \left\| v - w \right\|_{ H_{\nicefrac{1}{8}}}.
\end{split}
\end{align}
This shows that
\begin{align}
\label{nonlin:bur:last}
\begin{split}
& \|F(v) - F(w) \|_{H_{-\nicefrac{1}{2}}}  = \left\| c_1\! \left((v^2)' - (w^2)'\right) \right\|_{H_{-\nicefrac{1}{2}}} \\
& \leq |c_1|  |c_0|^{-\nicefrac{1}{2}} \left[\sup_{u \in  H_{ \nicefrac{1}{8}} \setminus \{0\}} \frac{\|u\|^2_{L^4(\lambda_{(0,1)}; \R)}}{\|u\|^2_{H_{\nicefrac{1}{8}}}} \right] \big( 1+ \left\| v \right\|_{ H_{\nicefrac{1}{8}}} +  \left\| w \right\|_{ H_{\nicefrac{1}{8}}} \big) \left\| v - w \right\|_{ H_{\nicefrac{1}{8}}}.
\end{split}
\end{align}
Next note that the Sobolev embedding theorem yields that
\begin{align}
\sup_{u \in  H_{ \nicefrac{1}{8} }\setminus \{0\}} \frac{\|u\|^2_{L^4(\lambda_{(0,1)}; \R)}}{\|u\|^2_{H_{\nicefrac{1}{8}}}}  < \infty.
\end{align}
Inequality \eqref{nonlin:bur:last} hence completes the proof of Lemma~\ref{nonlin:burgers}.
\end{proof}

\subsubsection{Strong convergence}

\begin{cor}
\label{burgers:cor:last}
Assume the setting in Subsection~\ref{setting:burgers} and let $ X\colon [0, T] \times \Omega \to H_{\varrho}$  be a  stochastic process with continuous sample paths which satisfies for all $t \in [0, T]$  that  $ [X_t ]_{\P, \mathcal{B}(H)} =    [ e^{ t A }   \xi  + \smallint_0^t e^{  ( t - s ) A}  \, F (  X_s ) \, ds ]_{\P, \mathcal{B}(H)} + \int_0^t e^{(t-s)A} \, dW_s$. Then it holds for all $p \in (0, \infty)$ that
\begin{align}
\limsup_{n \to \infty} \sup_{t \in [0,T]} \E \big[ \| X_t -\Y_t^n \|_H^p \big] = 0.
\end{align}
\end{cor}
\begin{proof}[Proof of Corollary~\ref{burgers:cor:last}]
Throughout this proof
let $(\tilde{P}_n)_{n \in \N} \colon \N \to L(H_{-1}) $
be the function which satisfies
for all $ n \in \N $, $ v \in H $ that
$ \tilde{P}_n(v) = P_n(v) $
and let $ \phi , \Phi \colon H_1 \to [0,\infty) $ be the functions which satisfy
for all $ v \in H_{1}$ that  $\phi(v) =  \max\big\{\frac{2|c_1|^2 }{c_0} , 4\big\}   \big[1+\sup\nolimits_{x \in (0,1)} |\und{v}(x)|^2\big]$ and $\Phi(v)= \max\big\{\frac{2|c_1|^2 }{c_0} , 4\big\} \big[1+ \sup\nolimits_{x \in (0,1)} |\und{v}(x)|^{ \max\{\nicefrac{2|c_1|^2 }{c_0} , 4\}}\big]$.
Next note that Lemma~\ref{coer:burgers} proves  for all $n \in \N$, $v, w \in \tilde{P}_n(H) = P_n(H)$  that
\begin{align}\label{burgers:cor:coer}
\langle v, \tilde{P}_n F( v + w ) \rangle_H = \left< v, F( v + w ) \right>_H \leq \phi(w) \| v \|^2_H+ \tfrac{3}{4} \| v \|^2_{ H_{ \nicefrac{1}{2} } }+ \Phi( w ).
\end{align}
Furthermore, Lemma~\ref{nonlin:burgers} ensures that for all $n \in \N$, $v, w \in \tilde{P}_n(H)$ it holds that $F \in \mathcal{C}(H_{\nicefrac{1}{8}}, H_{-\nicefrac{1}{2}})$ and
\begin{align}\label{burgers:cor:con}
\begin{split}
&\|F(v) - F(w) \|_{H_{-\nicefrac{1}{2}}} \\
& \leq |c_1|  |c_0|^{-\nicefrac{1}{2}} \left[\sup_{u \in  H_{ \nicefrac{1}{8} } \setminus \{0\}} \frac{\|u\|^2_{L^4(\lambda_{(0,1)}; \R)}}{\|u\|^2_{H_{\nicefrac{1}{8}}}} \right] \big( 1+ \left\| v \right\|_{ H_{\nicefrac{1}{8} }} +  \left\| w \right\|_{ H_{\nicefrac{1}{8} }} \big) \left\| v - w \right\|_{ H_{\nicefrac{1}{8} }} < \infty.
\end{split}
\end{align}
Combining  \eqref{burgers:cor:coer}--\eqref{burgers:cor:con}  and  Proposition~\ref{abs:prop:last} (with
$T=T$,
$c_0 = c_0$,	
$\gamma=\max\{\nicefrac{2|c_1|^2 }{c_0} , 4\}$,	
$ \theta=  1 + |c_1|  |c_0|^{-\nicefrac{1}{2}}  [\sup_{u \in  H_{ \nicefrac{1}{8}}\backslash \{0\}}  \nicefrac{\|u\|^2_{L^4(\lambda_{(0,1)}; \R)}}{\|u\|^2_{H_{\nicefrac{1}{8}}}} ]$,
$\vartheta=1$,
$\alpha= \nicefrac{1}{2}$,
$\varphi= \nicefrac{3}{4}$,
$\rho= \nicefrac{1}{8}$,
$\varrho =\varrho$, 
$\chi = \chi$,
$\xi = \xi$,
 $F=F|_{H_{\varrho}}$,
 $(P_n)_{n \in \N} = (\tilde{P}_n)_{n \in \N}$,
  $(h_n)_{n \in \N} = (h_n)_{n \in \N}$,
$\phi=\phi$, $\Phi=\Phi$,  $(\Y^n)_{n \in \N} = (\Y^n)_{n \in \N}$,    
$ (\mathcal{O}^n)_{n \in \N} = (\mathcal{O}^n)_{n \in \N}$,  $X=X$ in the notation of Proposition~\ref{abs:prop:last}) completes the proof of Corollary~\ref{burgers:cor:last}. 
\end{proof}

The next result, Corollary~\ref{cor:burgers:short} below, is a direct consequence of Corollary~\ref{burgers:cor:last}.
It establishes strong convergence for the stochastic Burgers equation, also see Remark~\ref{rem:burgers} below.

\begin{cor}\label{cor:burgers:short}
Consider the notation in Subsection~\ref{sec:notation},
let  $ T \in (0,\infty)$, 
$\varrho \in (\nicefrac{1}{8}, \nicefrac{1}{4})$,
$ \chi \in  (0, \nicefrac{\varrho }{2 } - \nicefrac{1}{16}] $,
$( H, \left< \cdot , \cdot \right>_H, \left\| \cdot \right\|_H ) = (L^2(\lambda_{(0,1)}; \R), \langle \cdot , \cdot \rangle_{L^2(\lambda_{(0,1)}; \R)}, \left\| \cdot \right\|_{L^2(\lambda_{(0,1)}; \R)} )$,
let $ A \colon D(A) \subseteq H \to H $ be the Laplace operator with Dirichlet boundary conditions on $H$,
let $ ( H_r, \left< \cdot , \cdot \right>_{ H_r }, \left\| \cdot \right\|_{ H_r } ) $, $ r \in \R $, be a family of interpolation spaces associated to $ -A $,
let
$\xi \in H_{\nicefrac{1}{2}} $,
let
$F \colon H_{\nicefrac{1}{8}} \to H_{-\nicefrac{1}{2}} $,
$(e_n)_{n \in \N} \colon \N \to H $,
$(P_n)_{n \in \N} \colon \N \to L(H) $,
and
$(h_n)_{n \in \N} \colon \N \to (0, T] $
be functions which satisfy
for all $v \in H_{\nicefrac{1}{8}}$, $n \in \N$ that
$F(v)= -\frac{1}{2}(v^2)'$,
$ e_n = [ (\sqrt{2} \sin(n \pi x) )_{x \in (0,1)}]_{\lambda_{(0,1)} , \mathcal{B}(\R)}$,
$ P_n(v) = \sum_{k = 1 }^n \langle e_k, v \rangle_H e_k $,
and
$ \limsup_{ m \to \infty} h_m =0$,
let $ ( \Omega, \F, \P ) $ be a probability space,
let $(W_t)_{t \in [0, T]}$ be an $\Id_H$-cylindrical $( \Omega, \F, \P )$-Wiener process,
let $ \Y^n, \mathcal{O}^n \colon [0, T] \times \Omega \to P_n(H)$, $ n \in \N$, be stochastic processes, let $ X\colon [0, T] \times \Omega \to H_{\varrho}$  be a  stochastic process with continuous sample paths, and assume  for all  $ n \in \N $, $t \in [0, T]$  that   $ [X_t ]_{\P, \mathcal{B}(H)} =    [ e^{ t A }   \xi  + \smallint_0^t e^{  ( t - s ) A}  \, F (  X_s ) \, ds ]_{\P, \mathcal{B}(H)} + \int_0^t e^{(t-s)A} \, dW_s$, $\left[\mathcal{O}_t^n \right]_{\P, \mathcal{B}(H)} = \int_0^t P_n \, e^{(t-s)A}  \, dW_s$,
and
\begin{equation}\label{eq:cor:burgers}
\P \Big( \Y_t^n = P_n \, e^{ t A } \xi + \smallint_0^t P_n \,  e^{  ( t - s ) A } \, \one_{ \{ \| \Y_{ \lf s \rf_{h_n} }^n \|_{ H_{\varrho} } + \| \mathcal{O}_{ \lf s \rf_{h_n} }^n +P_n \, e^{ \lf s \rf_{ h_n } A } \xi \|_{ H_{\varrho} } \leq | h_n|^{ - \chi } \}} \, F \big(  \Y_{ \lf s \rf_{ h_n } }^n \big) \, ds + \mathcal{O}_t^n \Big)=1. 
\end{equation} Then it holds for all $p \in (0, \infty)$ that
\begin{align}
\limsup_{n \to \infty} \sup_{t \in [0,T]} \E \big[ \| X_t -\Y_t^n \|_H^p \big] = 0.
\end{align}
\end{cor}

\begin{remark}\label{rem:burgers}
Consider the setting in Corollary~\ref{cor:burgers:short}.
Roughly speaking, Corollary~\ref{cor:burgers:short} demonstrates that
the full-discrete explicit numerical approximation scheme described by~\eqref{eq:cor:burgers}
converges strongly to a mild solution process of the stochastic Burgers equation 
\begin{align}
\tfrac{\partial}{\partial t} X_t(x) = \tfrac{\partial^2}{\partial x^2} X_t(x) - X_t(x) \cdot \tfrac{\partial}{\partial x} X_t(x) + \tfrac{\partial}{\partial t} W_t(x)
\end{align}
with  $X_0(x) = \xi(x)$ and $X_t(0)= X_t(1)=0$
for $t \in [0,T]$, $x \in (0,1)$
(cf., e.g.,
Da~Prato, Debussche, \& Temam~\cite[Section~1]{DaPrato1994}).
\end{remark}

\subsection{Stochastic Allen-Cahn equations}
\label{sec:cahn}

\subsubsection{Setting}
\label{setting:cahn}
Consider the notation in Subsection~\ref{sec:notation},
let $c_1 \in \R$, $ c_0, c_2, T \in (0,\infty)$, 
$\varrho \in (\nicefrac{1}{6}, \nicefrac{1}{4})$,
$ \chi \in  (0, \nicefrac{\varrho}{3}-\nicefrac{1}{18}] $,
$( H, \left< \cdot , \cdot \right>_H, \left\| \cdot \right\|_H )
= (L^2(\lambda_{(0,1)}; \R), \langle \cdot , \cdot \rangle_{L^2(\lambda_{(0,1)}; \R)}, \left\| \cdot \right\|_{L^2(\lambda_{(0,1)}; \R)} )$,
let $(e_n)_{n \in \N} \colon \N \to H $
be the function which satisfies
for all $ n \in \N $ that
$e_n = [ (\sqrt{2} \sin(n \pi x) )_{x \in (0,1)}]_{\lambda_{(0,1)} , \mathcal{B}(\R)}$,
let $ A \colon D(A) \subseteq H \to H $
be the linear operator which satisfies
$ D(A) = \{ v \in H \colon $ $ \sum_{k = 1}^\infty k^4 | \langle e_k , v \rangle_H |^2 < \infty \} $
and
$ \forall \, v \in D(A) \colon A v = -c_0 \pi^2 \sum_{k = 1}^\infty k^2 \langle e_k , v \rangle_H e_k$,
let $ ( H_r, \left< \cdot , \cdot \right>_{ H_r }, \left\| \cdot \right\|_{ H_r } ) $, $ r \in \R $,
be a family of interpolation spaces associated to $ - A $,
let $\xi \in H_{\nicefrac{1}{2}} $,
let
$F \colon H_{\nicefrac{1}{6}} \to H $,
$(P_n)_{n \in \N} \colon \N \to L(H) $,
and
$(h_n)_{n \in \N} \colon \N \to (0, T] $
be functions which satisfy
for all $v \in H_{\nicefrac{1}{6}}$, $n \in \N$ that
$F(v)= c_1 v - c_2 v^3$,
$ P_n(v) = \sum_{k = 1 }^n \langle e_k, v \rangle_H e_k $,
and $ \limsup_{ m \to \infty} h_m =0$,
let $ ( \Omega, \F, \P ) $ be a probability space,
let $(W_t)_{t \in [0, T]}$ be an $\Id_H$-cylindrical $( \Omega, \F, \P )$-Wiener process,
and let $ \Y^n, \mathcal{O}^n \colon [0, T] \times \Omega \to P_n(H)$, $ n \in \N$,
be stochastic processes which satisfy
for all $ n \in \N $, $ t \in [0,T] $ that
$\left[\mathcal{O}_t^n \right]_{\P, \mathcal{B}(H)} = \int_0^t P_n \, e^{(t-s)A}  \, dW_s$
and 
\begin{equation}\label{eq:set:cahn}
\P \Big( \Y_t^n = P_n \, e^{ t A } \xi + \smallint_0^t P_n \,  e^{  ( t - s ) A } \, \one_{ \{ \| \Y_{ \lf s \rf_{h_n} }^n \|_{ H_{\varrho} } + \| \mathcal{O}_{ \lf s \rf_{h_n} }^n +P_n \, e^{ \lf s \rf_{ h_n } A } \xi \|_{ H_{\varrho} } \leq | h_n|^{ - \chi } \}} \, F \big(  \Y_{ \lf s \rf_{ h_n } }^n \big) \, ds + \mathcal{O}_t^n \Big)=1. 
\end{equation}

\subsubsection{Properties of the nonlinearity}

In the results in this subsection, Lemmas~\ref{coer:cahn}--\ref{nonlin:cahn} below,
we collect a few elementary properties of the function $F$ in Subsection~\ref{setting:cahn} above.

\begin{lemma}
\label{coer:cahn}
Assume the setting in Subsection~\ref{setting:cahn} and let $r \in (\nicefrac{1}{4}, \infty)$, $v, w \in H_r$. Then
\begin{align}
\begin{split}
\left< v, F( v + w ) \right>_H &\leq \max\big\{6,  c_1+ |c_1|^2+|c_2|^2\big\}\big[\|v\|_H^2  + 1+ \sup\nolimits_{x \in (0,1)} |\und{w}(x)|^{\max\{6,  c_1+ |c_1|^2+|c_2|^2\}}\big].
\end{split}
\end{align}
\begin{proof}[Proof of Lemma~\ref{coer:cahn}]
Note that the fact that $\forall \, a,b \in \R \colon a^2 + 3ab+3b^2 \geq 0$ and the Cauchy-Schwartz inequality ensure that
\begin{align}
\begin{split}
&\left< v, F( v + w ) \right>_H = \left< v, c_1( v + w ) - c_2( v + w )^3 \right>_H = c_1 \|v\|_H^2 + c_1 \left<v,w \right>_H -c_2 \left< v, ( v + w )^3 \right>_H\\
&=c_1 \|v\|_H^2 + c_1 \left<v,w \right>_H -c_2 \int_0^1 [\und{v}(x)]^2 \big(  [\und{v}(x)]^2 + 3 \, \und{v}(x) \und{w}(x) + 3 \, [\und{w}(x)]^2\big) + \und{v}(x)  [\und{w}(x)]^3 \, dx\\
& \leq c_1 \|v\|_H^2 + c_1 \left<v,w \right>_H -c_2 \int_0^1 \und{v}(x)  [\und{w}(x)]^3 \, dx\\
& \leq c_1 \|v\|_H^2 + |c_1| \|v\|_H \|w\|_H + c_2 \|v\|_H \|w^3\|_H.
\end{split}
\end{align}
The fact that $\forall \,  a, b \in \R \colon ab \leq  a^2 + b^2$ and the fact that $\forall \, a \in \R, \, b \in [1, \infty) \colon a \leq 1+|a|^b$ hence prove that
\begin{align}
\begin{split}
\left< v, F( v + w ) \right>_H & \leq  c_1 \|v\|_H^2 + |c_1|^2\|v\|_H^2 + \|w\|_H^2 + |c_2|^2\|v\|_H^2 + \|w^3\|_H^2 \\
& \leq (c_1+ |c_1|^2+|c_2|^2) \|v\|_H^2  + \big[\!\sup\nolimits_{x \in (0,1)} |\und{w}(x)|^2\big]+ \big[\!\sup\nolimits_{x \in (0,1)} |\und{w}(x)|^6\big]\\
& \leq (c_1+ |c_1|^2+|c_2|^2) \|v\|_H^2  + 2 \, \big( 1+ \sup\nolimits_{x \in (0,1)} |\und{w}(x)|^{\max\{6,  c_1+ |c_1|^2+|c_2|^2\}} \big)\\
& \leq \max\big\{6,  c_1+ |c_1|^2+|c_2|^2\big\}\big[\|v\|_H^2  + 1+ \sup\nolimits_{x \in (0,1)} |\und{w}(x)|^{\max\{6,  c_1+ |c_1|^2+|c_2|^2\}}\big].
\end{split}
\end{align}
The proof of Lemma~\ref{coer:cahn} is thus completed.
\end{proof}
\end{lemma}

\begin{lemma}\label{nonlin:cahn}
Assume the setting in Subsection~\ref{setting:cahn} and let $v, w \in H_{\nicefrac{1}{6}}$. Then it holds that $F \in \mathcal{C}(H_{\nicefrac{1}{6}}, H)$ and
\begin{align}
\label{eq:cahn-cont}
\begin{split}
&\|F(v) - F(w) \|_{H} \\
& \leq  \left(\frac{|c_1|}{(c_0 \pi^2)^{\nicefrac{1}{6}}}+ 2 \, c_2 \!\left[\sup_{u \in  H_{ \nicefrac{1}{6}} \setminus \{0\}} \frac{\|u\|^3_{L^6(\lambda_{(0,1)}; \R)}}{\|u\|^3_{H_{\nicefrac{1}{6}}}} \right] \right)\big( 1+ \left\| v \right\|_{ H_{\nicefrac{1}{6} }}^2 +  \left\| w \right\|_{ H_{\nicefrac{1}{6} }}^2 \big) \left\| v - w \right\|_{ H_{\nicefrac{1}{6}}} < \infty.
\end{split}
\end{align}	
\end{lemma}
\begin{proof}[Proof of Lemma~\ref{nonlin:cahn}]
First, observe that
\begin{align}\label{v-w:cahn}
\begin{split}
\|v-w \|_{H} &\leq \|(-A)^{\nicefrac{-1}{6}} \|_{L(H)} \|v-w \|_{H_{\nicefrac{1}{6}}} = (c_0 \pi^2)^{\nicefrac{-1}{6}} \|v-w \|_{H_{\nicefrac{1}{6}}} \\
& \leq (c_0 \pi^2)^{\nicefrac{-1}{6}} \big( 1+ \left\| v \right\|_{ H_{\nicefrac{1}{6} }}^2 +  \left\| w \right\|_{ H_{\nicefrac{1}{6} }}^2 \big) \|v-w \|_{H_{\nicefrac{1}{6}}}.
\end{split}
\end{align}
In addition, note that H\"older's inequality implies that
\begin{align}
\begin{split}
&\|v^3 -w^3\|_H = \|(v-w) \cdot (v^2 + v \cdot w + w^2)\|_{L^2(\lambda_{(0,1)}; \R)}\\
&\leq \|v-w\|_{L^6(\lambda_{(0,1)}; \R)} \|v^2+v \cdot w+w^2\|_{L^3(\lambda_{(0,1)}; \R)}\\
& \leq \|v-w\|_{L^6(\lambda_{(0,1)}; \R)} \big(\|v\|_{L^6(\lambda_{(0,1)}; \R)}^2 + \|v\|_{L^6(\lambda_{(0,1)}; \R)} \|w\|_{L^6(\lambda_{(0,1)}; \R)} + \|w\|_{L^6(\lambda_{(0,1)}; \R)}^2\big)\\
& \leq 2\, \|v-w\|_{L^6(\lambda_{(0,1)}; \R)} \big(\|v\|_{L^6(\lambda_{(0,1)}; \R)}^2  + \|w\|_{L^6(\lambda_{(0,1)}; \R)}^2\big)\\
& \leq 2 \left[\sup_{u \in  H_{\nicefrac{1}{6} } \setminus \{0\}} \frac{\|u\|^3_{L^6(\lambda_{(0,1)}; \R)}}{\|u\|^3_{H_{\nicefrac{1}{6}}}} \right] \big( 1+ \left\| v \right\|_{ H_{\nicefrac{1}{6} }}^2 +  \left\| w \right\|_{ H_{\nicefrac{1}{6} }}^2 \big) \left\| v - w \right\|_{ H_{\nicefrac{1}{6} }}.
\end{split}
\end{align}
Combining this with \eqref{v-w:cahn} shows that
\begin{align}
\label{nonlin:cahn:last}
\begin{split}
& \|F(v) - F(w) \|_{H}  = \|c_1(v-w)-c_2(v^3-w^3) \|_{H} \leq |c_1| \|v-w\|_H+ c_2 \|v^3-w^3\|_H \\
& \leq  \left(\frac{|c_1|}{(c_0 \pi^2)^{\nicefrac{1}{6}}}+ 2 \, c_2 \!\left[\sup_{u \in  H_{ \nicefrac{1}{6}} \setminus \{0\}} \frac{\|u\|^3_{L^6(\lambda_{(0,1)}; \R)}}{\|u\|^3_{H_{\nicefrac{1}{6}}}} \right] \right)\big( 1+ \left\| v \right\|_{ H_{\nicefrac{1}{6} }}^2 +  \left\| w \right\|_{ H_{\nicefrac{1}{6} }}^2 \big) \left\| v - w \right\|_{ H_{\nicefrac{1}{6}}} .
\end{split}
\end{align}
Furthermore, observe that the Sobolev embedding theorem proves that
\begin{align}
\sup_{u \in  H_{ \nicefrac{1}{6} }\setminus \{0\}} \frac{\|u\|^3_{L^6(\lambda_{(0,1)}; \R)}}{\|u\|^3_{H_{\nicefrac{1}{6}}}}  < \infty.
\end{align}
This and \eqref{nonlin:cahn:last} establish \eqref{eq:cahn-cont}.
The proof of Lemma~\ref{nonlin:cahn} is thus completed.
\end{proof}

\subsubsection{Strong convergence}

\begin{cor}
\label{cahn:cor:last}
Assume the setting in Subsection~\ref{setting:cahn} and let $ X\colon [0, T] \times \Omega \to H_{\varrho}$  be a  stochastic process with continuous sample paths which satisfies for all $t \in [0, T]$  that  $ [X_t ]_{\P, \mathcal{B}(H)} =    [ e^{ t A }   \xi  + \smallint_0^t e^{  ( t - s ) A}  \, F (  X_s ) \, ds ]_{\P, \mathcal{B}(H)} + \int_0^t e^{(t-s)A} \, dW_s$. Then it holds for all $p \in (0, \infty)$ that
\begin{align}
\limsup_{n \to \infty} \sup_{t \in [0,T]} \E \big[ \| X_t -\Y_t^n \|_H^p \big] = 0.
\end{align}
\end{cor}
\begin{proof}[Proof of Corollary~\ref{cahn:cor:last}]
Throughout this proof
let $(\tilde{P}_n)_{n \in \N} \colon \N \to L(H_{-1}) $
be the function which satisfies
for all $ n \in \N $, $ v \in H $ that
$ \tilde{P}_n(v) = P_n(v) $
and let $ \phi , \Phi \colon H_1 \to [0,\infty) $ be the functions which satisfy for all $ v \in H_{1}$ that  $\phi(v) =  \max\{6,  c_1+ |c_1|^2+|c_2|^2\} \big[1+\sup\nolimits_{x \in (0,1)} |\und{v}(x)|^2\big]$ and $\Phi(v)=  \max\{6,  c_1+ |c_1|^2+|c_2|^2\} \big[1+ \sup\nolimits_{x \in (0,1)} |\und{v}(x)|^{  \max\{6,  c_1+ |c_1|^2+|c_2|^2\}}\big]$.
Observe that Lemma~\ref{coer:cahn} ensures for all $n \in \N$, $v, w \in \tilde{P}_n(H) = P_n(H)$  that
\begin{align}\label{cahn:cor:coer}
\langle v, \tilde{P}_n F( v + w ) \rangle_H = \left< v, F( v + w ) \right>_H \leq \phi(w) \| v \|^2_H+\Phi( w ).
\end{align}
In addition, Lemma~\ref{nonlin:cahn} shows that  for all $n \in \N$, $v, w \in \tilde{P}_n(H)$ it holds  that $F \in \mathcal{C}(H_{\nicefrac{1}{6}}, H)$ and
\begin{align}\label{cahn:cor:con}
\begin{split}
&\|F(v) - F(w) \|_H \\
& \leq  \left(\frac{|c_1|}{(c_0 \pi^2)^{\nicefrac{1}{6}}}+ 2 \, c_2 \!\left[\sup_{u \in  H_{ \nicefrac{1}{6}} \setminus \{0\}} \frac{\|u\|^3_{L^6(\lambda_{(0,1)}; \R)}}{\|u\|^3_{H_{\nicefrac{1}{6}}}} \right] \right)\big( 1+ \left\| v \right\|_{ H_{\nicefrac{1}{6} }}^2 +  \left\| w \right\|_{ H_{\nicefrac{1}{6} }}^2 \big) \left\| v - w \right\|_{ H_{\nicefrac{1}{6}}} < \infty.
\end{split}
\end{align}
Combining  \eqref{cahn:cor:coer}--\eqref{cahn:cor:con}  with  Proposition~\ref{abs:prop:last}
(with
$T=T$,
$c_0 = c_0$,	
$\gamma=\max\{6,  c_1+ |c_1|^2+|c_2|^2\}$,	
$ \theta= \nicefrac{|c_1|}{(c_0 \pi^2)^{\nicefrac{1}{6}}}+  2 \, c_2  [\sup_{u \in  H_{ \nicefrac{1}{6} }\backslash \{0\}}  \nicefrac{\|u\|^3_{L^6(\lambda_{(0,1)}; \R)}}{\|u\|^3_{H_{\nicefrac{1}{6}}}} ]$,
$\vartheta=2$,
$\alpha= 0$,
$\varphi= 0$,
$\rho= \nicefrac{1}{6}$,
$\varrho =\varrho$, 
$\chi = \chi$,
$\xi = \xi$,
$F=F|_{H_{\varrho}}$,
$(P_n)_{n \in \N} = (\tilde{P}_n)_{n \in \N}$,
$(h_n)_{n \in \N} = (h_n)_{n \in \N}$,
$\phi=\phi$, $\Phi=\Phi$,  $(\Y^n)_{n \in \N} = (\Y^n)_{n \in \N}$,    
$ (\mathcal{O}^n)_{n \in \N} = (\mathcal{O}^n)_{n \in \N}$,  $X=X$ in the notation of Proposition~\ref{abs:prop:last})
completes the proof of Corollary~\ref{cahn:cor:last}. 
\end{proof}

The next result, Corollary~\ref{cor:cahn:short} below,
proves strong convergence for the stochastic Allen-Cahn equation, also see Remark~\ref{rem:allen-cahn} below.
Corollary~\ref{cor:cahn:short} follows immediately from Corollary~\ref{cahn:cor:last} above.

\begin{cor}\label{cor:cahn:short}
Consider the notation in Subsection~\ref{sec:notation},
let  $ T \in (0,\infty)$, 
$\varrho \in (\nicefrac{1}{6}, \nicefrac{1}{4})$,
$ \chi \in  (0, \nicefrac{\varrho }{3 } - \nicefrac{1}{18}] $,
$( H, \left< \cdot , \cdot \right>_H, \left\| \cdot \right\|_H ) = (L^2(\lambda_{(0,1)}; \R), \langle \cdot , \cdot \rangle_{L^2(\lambda_{(0,1)}; \R)}, \left\| \cdot \right\|_{L^2(\lambda_{(0,1)}; \R)} )$,
let $ A \colon D(A) \subseteq H \to H $ be the Laplace operator with Dirichlet boundary conditions on $H$,
let $ ( H_r, \left< \cdot , \cdot \right>_{ H_r }, \left\| \cdot \right\|_{ H_r } ) $, $ r \in \R $, be a family of interpolation spaces associated to $ -A $,
let
$\xi \in H_{\nicefrac{1}{2}} $,
let
$F \colon H_{\nicefrac{1}{6}} \to H$,
$(e_n)_{n \in \N} \colon \N \to H $,
$(P_n)_{n \in \N} \colon \N \to L(H) $,
and
$(h_n)_{n \in \N} \colon \N \to (0, T] $
be functions which satisfy
for all $v \in H_{\nicefrac{1}{6}}$, $n \in \N$ that
$F(v)= v- v^3$,
$ e_n = [ (\sqrt{2} \sin(n \pi x) )_{x \in (0,1)}]_{\lambda_{(0,1)} , \mathcal{B}(\R)}$,
$ P_n(v) = \sum_{k = 1 }^n \langle e_k, v \rangle_H e_k $,
and
$ \limsup_{ m \to \infty} h_m =0$,
let $ ( \Omega, \F, \P ) $ be a probability space,
let $(W_t)_{t \in [0, T]}$ be an $\Id_H$-cylindrical $( \Omega, \F, \P )$-Wiener process,
let $ \Y^n, \mathcal{O}^n \colon [0, T] \times \Omega \to P_n(H)$, $ n \in \N$, be stochastic processes, let $ X\colon [0, T] \times \Omega \to H_{\varrho}$  be a  stochastic process with continuous sample paths, and assume  for all  $ n \in \N $, $t \in [0, T]$ that   $ [X_t ]_{\P, \mathcal{B}(H)} =    [ e^{ t A }   \xi  + \smallint_0^t e^{  ( t - s ) A}  \, F (  X_s ) \, ds ]_{\P, \mathcal{B}(H)} + \int_0^t e^{(t-s)A} \, dW_s$, $\left[\mathcal{O}_t^n \right]_{\P, \mathcal{B}(H)} = \int_0^t P_n \, e^{(t-s)A}  \, dW_s$,
and
\begin{equation}\label{eq:short:cahn}
	\P \Big( \Y_t^n = P_n \, e^{ t A } \xi + \smallint_0^t P_n \,  e^{  ( t - s ) A } \, \one_{ \{ \| \Y_{ \lf s \rf_{h_n} }^n \|_{ H_{\varrho} } + \| \mathcal{O}_{ \lf s \rf_{h_n} }^n +P_n \, e^{ \lf s \rf_{ h_n } A } \xi \|_{ H_{\varrho} } \leq | h_n|^{ - \chi } \}} \, F \big(  \Y_{ \lf s \rf_{ h_n } }^n \big) \, ds + \mathcal{O}_t^n \Big)=1. 
\end{equation} Then it holds for all $p \in (0, \infty)$ that
\begin{align}
	\limsup_{n \to \infty} \sup_{t \in [0,T]} \E \big[ \| X_t -\Y_t^n \|_H^p \big] = 0.
\end{align}
\end{cor}

\begin{remark}\label{rem:allen-cahn}
Consider the setting in Corollary~\ref{cor:cahn:short}.
Roughly speaking, Corollary~\ref{cor:cahn:short}
reveals that the full-discrete explicit numerical approximation scheme described by \eqref{eq:short:cahn}
converges strongly to a mild solution process of the stochastic Allen-Cahn equation 
\begin{align}
\tfrac{\partial}{\partial t} X_t(x) = \tfrac{\partial^2}{\partial x^2} X_t(x) + X_t(x) - [X_t(x)]^3 + \tfrac{\partial}{\partial t} W_t(x)
\end{align}
with $X_0(x) = \xi(x)$ and $X_t(0)= X_t(1) = 0$ 
for $x \in (0,1)$,  $t \in [0,T]$
(cf., e.g., Kov\'acs, Larsson, \& Lindgren~\cite[Section~1]{Kovacs2015}).
\end{remark}

\section*{Acknowledgements}

Mario Hefter is gratefully acknowledged for bringing a few typos into our notice.
This project has been partially supported through the ETH Research Grant \mbox{ETH-47 15-2}
``Mild stochastic calculus and numerical approximations for nonlinear stochastic evolution equations with L\'evy noise''.

\bibliographystyle{acm}
\bibliography{bibfile}

\end{document}